\documentclass[twoside,11pt]{article}

\usepackage{blindtext}

\usepackage[preprint]{format}

\RequirePackage{amsmath,amsfonts} 
\usepackage{caption}
\usepackage{subcaption}
\usepackage{float}

\usepackage{xcolor}
\usepackage{tabularx}
\usepackage{adjustbox}

\usepackage{comment}

\DeclareMathOperator{\E}{\mathbb{E}}
\DeclareMathOperator{\Y}{\mathcal{Y}}

\DeclareMathOperator{\X}{\mathcal{X}}
\DeclareMathOperator{\W}{\mathcal{W}}
\DeclareMathOperator{\Z}{\mathcal{Z}}

\newcommand{\reals}{\mathbb{R}}

\usepackage{bm}

\ShortHeadings{Understanding Approximate Methods in BSL}{Wang and Gorodetsky}
\firstpageno{1}

\begin{document}

\title{A Global Lipschitz Stability Perspective for Understanding Approximate Approaches in Bayesian Sequential Learning}

\author{\name Liliang Wang \email liliangw@umich.edu \\
       \addr Department of Aerospace Engineering\\
       University of Michigan\\
       Ann Arbor, MI 48109, USA
       \AND
       \name Alex Gorodetsky \email goroda@umich.edu \\
       \addr Department of Aerospace Engineering\\
       University of Michigan\\
       Ann Arbor, MI 48109, USA}

\editor{}

\maketitle

\begin{abstract}

We establish a general, non-asymptotic error analysis framework for understanding the effects of incremental approximations made by practical approaches for Bayesian sequential learning (BSL) on their long-term inference performance. Our setting covers inverse problems, state estimation, and parameter-state estimation. In these settings, we bound the difference---termed the learning error---between the unknown true posterior and the approximate posterior computed by these approaches, using three widely used distribution metrics: total variation, Hellinger, and Wasserstein distances. This framework builds on our establishment of the global Lipschitz stability of the posterior with respect to the prior across these settings. To the best of our knowledge, this is the first work to establish such global Lipschitz stability under the Hellinger and Wasserstein distances and the first general error analysis framework for approximate BSL methods.

Our framework offers two sets of upper bounds on the learning error. The first set demonstrates the stability of general approximate BSL methods with respect to the incremental approximation process, while the second set is estimable in many practical scenarios. 

Furthermore, as an initial step toward understanding the phenomenon of learning error decay, which is sometimes observed, we identify sufficient conditions under which data assimilation leads to learning error reduction. 

\end{abstract}


\section{Introduction}
\label{sec:intro}

Approximations are at the heart of Bayesian learning approaches addressing real-world problems: in most practical scenarios, exact Bayesian solutions are intractable and approximations are made. In this paper, we study approximate methods for Bayesian sequential learning (BSL), in the context of inverse problems (IP), state estimation (SE), and parameter-state estimation (PS). When data arrive and are assimilated incrementally in these settings, due to either large-scale data or a dynamical system, BSL seeks to obtain a sequence of posteriors $P_k, k = 1,2,\cdots$, by recursively updating an initial prior distribution $P_0$. This process is illustrated in Figure~\ref{fig:sequential_bayes_learning}. At each step $k$, BSL updates the prior $P_{k-1}$, which is the posterior of the previous step, using the newly received data $y_k$, to obtain the posterior $P_k$. This update can be expressed as $P_k=F_k(P_{k-1})$. Here, $F_k$ denotes the prior-to-posterior map, which varies across the three problem contexts. In IP, SE, and PS, the posterior $P_k$ at step $k$ represents, respectively, the distribution of some unknown parameter $X$ given all available data $\mathcal{Y}_{1:k} = (y_1, \ldots, y_k)$, the distribution of a time-evolving state $X_k$ conditioned on $\mathcal{Y}_{1:k}$, and the joint distribution of $X_k$ and system parameters $W$ given $\mathcal{Y}_{1:k}$. 

\begin{figure}
     \centering
     \includegraphics[width=0.8\linewidth]{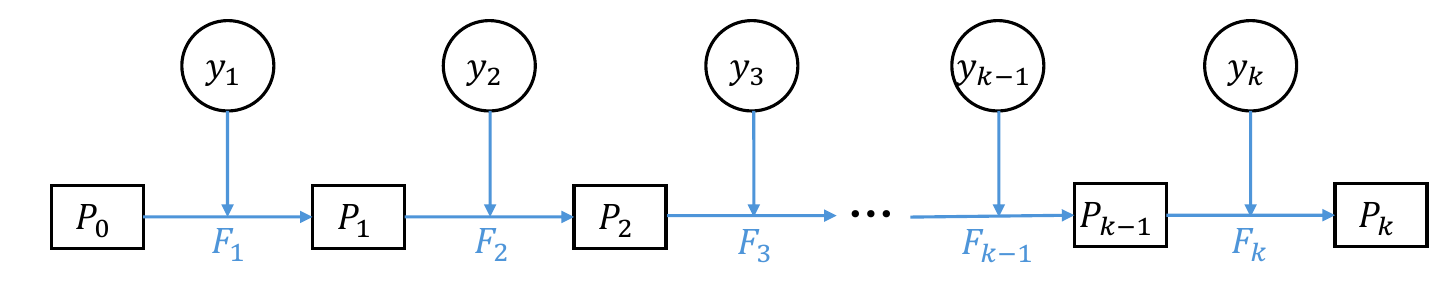}
     \caption{Process of BSL. At each step $k$, BSL updates the prior $P_{k-1}$ using the data $y_k$ to obtain the posterior $P_k$. The symbol $F_k$ denotes the map from a prior to its corresponding posterior at step $k$.}
    \label{fig:sequential_bayes_learning}
\end{figure}

In practice, the exact posterior $P_k$ is typically intractable to compute. Therefore, practical BSL approaches generate approximate solutions. More specifically, at each step $k$, they compute an approximate posterior $Q_k$, by approximately assimilating the new data $y_k$, using the previous approximate posterior $Q_{k-1}$ as the prior. In other words, $Q_k$ is the result of approximately applying $F_k$ onto $Q_{k-1}$. We denote this approximate map by $\hat{F}_k$, i.e., $Q_k=\hat{F}_k(Q_{k-1})$. This procedure is shown in Figure~\ref{fig:approx_sequential_learning}. Let $d$ denote a metric for probability measures, such as the total variation (TV) or Hellinger distance. We define the distance between the exact solution $P_k$ and the approximate solution $Q_k$, denoted by $d(P_k, Q_k)$, as the {\it learning error}. As a result, the prior error at step $k$ is precisely the learning error at previous step $k-1$. 
\begin{figure}
     \centering
     \includegraphics[width=0.8\linewidth]{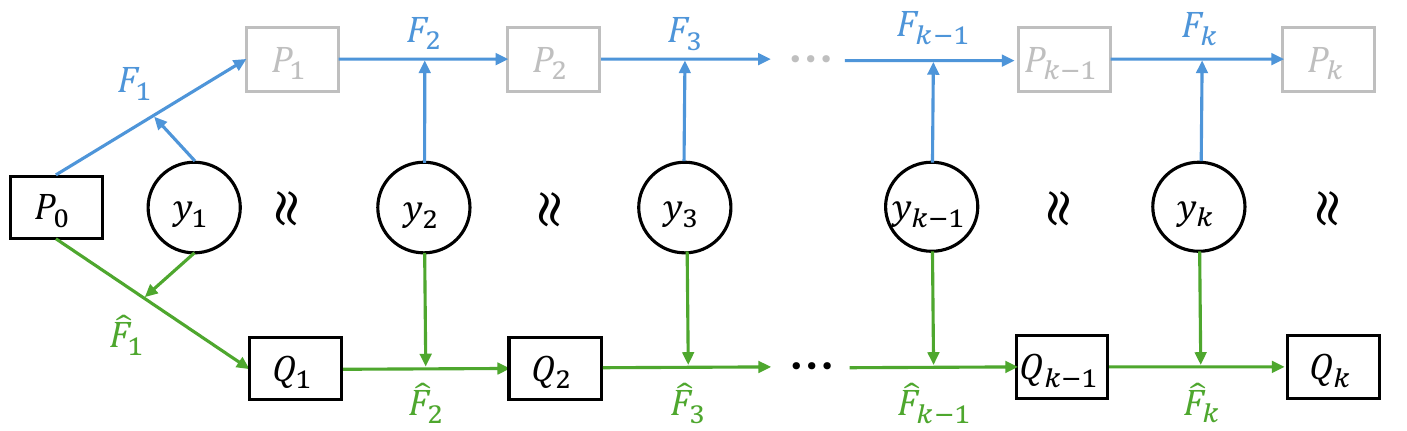}
     \caption{Procedure of approximate approaches for BSL. The exact posterior $P_k$ is typically intractable. The approximate approaches for BSL instead compute an approximate posterior $Q_k$ at each step $k$, by using the previous approximation $Q_{k-1}$ as the prior and approximately assimilating the new data $y_k$. Here, $Q_k$ is obtained from $Q_{k-1}$ via $Q_k=\hat{F}_k(Q_{k-1})$. The distance between $P_k$ and $Q_k$, $d(P_k,Q_k)$, is referred to as the learning error.}
    \label{fig:approx_sequential_learning}
\end{figure}

It is important to note that $Q_k$ is not the exact posterior obtained by using $Q_{k-1}$ as the prior and assimilating the data $y_k$ at step $k$. Let $Q_k^*$ denote this exact posterior, i.e., $Q_k^*=F_k(Q_{k-1})$. As $Q_k^*$ is normally intractable to obtain, approximate methods for BSL instead approximate it by $Q_k$. This procedure is illustrated in Figure~\ref{fig:break_down_approx_sequential_learning}. The process of approximating $Q_k^*$ with $Q_k$ is referred to as the {\it incremental approximation process} at step $k$. We call the distance between $Q_k$ and $Q_k^*$, $d(Q_k,Q_k^*)$, the {\it incremental approximation error} of step $k$.
\begin{figure}
     \centering
     \includegraphics[width=0.8\linewidth]{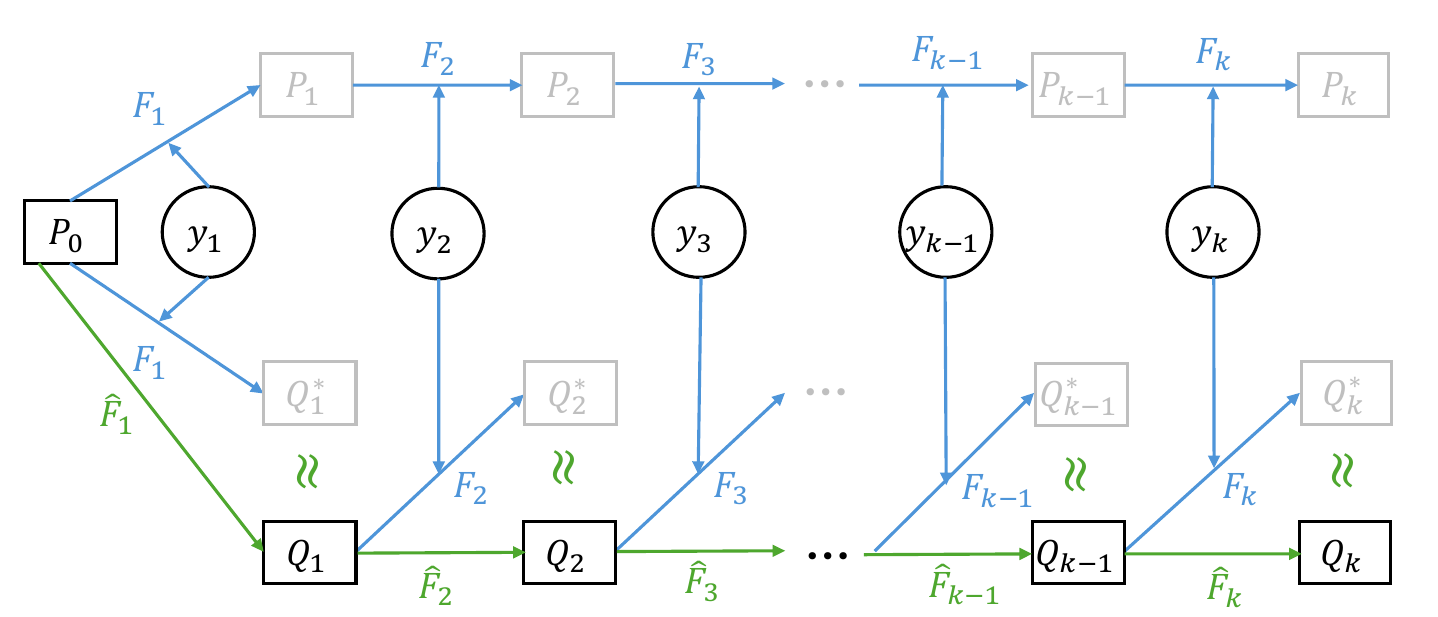}
     \caption{Break-down procedure of approximate approaches in BSL. The exact posterior and approximate posterior at step $k$ are denoted by $P_k$ and $Q_k$, respectively. When the previous approximate posterior, $Q_{k-1}$, is used as the prior, the corresponding exact posterior is denoted by $Q_k^*=F_k(Q_{k-1})$. As $Q_k^*$ is normally infeasible to compute, approximate approaches for BSL construct $Q_k=\hat{F}_k(Q_{k-1})$ to approximate $Q_k^*$. This process is referred to as the incremental approximation process, and the distance $d(Q_k, Q_k^*)$ is called the incremental approximation error.}
    \label{fig:break_down_approx_sequential_learning}
\end{figure}

In recent years, a wide range of methods have been developed for approximately solving BSL problems. These include online/sequential variational inference (VI) methods (e.g., \cite{onlinevi_nonlinear,onlinevi_realtime, online_VI_linear, fbovi}) and optimal transport-based sequential learning approaches (e.g., \cite{filter_ot_data_driven, filter_amortize_ot, ot_dual_estimation}). Approximate BSL methods also encompass classical Bayesian filtering techniques such as the Gaussian filter (\cite{Gaussian_filter}) and sequential importance sampling (particle filter) (e.g., \cite{particle_filter,block_pf_jmlr}). The Gaussian filter category includes Kalman filter variants, such as the unscented Kalman filter (UKF) and ensemble Kalman filter (EnKF) (\cite{bayesian_filter}).

Despite the extensive development of approximate approaches for BSL, rigorous analysis of these approaches remains limited. Among the available studies, most focus on asymptotic properties of specific approaches (e.g., \cite{Stable_approx_optimal_filter, enkf_large_sample, survey_convergence_pf}). Non-asymptotic analysis, by contrast, has been developed to a lesser extent. Existing contributions, such as \cite{pf_adaptive}, \cite{tt_sequential}, \cite{local_pf_dim}, \cite{stats_accuracy}, and \cite{smoothing_error_jmlr}, typically provide results tailored to very specific methods. There is a gap in the literature: a general, method-agnostic framework for non-asymptotic analysis of approximate approaches for BSL is still lacking, and general properties of these methods are rarely studied. 

This work aims to bridge this gap by providing a non-asymptotic error analysis framework that applies broadly to various approximate approaches for BSL. This framework enables us to understand how the incremental approximation errors affect the long-term learning performance of approximate BSL methods. It can also serve as a foundational tool for further analysis and evaluation of a specific method when integrated with method-specific theoretical results. For instance, existing accuracy analysis for offline VI (e.g., \cite{approx_accuracy_vi, sparse_vi, convergence_rate_vi}) can be incorporated to assess the incremental approximation error $(Q_k, Q_k^*)$ when applying our framework to sequential VI methods.

Error analysis of general approximate approaches for BSL requires understanding how the prior error, $d(P_{k-1},Q_{k-1})$, propagates to the posterior error, $d(P_k,Q_k^*)$. In practical scenarios, the prior error can be substantial. As a result, existing local stability results for the prior-to-posterior map $F_k$, which require the prior error to be sufficiently small, often fail to capture the relationship between the prior and posterior errors in practical BSL settings. 

Moreover, since the exact previous posterior $P_{k-1}$ is typically intractable, our knowledge of it is highly limited. It is often impossible to identify a class of distributions which $P_{k-1}$ belongs to, let alone a class that contains both $P_{k-1}$ and $Q_{k-1}$. Additionally, because our analysis targets general approximate BSL methods, we do not restrict $Q_{k-1}$ to lie within any specific class. Consequently, the existing results of global robustness of posterior, which rely on assumptions that the priors belong to a certain class, are not applicable in this context. 

As elaborated in Section~\ref{sec:related_work}, most existing results on the stability of the posterior with respect to the prior are either limited to local stability or rely on restrictive assumptions that the prior belongs to a specific class. Therefore, to establish meaningful error bounds for general approximate approaches in BSL, it is essential to answer the following fundamental question:

\noindent \textbf{QUESTION 1 (global stability of posterior)}: Is the posterior globally stable with respect to the prior, irrespective of the prior class? 

We then use the answer to derive results addressing the following questions:

\noindent \textbf{QUESTION 2 (stability w.r.t. the incremental approximation process)}: Is the overall learning error $d(P_k,Q_k)$ of general approximate approaches for BSL guaranteed to be upper bounded, provided that the incurred incremental approximation errors $d(Q_j,Q_j^*), j \leq k$ are upper bounded, regardless of how large those bounds may be? Furthermore, given another sequence of approximate posteriors $\tilde{Q}_j, j \leq k$, is the distance between the two approximate posteriors $Q_k$ and $\tilde{Q}_k$, namely $d(Q_k,\tilde{Q}_k)$, also guaranteed to be bounded if their respective incremental approximation errors $d(Q_j,Q_j^*)$ and $d(\tilde{Q}_j,\tilde{Q}_j^*), j \leq k$ are bounded?

\noindent \textbf{QUESTION 3 (accuracy analysis)}: Does there exist a finite upper bound for the learning error $d(P_k,Q_k)$ which is independent of the true posteriors $P_j, j \leq k$, and determined solely by the approximate posteriors $Q_j, j \leq k$ and incremental approximation errors $d(Q_j,Q_j^*), j \leq k$?

In practical applications, it is often observed that the learning error $d(P_k,Q_k)$ decreases as the number of update steps $k$ increases. Existing analysis of this phenomenon typically focus on specific algorithms, such as bootstrap particle filter and mean-field EnKF. To gain theoretical insight into this behavior for general BSL approaches, we pose the following question:

\noindent \textbf{QUESTION 4 (learning error decay)}: Under what conditions does the learning error of \textit{general} approximate methods for BSL decay as more data are assimilated? 

\subsection{Contributions}
We answer these questions through three sets of results:
\begin{enumerate}
\item (\textbf{Global stability of Bayesian inference for arbitrary priors})
We establish global Lipschitz stability of the posterior with respect to the prior, i.e., the pointwise global Lipschitz continuity of the prior-to-posterior map $F_k$ under suitable assumptions in Theorem~\ref{thm:pointwise_lipschitz}. The formal definition of pointwise globally Lipschitz continuous functions can be found in Definition~\ref{def:pointwise_lipschitz}. The assumptions required by Theorem~\ref{thm:pointwise_lipschitz} under the TV and Hellinger distances are very mild and hold in most problems. To the best of our knowledge, this is the first work to prove pointwise \textit{global} Lipschitz continuity of the prior-to-posterior map under the Hellinger and Wasserstein distances. For the TV distance, for which global Lipschitz stability bounds already exist, our derived bound improves the state-of-the-art one by $50\%$ under same assumptions.

\item (\textbf{Stability and accuracy analyses}) Building on Theorem~\ref{thm:pointwise_lipschitz}, we establish two sets of upper bounds on the learning error $d(P_k,Q_k)$ in Theorem~\ref{thm:learning_error}. The first set is linear in the incremental approximation errors $d(Q_j, Q_j^*), j \leq k$, with the approximate posteriors $Q_j,j\leq k$ entering the bounds only through the incremental approximation errors. This result shows that, under mild assumptions, the overall learning error remains bounded as long as the incremental approximation errors are bounded. Moreover, under the same assumptions, the first set of bounds also implies that the distance between two approximate posteriors, $Q_k$ and $\tilde{Q}_k$, is bounded, provided that both sequences of incremental approximation errors, $d(Q_j,Q_j^*)$ and $d(\tilde{Q}_j,\tilde{Q}_j^*), j \leq k$, are bounded.
The second set of bounds depends only on the approximate posteriors and incremental approximation errors, and is independent of the exact posteriors $P_j, j\leq k$, making it computable in many practical applications. 

\item (\textbf{Learning error reduction by data assimilation}) In Theorems~\ref{thm:er:all:tv}---\ref{thm:er:se&ps:W}, we provide sufficient conditions on the system and the distributions $P_{k-1}$ and $Q_{k-1}$ under which $d(P_k,Q_k^*) \leq d(P_{k-1},Q_{k-1})$ holds. Combining these theorems with the triangle inequality $d(P_k,Q_k) \leq d(P_k,Q_k^*)+d(Q_k^*,Q_k)$ enables us to derive sufficient conditions on the system, posteriors and approximate posteriors that guarantee $d(P_k,Q_k) \leq d(P_{k-1},Q_{k-1})$—without assuming that the approximate posteriors belong to any specific class or satisfy certain structural properties. These results offer an initial step toward explaining the phenomenon of learning error decay in general BSL methods.

In addition to the above, the contributions of this work also include:
\item (\textbf{Analysis under the Wasserstein distance}) We provide an error analysis under the Wasserstein distance for problems involving dynamical systems, including both state estimation and parameter-state estimation. To the best of our knowledge, this work presents the first error analysis under the Wasserstein distance for BSL in parameter-state estimation. Additionally, there are very few existing studies on learning error in state estimation under the Wasserstein distance. Compared to other commonly used probability metrics, the Wasserstein distance has two key advantages. First, it captures not only the magnitude of differences between distributions but also how these differences are spatially distributed, making it well-suited for comparing distributions with different supports (\cite{wasserstein_fundamental,wasserstein_2,wasserstein_1}). Second, it enables meaningful comparisons between continuous and discrete distributions (\cite{wasserstein_2,wasserstein_fundamental}), which is especially useful for analyzing Monte Carlo-based BSL methods. 

\item (\textbf{Practical insights}) Our theoretical results offer guidance for choosing the initial prior and approximate posteriors in practical BSL applications. Specifically, under certain conditions, choosing a well-designed alternative of the true initial prior can enhance learning accuracy in state estimation problems. Additionally, under the same conditions, selecting an approximate posterior $Q_k$ with a larger immediate incremental approximation error $d(Q_k^*,Q_k)$ may be preferable if it facilitates more accurate incremental approximations in later stages of the learning process, relative to a choice that minimizes the immediate error but complicates the incremental approximation process over time.
\end{enumerate} 

\subsection{Related work}
This paper makes two main contributions: (i) establishing the global Lipschitz stability of the posterior with respect to the prior in Bayesian learning, and (ii) developing a general, non-asymptotic error analysis framework for approximate approaches to BSL. In this section, we review related literature on posterior stability in Bayesian learning and on non-asymptotic analyses of approximate BSL methods.
\subsubsection{Posteriors stability in Bayesian learning}
This section reviews the literature on posterior stability with respect to the prior in inverse and state estimation problems. 
\label{sec:related_work}
\paragraph{Posterior stability in inverse problems}
\label{sec:stability:ip}
In this work, we establish the global Lipschitz stability of the posterior w.r.t. the prior, as formalized in Theorem~\ref{thm:pointwise_lipschitz}. Specifically, for inverse problems, the theorem shows that for every prior $\mu \in \bar{\mathcal{P}}_k(\X)$, the distance between the posteriors $F_k(\mu)$ and $F_k(\mu^\prime)$ is bounded by
\begin{equation}
d(F_k(\mu), F_k(\mu^\prime)) \leq K(\mu)d(\mu,\mu^\prime), \quad \forall \mu^\prime \in \bar{\mathcal{P}}_k(\X),
\label{ieq:lit_review:global}
\end{equation}
where $\bar{\mathcal{P}}_k(\X)$ denotes the set of all prior probability measures of which the corresponding posteriors are well-defined. The formal definition of $\bar{\mathcal{P}}_k(\X)$ can be found in Definition~\ref{def:set_P_k}. Here $\X$ represents the space of the unknown variable $X$. Equation~\eqref{ieq:lit_review:global} establishes the pointwise global Lipschitz continuity of the prior-to-posterior map $F_k$. 

Over the past four decades, considerable work has been devoted to understanding the posterior stability w.r.t. the prior in Bayesian inverse problems. Broadly, the literature falls into three categories: global robustness analysis, local sensitivity analysis, and local stability analysis.

Global robustness analysis investigates the range of values that a functional of the posterior, $\phi(F_k(\mu))$, may take as the prior $\mu$ varies within a specified class $\Gamma$. Specifically, it computes
\begin{equation*}
\underline{\phi}  = \inf_{\mu \in \Gamma}\phi(F_k(\mu)), \quad \overline{\phi} = \sup_{\mu \in \Gamma}\phi(F_k(\mu)),
\end{equation*}
and treats the interval $(\underline{\phi},\overline{\phi})$ as a robustness range of possible outputs of $\phi(F_k(\mu))$ (\cite{robust_overview}). The prior class $\Gamma$ can be parametric (e.g., normal distribution with fixed mean and bounded covariance (\cite{sets_post_bounded_variance}) or non-parametric (e.g., \cite{bounds_post_density_bounded, class_contaminations_example,quantile_example}). 

Local sensitivity analysis (e.g., \cite{local_sensitivity_original,local_sensitivity_new}) studies the local sensitivity of the posterior w.r.t. the prior. This sensitivity is defined as:
\begin{equation*}
s := \lim_{\epsilon \rightarrow 0}\frac{d(F_k(\mu),F_k(\mu_{\epsilon}))}{d(\mu,\mu_{\epsilon})},
\end{equation*}
where $\mu_{\epsilon}$ is the prior obtained by applying the perturbation $\epsilon$ to the base prior $\mu$.

Posterior stability w.r.t. the prior, without restricting the prior to a particular class, has been studied in different works (e.g., \cite{uniform_stability, infinitesimal_robustness,  lipschitz_stability_ip,post_stability_IPM}). For example, \cite{post_stability_IPM} proves a relation between the posterior difference and the prior difference under a family of integral probability metrics (IPM). Specifically, it shows:
\begin{equation}
d_c(F_k(\mu),F_k(\mu^\prime)) \leq K^\prime(\mu,\mu^\prime)d_{c_{y_k}}(\mu,\mu^\prime), \qquad \forall \mu,\mu^\prime \in \mathcal{P}_1(\X;c_{y_k}),
\label{ieq:ipm_bound}
\end{equation}
where $d_c$ and $d_{c_{y_k}}$ are two metrics within the IPM family, and the space $\mathcal{P}_1(\X;c_{y_k})$ is a set depending on the data $y_k$.
Because $K^\prime$ depends on both priors $\mu$ and $\mu^\prime$, this result does not establish global pointwise Lipschitz continuity of $F_k$.

The works in \cite{uniform_stability, infinitesimal_robustness} establish the local Lipschitz stability, showing that if the likelihood is bounded above, then for each prior $\mu \in \bar{\mathcal{P}}_k(\X)$, 
there exists $\delta >0$ such that 
\begin{equation}
d(F_k(\mu),F_k(\mu^\prime)) \leq  \tilde{K}(\mu)d(\mu,\mu^\prime), \quad if \; \mu^\prime \in \bar{\mathcal{P}}_k(\X) \; and \; d(\mu,\mu^\prime) < \delta.
\label{ieq:related_work:local_lipschitz}
\end{equation}

The study in \cite{lipschitz_stability_ip} proves the pointwise global Lipschitz continuity of $F_k$ under the TV distance, and the pointwise local Lipschitz continuity of $F_k$ under the Hellinger and $1$-Wasserstein distances. Under the same assumptions, our Theorem~\ref{thm:pointwise_lipschitz} yields an upper bound on the posterior distance $d(F_k(\mu),F_k(\mu^\prime))$ under the TV distance that is half the bound established in \cite{lipschitz_stability_ip}. A detailed comparison between our results and those in \cite{lipschitz_stability_ip} is provided in Appendix~\ref{sec:app:compare1}.

To conclude, prior work on global posterior robustness typically requires that the priors belong to a specific class. In contrast, most existing work that do not impose such restrictions on priors only establish local posterior stability with respect to the prior.

It is important to note that the Lipschitz constants in posterior stability results typically depend on the evidence terms associated with at least one of the priors.  The evidence, also known as the marginal likelihood, is a functional of the prior. Specifically, the evidence of a prior $\mu$, $Z_k(\mu)$, is defined as
\begin{equation*}
Z_k(\mu) = \E_{X \sim \mu}[p(y \mid X)],
\end{equation*}
where $y$ denotes the data and $p(y \mid X)$ is the likelihood. The evidence $Z_k(\mu)$ represents the expected value of the data’s probability density under the prior distribution $\mu$. 

In our result, the Lipschitz constant $K(\mu)$ in Equation~\eqref{ieq:lit_review:global} depends on the evidence $Z_k(\mu)$. To the best of our knowledge, the Lipschitz constants appearing in the existing results on the Lipschitz stability of the posterior either explicitly depend on the evidence terms of one or both priors (e.g., \cite{uniform_stability, infinitesimal_robustness, lipschitz_stability_ip}), or can be easily shown to be larger than another Lipschitz constant that is a functional of the priors' evidence terms under their respective assumptions (e.g., \cite{post_stability_IPM}). For instance, the Lipschitz constant $\tilde{K}$ in Equation~\eqref{ieq:related_work:local_lipschitz} is directly expressed as a functional of $Z_k(\mu)$ in \cite{uniform_stability, infinitesimal_robustness} and \cite{lipschitz_stability_ip}. Furthermore, the proof of Equation~\eqref{ieq:ipm_bound} in \cite{post_stability_IPM} shows the existence of another Lipschitz constant—depending on both $Z_k(\mu)$ and $Z_k(\mu^\prime)$—that is strictly smaller than the Lipschitz constant $K^\prime(\mu, \mu^\prime)$ stated in their result.

\paragraph{Posterior stability in state estimation problems}
\label{sec:stability:se}
For state estimation problems, the global Lipschitz continuity of $F_k$ under the TV distance has been shown in \cite{local_pf_dim}, \cite{ds}, and \cite{ip_ds} under a strong assumption that the likelihood is bounded above and below by positive constants. This assumption is stronger than ours in Theorem~\ref{thm:pointwise_lipschitz}. Under this assumption, our Theorem~\ref{thm:pointwise_lipschitz} also establishes global Lipschitz continuity of $F_k$ and provides an upper bound on the posterior distance $d(F_k(\mu),F_k(\mu^\prime))$ that is $\frac{1}{2}$ of the bound given in these works. See Appendix~\ref{sec:app:compare2} for a detailed derivation.

\subsubsection{Existing error analyses for approximate approaches for BSL}
This section reviews existing non-asymptotic error analyses for approximate approaches to BSL. As previously noted, such analyses are relatively underexplored. The existing results generally exhibit two limitations: 
\begin{itemize}
\item Most of them are restricted to specific methods and can not generalize to other algorithms;
\item Many rely on strong assumptions.
\end{itemize}

For instance, \cite{smoothing_error_jmlr} derives upper bounds on the error between the expectation of a specific type of functional under the approximate smoothing distribution and the expectation of the same functional under the true, unknown smoothing distribution in state estimation problems. The approximate smoothing distribution is the output of a particular sequential VI method based on the backward factorization of the smoothing distributions. The assumptions imposed in \cite{smoothing_error_jmlr} include the following condition:
\begin{equation*}
\sigma_1 \leq p(y_{k+1} \mid X_{k+1})p(X_{k+1} \mid X_k) \leq \sigma_2, \quad \forall X_k, X_{k+1},
\end{equation*}
where the data $y_{k+1}$ is given, $p(X_{k+1} \mid X_k)$ denotes the transition probability density from the state $X_{k+1}$ to the state $X_k$ (i.e., the state transition model), and $p(y_{k+1} \mid X_{k+1})$ denotes the probability density of $y_{k+1}$ given $X_{k+1}$ (i.e., the observation model). The condition $\sigma_1 \leq p(y_{k+1} \mid X_{k+1})p(X_{k+1} \mid X_k)$ is quite restrictive. For example, it is not satisfied in common settings where both the state transition and observation models are Gaussian.

The work in \cite{pf_adaptive} analyzes the learning error of a block-adaptive particle filter that periodically updates the number of particles. The study in \cite{ip_ds} provides learning error bounds for the bootstrap particle filter. Both analyses are based on the assumption that the observation model is bounded both above and below by positive constants—an assumption that is quite strong.

In \cite{tt_sequential}, the authors present an error analysis for a tensor train-based sequential learning method in parameter-state estimation problems. Although the assumptions made in \cite{tt_sequential} are mild, the analysis is tailored to a specific procedure used by this tensor train-based method and is not able to provide meaningful results for other approaches.

\subsection{Organization of the paper}
The remainder of this paper is organized as follows. Section~\ref{sec:background} introduces the necessary notation, definitions, and the problems considered. Section~\ref{sec:main_results} presents our assumptions and main theorems. Section~\ref{sec:discussion} provides several discussions of the main theorems. In Section~\ref{sec:application_example}, we apply our results to a specific class of BSL methods—online VI. Section~\ref{sec:illustrative_example} presents an illustrative numerical example. Finally, we conclude the paper in Section~\ref{sec:conclusion}.

\section{Background}
\label{sec:background}
In this section, we present the necessary notation, provide the definition of pointwise globally Lipschitz continuous functions, introduce three types of distances for probability measures, and formally describe the problems considered in this paper.
\subsection{Notation}
Let $\reals$ denote the set of real numbers, and $\reals_{+}$ the set of positive real numbers. For $a, b \in \reals$, we let $a \wedge b$ and $a \vee b$ denote the minimum and maximum of $a$ and $b$, respectively. The symbol $d_{\Z}$ denotes a metric for the space $\Z$ and $\mathcal{B}(\Z)$ denotes the Borel $\sigma$-algebra on $\Z$. If a measure $\mu$ is absolutely continuous w.r.t. a measure $\nu$, we write $\mu \ll \nu$. 

Random variables are denoted by uppercase letters. The realization of a random variable is denoted by the corresponding lowercase letter (e.g., $y$ for $Y$). Let expectation and probability be denoted as $\E$ and $\mathbb{P}$, respectively. If the probability measure $\mathbb{P}$ is absolutely continuous w.r.t. the Lebesgue measure, its probability density is denoted by $p$. The conditional probability density $p(X=x \mid Y=y)$ is denoted by $p(x \mid y)$ for brevity. The symbol $\mathcal{N}(m,R)$ denotes a normal distribution with mean $m$ and covariance $R$. The value of the probability density function (pdf) of this normal distribution evaluated at $z$ is denoted by $p_{\mathcal{N}}(z \mid m,R)$.

We use $d_{TV}$, $d_H$ and $W_1$ to denote the total variation distance, Hellinger distance and $1$-Wasserstein distance, respectively.

\subsection{Definitions}
This section provides definitions of pointwise globally Lipschitz continuous functions and the distances for probability measures. 

\subsubsection{Pointwise globally Lipschitz continuous functions}
The definition of pointwise globally Lipschitz continuous functions is as follows.
\begin{definition}[\textbf{Pointwise globally Lipschitz continuous functions}]\label{def:pointwise_lipschitz}
Let $(\Z,d_{\Z})$ and $(\tilde{\Z},d_{\tilde{\Z}})$ be two metric spaces. A map $f: \Z \rightarrow \tilde{\Z}$ is called pointwise globally Lipschitz continuous if for every $z \in \Z$, there exists a constant $K(z) \in \reals_{+}$ such that
\begin{equation}
d_{\tilde{\Z}}\left(f(z), f(z^\prime)\right) \leq K(z)d_{\Z}(z,z^\prime), \quad \forall z^\prime \in \Z.
\label{eq:pointwise_global_Lipschitz}
\end{equation}
For every $z \in \Z$, a constant $K(z) \in \reals_{\geq 0}$ that satisfies Equation~\eqref{eq:pointwise_global_Lipschitz} is called a pointwise Lipschitz constant of $f$ at $z$.
\end{definition}

\begin{remark}
Function $f$ is called pointwise \textit{locally} Lipschitz continuous if the inequality in Equation~\eqref{eq:pointwise_global_Lipschitz} only holds for $z^\prime \in B(z,r(z))$, where $B(z,r(z))$ denotes an open ball centered at $z$ with radius $r(z)$ (\cite{pointwise_lipschitz}).
\end{remark}

\subsubsection{Distances for probability measures}
In this section, we introduce three types of distances that quantify the difference between probability measures: total variation distance, Hellinger distance, and Wasserstein distance.

Let $\mu$ and $\mu^\prime$ be two probability measures defined on the measurable space $(\Z,\mathcal{B}(\Z))$. Let $\mathcal{P}(\Z)$ denote the set of all probability measures defined on the space $(\Z,\mathcal{B}(\Z))$. Assume that $\mu$ and $\mu^\prime$ are both absolutely continuous with respect to the $\sigma$-finite measure $\nu$ (such $\nu$ always exists). The total variation distance and Hellinger distance are defined as follows.

\begin{definition}[\textbf{Total variation distance}]
The total variation distance between two probability measures $\mu, \mu^\prime \in \mathcal{P}(\Z)$ is defined by
\begin{equation*}
d_{TV}(\mu, \mu^\prime) := \sup_{A \in \mathcal{B}(\Z)}\big \lvert \mu(A) - \mu^\prime(A) \big \rvert =\frac{1}{2}\int_{\Z}\bigg \lvert \frac{d\mu}{d\nu}(z) - \frac{d\mu^\prime}{d\nu}(z) \bigg \rvert \nu(dz). 
\end{equation*}
\end{definition}

\begin{definition}[\textbf{Hellinger distance}]
The Hellinger distance between two probability measures $\mu, \mu^\prime \in \mathcal{P}(\Z)$ is defined by
\begin{equation*}
d_{H}(\mu, \mu^\prime) := \sqrt{\frac{1}{2}\int_{\Z}\left( \sqrt{\frac{d\mu}{d\nu}(z)} - \sqrt{\frac{d\mu^\prime}{d\nu}(z)} \right )^2\nu(dz)}  . 
\end{equation*}
\end{definition}

The definitions of the Wasserstein distance and Wasserstein space are provided below.
\begin{definition}[\textbf{Wasserstein distance and Wasserstein space}]
Let $(\Z,d_{\Z})$ be a Polish metric space. For two probability measures $\mu,\mu^\prime \in \mathcal{P}_q(\Z)$, the Wasserstein distance of order $q$ between $\mu$ and $\mu^\prime$ is defined as
\begin{align*}
W_q(\mu,\mu^\prime) :=& \inf_{\gamma \in \Gamma(\mu,\mu^\prime)}\left(\E_{(Z,Z^\prime)\sim \gamma}[d^q_{\Z}(Z,Z^\prime)]\right)^{\frac{1}{q}} ,
\end{align*}
where $\Gamma(\mu,\mu^\prime)$ denotes the set of all couplings of $\mu$ and $\mu^\prime$, and $\mathcal{P}_q(\Z)$ is a space defined as:
\begin{equation*}
\mathcal{P}_q(\Z):= \bigg \{\mu \in \mathcal{P}(\Z) : \int_{\Z}d_{\Z}(z_0,z)^q\mu(dz) < \infty \bigg \},
\end{equation*}
where $z_0 \in \Z$. The space $\mathcal{P}_q$ is called the Wasserstein space of order $q$ (\cite{ot_old_new}). The distance $W_q$ is referred to as the $q$-Wasserstein distance. In this paper, we focus on the $1$-Wasserstein distance, $W_1$, which is well-defined on the space $\mathcal{P}_1$. 

\end{definition}

\subsection{Problem statement}
\label{sec:problem_statement}
This section formally introduces the three problems considered in this paper: inverse problems, state estimation in dynamical systems, and parameter-state estimation in dynamical systems.

\subsubsection{Inverse problems}\label{section:def_ip}

In inverse problems, the objective is to learn an unknown variable $X$ using the received data $\Y_{1:N}=(y_1,y_2,\cdots,y_N)$. From the Bayesian perspective, $X$ is treated as a random variable. Let $\X$ denote the space for the realizations of $X$, equipped with a metric $d_{\X}$. We assume $(\X,d_{\X})$ is a Polish metric space. Each data point $y_i$ can be viewed as a realization of the observable $Y$. We assume the conditional distribution of $Y$ given $X$ admits a density $p(y \mid x) = h(y,x)$ w.r.t. the Lebesgue measure. Given the data $y_k$, the likelihood model $h(y_k,x)$, as a function of $x$, is assumed to be measurable. 

At each update step $k$, the prior refers to the distribution of $X$ before assimilating the data $y_k$, and the posterior refers to the distribution of $X$ after incorporating $y_k$. 
Given a prior $\mu$, by Bayes' rule, the corresponding posterior $F_k\mu$ (also denoted by $F_k(\mu)$) is given by
\begin{equation*}
F_k\mu(dx) = \frac{h(y_k,x)\mu(dx)}{\int_{\X}h(y_k ,x)\mu(dx)}.
\end{equation*}
If the prior $\mu$ satisfies 
\begin{equation*}
0 < \int_{\X}h(y_k,x)\mu(dx) < \infty,
\end{equation*}
then the posterior $F_k\mu$ is well-defined. 

\subsubsection{State estimation}\label{section:def_se}
State estimation requires the solution of a sequence of inverse problems, interspersed with propagation of the measure through the dynamics. 

Consider the following discrete-time Markov dynamical system with partially observable states: 
\begin{align}
\begin{split}
p(x_k \mid x_{0:k-1}) &= p(x_k \mid x_{k-1}) =T_k(x_{k},x_{k-1}), \\
p(y_k \mid x_{0:k},\Y_{1:k-1}) &= p(y_k \mid x_k) = h_k(y_k,x_k),
\end{split}
\label{sys:known}
\end{align}
where $x_k$ and $y_k$ represent the realizations of the hidden state $X_k$ and observation $Y_k$, respectively. Here $x_{0:k}$ represents the sequence $(x_0,x_1,\cdots,x_k)$. The state space $\X$ is assumed to be identical across all time steps $k$ and forms a Polish metric space with a metric $d_{\X}$. Given the data $y_k$, the observation model $h_k(y_k,x_k)$, as a function of $x_{k}$, is assumed to be measurable. We also assume that the state transition model $T_k(x_k,x_{k-1})$ is measurable. 

At each time step $k$, the prior denotes the distribution of the state $X_{k-1}$ before incorporating the data $y_k$, and the posterior represents the distribution of the state $X_{k}$ after assimilating $y_k$. Let the prior $\mu$ be an arbitrary probability measure. By Bayes' rule, the corresponding posterior $F_k\mu$ is given by
\begin{equation*}
F_k\mu(dx_k) = \frac{h_k(y_k,x_k)\left(\int_{\X}T_k(x_k, x_{k-1})\mu(dx_{k-1})\right)dx_k}{\int_{\X}h_k(y_k,x_k)\left(\int_{\X}T_k(x_k, x_{k-1})\mu(dx_{k-1})\right)dx_k}.
\end{equation*}
If $\mu$ satisfies
\begin{equation*}
\int_{\X}T_k(x_k,x_{k-1})\mu(dx_{k-1}) < \infty, \quad \forall x_k \in \X,
\end{equation*}
and 
\begin{equation*}
0 < \int_{\X}h_k(y_k,x_k)\left(\int_{\X}T_k(x_k, x_{k-1})\mu(dx_{k-1})\right)dx_k < \infty,
\end{equation*}
then the posterior $F_k\mu$ is well-defined.

State estimation problems assume fully known models. The next subsection introduces parameter-state estimation, where model parameters are unknown.

\subsubsection{Parameter-state estimation}\label{section:def_ps}
Consider the following discrete-time Markov dynamical system with unknown system parameters:
\begin{align}
\begin{split}
p(x_k \mid x_{0:k-1},w) &= p(x_k \mid x_{k-1},w) =T_k(x_{k},x_{k-1},w), \\
p(y_k \mid x_{0:k},\Y_{1:k-1},w) &= p(y_k \mid x_k,w) = h_k(y_k,x_k,w),
\end{split}
\label{sys:unknown}
\end{align}
where $w$ represents the realization of the system parameters $W$. We assume that the state space $\X$ is time-invariant. The state space $\X$ and parameter space $\W$ form a Polish metric space $(\X \times \W, d_{\X \times \W})$ with a metric $d_{\X \times \W}$. Assume the state transition model $T_k$ is measurable. Given the data $y_k$, 
the observation model $h(y_k,x_k,w)$, as a function of $x_k $ and $w$, is also assumed to be measurable.

At each time step $k$, the prior refers to the joint distribution of the state $X_{k-1}$ and system parameters $W$ before assimilating the data $y_k$, and the posterior refers to the joint distribution the state $X_{k}$ and $W$ after assimilating $y_k$. Suppose that the prior $\mu$ is absolutely continuous with respect to the Lebesgue measure, with density $\pi$. Then, by Bayes' rule, the corresponding posterior $F_k\mu$ is given by
\begin{align*}
    &F_k\mu(dx_k,dw) = \frac{h_k(y_k, x_k,w)\left(\int_{\X}T_k(x_k, x_{k-1},w)\pi(x_{k-1},w)dx_{k-1}\right)dx_kdw}{\int_{\X\times \W}h_k(y_k, x_k,w)\left(\int_{\X}T_k(x_k, x_{k-1},w)\pi(x_{k-1},w)dx_{k-1}\right)dx_kdw}.
\end{align*}
If $\mu$ satisfies:
\begin{equation*}
\int_{\X}T_k(x_k,x_{k-1},w)\pi(x_{k-1},w)dx_{k-1} < \infty, \quad \forall x_k \in \X, w \in \W,
\end{equation*}
and 
\begin{equation*}
0 < \int_{\X\times \W}h_k(y_k, x_k,w)\left(\int_{\X}T_k(x_k, x_{k-1},w)\pi(x_{k-1},w)dx_{k-1}\right)dx_kdw < \infty,
\end{equation*}
then the posterior $F_k\mu$ is well-defined. 
\begin{remark}
Systems~\ref{sys:known} and~\ref{sys:unknown} encompass both autonomous and non-autonomous systems. If external inputs are present, they are assumed to be incorporated into the state transition model $T_k$. 
\end{remark}

\subsection{Unified Bayesian learning framework for different problems}
\label{sec:unified_framework}
In this section, we present a unified framework for Bayesian learning in the three problem settings introduced earlier: inverse problems, state estimation, and parameter-state estimation.

Throughout this paper, let $\bar{\X}$ represent the space $\X$ in inverse and state estimation problems, and $\X\times \W$ in parameter-state estimation problems. Let $\mathcal{P}(\bar{\X})$ denote the set of all probability measures defined on the measurable space $(\bar{\X},\mathcal{B}(\bar{\X}))$. Similarly, let $\mathcal{M}(\bar{\X})$ denote the set of all $\sigma$-finite measures defined on $(\bar{\X},\mathcal{B}(\bar{\X}))$. Let $\lambda$ denote the Lebesgue measure defined on the measurable space $(\X \times \W, \mathcal{B}(\X\times\W))$.

We refer to the prior distributions at step $k$ of which the corresponding posteriors are well-defined as \textit{admissible priors}. We then define $\bar{\mathcal{P}}_k(\X)$ as the set of all admissible prior probability measures in inverse and state estimation problems. For parameter-state estimation problems, $\bar{\mathcal{P}}_k(\X\times \W)$ is defined as the set of admissible priors that are absolutely continuous with respect to the Lebesgue measure. The formal definition of $\bar{\mathcal{P}}_k(\bar{\X})$ is provided below.
\begin{definition}[Admissible prior space $\bar{\mathcal{P}}_k(\bar{\X})$]
\label{def:set_P_k}
At step $k$, given the data $y_k$, the space $\bar{\mathcal{P}}_k(\bar{\X})$ is defined as follows:
\begin{itemize}
\item in inverse problems: $\bar{\mathcal{P}}_k(\X):=\{\mu \in \mathcal{P}(\X):0<\int_{\X}h(y_k,x)\mu(dx)<\infty \}$,

\item in state estimation problems: 
\begin{align*}
&\bar{\mathcal{P}}_k(\X) :=\bigg \{\mu \in \mathcal{P}(\X): \int_{\X}T_k(x_k,x_{k-1})\mu(dx_{k-1})<\infty \quad \forall x_{k}\in \X, \\
&\qquad 0< \int_{\X}h_k(y_k,x_k)\left(\int_{\X}T_k(x_k,x_{k-1})\mu(dx_{k-1})\right)dx_k < \infty \bigg \},
\end{align*}

\item in parameter-state estimation problems: 
\begin{align*}
&\bar{\mathcal{P}}_k(\X\times\W):=\bigg \{\mu \in \mathcal{P}(\X\times \W): \mu \ll \lambda, \\
&\quad \int_{\X}T_k(x_k,x_{k-1},w)\frac{d\mu}{d\lambda}(x_{k-1},w)dx_{k-1}<\infty \quad  \forall x_{k}\in \X, w \in \W, \\
&\qquad 0< \int_{\X\times\W}h_k(y_k,x_k,w)\left(\int_{\X}T_k(x_k,x_{k-1},w)\frac{d\mu}{d\lambda}(x_{k-1},w)dx_{k-1}\right)dx_kdw< \infty \bigg \}.
\end{align*}
\end{itemize}
\end{definition}
\begin{remark}
The space $\bar{\mathcal{P}}_k(\bar{\X})$ is determined by the system. For instance, in inverse problems, given the data $y_k$, if the likelihood model $h$ satisfies $h(y_k,x)>0$ for all $x \in \X$ and $\sup_{x \in \X}h(y_k,x) < \infty$, then we have $\bar{\mathcal{P}}_k(\X)=\mathcal{P}(\X)$.
\end{remark}

Next, we introduce the formal definition of the prior-to-posterior map $F_k$ and evidence $Z_k$. 
\begin{definition}[Prior-to-posterior map $F_k$ and evidence $Z_k$ ]
\label{def:post_propagation&evidence}
Let $F_k: \bar{\mathcal{P}}_k(\bar{\X}) \rightarrow \mathcal{P}(\bar{\X})$ and $\tilde{F}_k: \bar{\mathcal{P}}_k(\bar{\X}) \rightarrow \mathcal{M}(\bar{\X})$ be two functions. For a probability measure $\mu \in \bar{\mathcal{P}}_k(\bar{\X})$, let $F_k\mu$ and $\tilde{F}_k\mu$ denote the function applications $F_k(\mu)$ and $\tilde{F}_k(\mu)$, respectively. At step $k$, given the data $y_k$,  the probability measure $F_k\mu$ is defined by 
\begin{equation}
F_k\mu(A) := \frac{\tilde{F}_k\mu(A)}{Z_k(\mu)}, \quad \forall A \in \mathcal{B}(\bar{\X}),
\label{eq:def_F_k}
\end{equation}
where the evidence $Z_k$, which is a real-valued function, is defined as
\begin{equation}
Z_k(\mu) := \int_{\bar{\mathcal{X}}}\tilde{F}_k\mu(d\bar{x}).
\end{equation}
The measure $\tilde{F}_k\mu$ is defined as follows:
\begin{itemize}
    \item in inverse problems: $\tilde{F}_k\mu(dx) := h(y_k, x)\mu(dx)$,

\item in state estimation problems: 
$$\tilde{F}_k\mu(dx):=h_k(y_k,x)\left(\int_{\mathcal{X}}T_k(x,x_{k-1})\mu(dx_{k-1})\right)dx,$$

\item in parameter-state estimation problems: 
\begin{equation*}
\tilde{F}_k\mu(dx,dw) := h_k(y_k,x,w)\left(\int_{\X}T_k(x, x_{k-1},w)\pi(x_{k-1},w)dx_{k-1}\right)dxdw,
\end{equation*}
where $\pi$ is the density of $\mu$ with respect to the Lebesgue measure.
\end{itemize}
\end{definition}
As illustrated in Figure~\ref{fig:break_down_approx_sequential_learning}, given the data $y_k$, $P_k$ and $Q_k^*$ are obtained by $P_k = F_k(P_{k-1})$ and $Q_k^* = F_k(Q_{k-1})$, respectively. The approximate posterior $Q_k$ is computed through $Q_k = \hat{F}_k(Q_{k-1})$.

\section{Main results}\label{sec:main_results}
In this section, we present the main theorems of this paper. 
\subsection{Assumptions}
In this section, we present the assumptions necessary for establishing the theorems in the following sections. To distinguish the assumptions by context, we label them with "IP" for inverse problems, "SE" for state estimation problems, and "PS" for parameter-state estimation problems. We begin with the assumptions for inverse problems.
\begin{assumptionp}{IP.1}[\textbf{Upper-bounded likelihood}]
Given the data $y_k$, the likelihood model $h(y_k,x)$ in the inverse problems defined in Section~\ref{section:def_ip} satisfies 
\begin{equation}
C_h(y_k) := \sup_{x \in \X}h(y_k,x) < \infty. 
\label{eq:bound_likelihood}
\end{equation}
\label{assump:ip_single_step}
\end{assumptionp}

\begin{remark}
For a given data $y_k$, $C_h(y_k)$ is a constant. This assumption is quite mild and holds for most inverse problems. For instance, in the common setting where $Y = H(X) + V$, with $H$ being a forward map and $V$ an additive noise, this assumption is satisfied when the pdf of $V$ is upper bounded—such as when $V$ follows a normal distribution.
\end{remark}

The following assumption is required for theorems related to the $1$-Wasserstein distance for inverse problems.
\begin{assumptionp}{IP.2}[\textbf{Lipschitz continuous likelihood}]\label{assump:lipschitz_likelihood}
Given the data $y_k$, for inverse problems, the likelihood model $h(y_k,x)$, as a function of $x \in \X$, is (globally) Lipschitz continuous. Let $\lVert h \rVert_{\text{Lip}}(y_k)$ denote its best Lipschitz constant. 
\end{assumptionp}

Next, we present the assumptions required for state estimation problems.
\begin{assumptionp}{SE.1}
 At time step $k$, given the data $y_k$, the state transition model $T_k$ and observation model $h_k$ in system~\ref{sys:known} satisfy
\begin{equation}
C_{Th}(y_k;k) := \sup_{x_{k-1}\in \X}\int_{\X}h_k(y_k,  x)T_k(x, x_{k-1})dx < \infty.
\label{eq:C_hf}
\end{equation}
\label{assump:se_single_step}
\end{assumptionp}

\begin{remark}
 Assumption~\ref{assump:se_single_step} is a weaker condition than assuming $h_k(y_k,x)$ is uniformly upper bounded for $x \in \X$, as the latter directly implies Equation~\eqref{eq:C_hf}. 
\label{remark:se_single_step}
\end{remark}

The following assumption is used for theorems involving the $1$-Wasserstein distance for state estimation problems.
\begin{assumptionp}{SE.2}\label{assump:lipschitz_hf:se}
At time step $k$, given the data $y_k$, the state transition model $T_k$ and observation model $h_k$ in system~\ref{sys:known} satisfy
\begin{equation*}
T_{\text{Lip}}(x_k;k) := \sup_{x,x^\prime \in \mathcal{X}, x \neq x^\prime}\frac{\lvert T_k(x_k,x) - T_k(x_k, x^\prime) \rvert}{d_{\X}(x,x^\prime)} < \infty, \quad \forall x_k \in \mathcal{X},
\end{equation*}
and
\begin{equation}
C^*_{Th}(y_k;k) := \int_{\mathcal{X}}h_k(y_k, x_k)T_{\text{Lip}}(x_k;k)dx_k < \infty.
\label{eq:C_hf_star}
\end{equation}
\end{assumptionp}

Lastly, we state the assumptions required for parameter-state estimation problems.
\begin{assumptionp}{PS.1}
At time step $k$, given the data $y_k$, the state transition model $T_k$ and observation model $h_k$ in system~\ref{sys:unknown} satisfy
\begin{equation}
\tilde{C}_{Th}(y_k;k) := \sup_{x_{k-1}\in \X, w \in \W}\int_{\X}h_k(y_k,x, w)T_k(x,x_{k-1},w)dx < \infty.
\label{eq:tilde_C_hf}
\end{equation}
\label{assump:ps_single_step}
\end{assumptionp}

\begin{remark}
Assumption~\ref{assump:ps_single_step} is mild and sets a weaker condition than that the observation model $h_k$ satisfies $\sup_{x \in \X, w \in \W}h_k(y_k,x,w) < \infty$. 
\label{remark:ps_single_step}
\end{remark}

The following assumptions are required for theorems related to the $1$-Wasserstein distance for parameter-state estimation problems.
\begin{assumptionp}{PS.2}\label{assump:lipschitz_hf:ps}
At time step $k$, given the data $y_k$, the state transition model $T_k$ and observation model $h_k$ in system~\ref{sys:unknown} satisfy
\begin{align*}
&\tilde{f}_{\text{Lip}}(x_k;k) \\
&\quad := \sup_{(x,w),(x^\prime,w^\prime) \in \mathcal{X}\times \W, (x,w) \neq (x^\prime,w^\prime)}\frac{\lvert \tilde{f}_k(y_k,x_k,x,w) - \tilde{f}_k(y_k,x_k,x^\prime,w^\prime) \rvert}{d_{\X\times \W}((x,w),(x^\prime,w^\prime))} < \infty, \quad \forall x_k \in \mathcal{X},
\end{align*}
where $\tilde{f}_k(y_k,x_k,x,w):=h_k(y_k,x_k,w)T_k(x_k,x,w), \quad \forall x \in \X,  w \in \W$,\\
and
\begin{equation}
\tilde{C}^*_{Th}(y_k;k) := \int_{\mathcal{X}}\tilde{f}_{\text{Lip}}(x_k;k)dx_k < \infty.
\label{eq:tilde_C_hf_star}
\end{equation}
\end{assumptionp}

\begin{assumptionp}{PS.3}\label{assump:metric_space:ps}
The metric space $(\X\times \W,d_{\X \times \W})$ satisfies
\begin{equation*}
d_{\X \times \W}((x,w),(x,w^\prime)) \leq d_{\X \times \W}((\tilde{x},w),(x^\prime,w^\prime)), \quad \forall x,\tilde{x},x^\prime \in \X, \; w,w^\prime \in \W.
\end{equation*}
\end{assumptionp}
\begin{remark}
Assumption~\ref{assump:metric_space:ps} is generally not restrictive as it is actually the property shared by many metrics. For instance, the Euclidean distance satisfies this assumption.
\end{remark}

An additional assumption required for establishing results under the $1$-Wasserstein distance for all problems is as follows.
\begin{assumptionp}{IP-SE-PS.4}[\textbf{Bounded metric space}]\label{assump:bounded_space}
The metric space $(\bar{\X},d_{\bar{\X}})$ satisfies
\begin{equation*}
\sup_{\bar{x},\bar{x}^\prime \in \bar{\X}}d_{\bar{\X}}(\bar{x},\bar{x}^\prime) = D < \infty,
\end{equation*}
where $\bar{\X}$ represents $\X$ in inverse problems and state estimation problems, and represents $\X \times \W$ in parameter-state estimation problems, as introduced in Section~\ref{sec:unified_framework}.
\end{assumptionp}

\subsection{Pointwise global Lipschitz continuity of the prior-to-posterior map}
\label{sec:post}
Theorem~\ref{thm:pointwise_lipschitz} establishes the pointwise global Lipschitz continuity of the prior-to-posterior map $F_k$.
\begin{theorem}
[\textbf{Pointwise global Lipschitz continuity of prior-to-posterior map $F_k$}]
\label{thm:pointwise_lipschitz}
Given the data $y_k$, the prior-to-posterior map $F_k: (\bar{\bar{\mathcal{P}}}_k(\bar{\X}), d) \rightarrow (\mathcal{P}(\bar{\X}), d)$ is pointwise globally Lipschitz continuous, i.e., for every $\mu \in \bar{\bar{\mathcal{P}}}_k(\bar{\X})$, there exists a constant $K(\mu;y_k)$ such that
\begin{equation}
d(F_k(\mu), F_k(\mu^\prime)) \leq K(\mu;y_k)d(\mu,\mu^\prime), \quad \forall \mu^\prime \in \bar{\bar{\mathcal{P}}}_k(\bar{\X}),
\label{eq:thm:pointwise_Lipschitz}
\end{equation}
\begin{itemize}
\item in \textbf{inverse problems}
\begin{itemize}
    \item when $d$ is the total variation distance or the Hellinger distance, and $\bar{\bar{\mathcal{P}}}_k(\bar{\X})$ represents $\bar{\mathcal{P}}_k(\X)$, under Assumption~\ref{assump:ip_single_step}; or
    \item when $d$ is the $1$-Wasserstein distance, and $\bar{\bar{\mathcal{P}}}_k(\bar{\X})$ represents $\bar{\mathcal{P}}_k(\X) \bigcap \mathcal{P}_1(\X)$, under Assumptions~\ref{assump:ip_single_step},~\ref{assump:lipschitz_likelihood}, and~\ref{assump:bounded_space};
\end{itemize}
\item in \textbf{state estimation problems}
\begin{itemize}
    \item when $d$ is the total variation distance or the Hellinger distance, and $\bar{\bar{\mathcal{P}}}_k(\bar{\X})$ represents $\bar{\mathcal{P}}_k(\X)$, under Assumption~\ref{assump:se_single_step}; or
    \item when $d$ is the $1$-Wasserstein distance, and $\bar{\bar{\mathcal{P}}}_k(\bar{\X})$ represents $\bar{\mathcal{P}}_k(\X) \bigcap \mathcal{P}_1(\X)$, under Assumptions~\ref{assump:lipschitz_hf:se} and~\ref{assump:bounded_space};
\end{itemize}
\item in \textbf{parameter-state estimation problems}
\begin{itemize} 
    \item  when $d$ is the total variation distance or the Hellinger distance, and $\bar{\bar{\mathcal{P}}}_k(\bar{\X})$ represents $\bar{\mathcal{P}}_k(\X\times \W)$, under Assumption~\ref{assump:ps_single_step}; or
    \item when $d$ is the $1$-Wasserstein distance, and $\bar{\bar{\mathcal{P}}}_k(\bar{\X})$ is $\bar{\mathcal{P}}_k(\X\times \W) \bigcap \mathcal{P}_1(\X \times \W)$, under Assumptions~\ref{assump:ps_single_step},~\ref{assump:lipschitz_hf:ps},~\ref{assump:metric_space:ps}, and~\ref{assump:bounded_space}.
\end{itemize}
\end{itemize}

The pointwise global Lipschitz constants $K(\mu;y_k)$ associated with different distances and various problems are summarized in Table~\ref{table:normalized:Z}.

\begin{table}
\centering
\caption{Pointwise global Lipschitz constant $K(\mu;y_k)$ of the prior-to-posterior map $F_k$ for different problems and distances of probability measures.}
\label{table:normalized:Z}
\begin{tabularx}{1.0\linewidth}{ 
  | >{\hsize=0.8\hsize\centering\arraybackslash}X 
  | >{\hsize=1.1\hsize\centering\arraybackslash}X 
  | >{\hsize=0.9\hsize\centering\arraybackslash}X |>{\hsize=1.2\hsize\centering\arraybackslash}X | }
 \hline
 $d$ / Problem & Inverse problems & State estimation &Parameter-state estimation \\
 \hline
 \rule{0pt}{15pt} $d_{TV}$  & $\frac{C_h(y_k)}{Z_k(\mu)}$ & $\frac{C_{Th}(y_k;k)}{Z_k(\mu)}$ & $\frac{\tilde{C}_{Th}(y_k;k)}{Z_k(\mu)}$\\
 \hline
 \rule{0pt}{15pt} $d_H$ & $2\sqrt{\frac{C_h(y_k)}{Z_k(\mu)}}$ & $2\sqrt{\frac{C_{Th}(y_k;k)}{Z_k(\mu)}}$ &$2\sqrt{\frac{\tilde{C}_{Th}(y_k;k)}{Z_k(\mu)}}$ \\
 \hline
 \rule{0pt}{15pt} $W_1$ &$\frac{2D\lVert h \rVert _{\text{Lip}}(y_k)+C_h(y_k)}{Z_k(\mu)}$ & $\frac{2DC^*_{Th}(y_k;k)}{Z_k(\mu)}$ & $\frac{2D\tilde{C}^*_{Th}(y_k;k)+\tilde{C}_{Th}(y_k;k)}{Z_k(\mu)}$ \\
\hline
\end{tabularx}
\end{table}
\end{theorem} 

\begin{proof}
See Appendix~\ref{sec:app:pointwise_lipschitz}.
\end{proof}

\begin{remark}
Note that the posterior $F_k(\mu')$ depends on the evidence $Z_k(\mu')$, which typically differs from the evidence $Z_k(\mu)$. As a result, without imposing very strong assumptions, it is challenging to establish a global bound on the posterior distance $d(F_k(\mu), F_k(\mu'))$ that holds uniformly for all $\mu' \in \bar{\bar{\mathcal{P}}}_k(\bar{\X})$ and is independent of $Z_k(\mu^\prime)$. To the best of our knowledge, Theorem~\ref{thm:pointwise_lipschitz} provides the first such global bound under the Hellinger and Wasserstein distances.
\end{remark}
The Lipschitz constant $K(\mu; y_k)$ depends only on one of the priors, $\mu$, and is independent of the other prior, $\mu^\prime$. Moreover, inequality~\eqref{eq:thm:pointwise_Lipschitz} holds for all $\mu, \mu^\prime \in \bar{\bar{\mathcal{P}}}_k(\bar{\X})$. 

Given that $P_k=F_k(P_{k-1})$ and $Q_k^*=F_k(Q_{k-1})$, Theorem~\ref{thm:pointwise_lipschitz} implies that, regardless of how different the priors $P_{k-1}$ and $Q_{k-1}$ are, the propagated learning error satisfies 
\begin{equation}
d(P_k,Q_k^*) \leq K(P_{k-1};y_k)d(P_{k-1},Q_{k-1}), 
\label{eq:remark:1}
\end{equation}
and 
\begin{equation}
d(P_k,Q_k^*) \leq K(Q_{k-1};y_k)d(P_{k-1},Q_{k-1}).
\label{eq:remark:2}
\end{equation}
As discussed in Section \ref{sec:related_work}, for the Hellinger and Wasserstein distances, the existing stability results for Bayesian learning either (i) hold with a Lipschitz constant depending on both priors (e.g.,\cite{post_stability_IPM}), or (ii) establish only the pointwise local Lipschitz continuity of $F_k$ (e.g., \cite{uniform_stability, infinitesimal_robustness,lipschitz_stability_ip}). In case (i), the bound takes the form
\begin{equation}
d(P_k,Q_k^*) \leq K^\prime(P_{k-1},Q_{k-1};y_k)d(P_{k-1},Q_{k-1}).
\label{eq:remark:other}
\end{equation}
In case (ii), the bound on $d(P_k,Q_k^*)$ is valid only when $d(P_{k-1},Q_{k-1})$ is sufficiently small. In practical settings, however, the prior error $d(P_{k-1},Q_{k-1})$ can be large, and such results often fail to yield bounds on $d(P_k,Q_k^*)$.

Combining inequality~\eqref{eq:remark:1} with the triangle inequality $d(P_k,Q_k) \leq d(P_k,Q_k^*)+d(Q_k^*,Q_k)$, we obtain
\begin{equation}
d(P_k,Q_k) \leq K(P_{k-1};y_k)d(P_{k-1},Q_{k-1})+d(Q_k^*,Q_k).
\label{eq:remark:1_2}
\end{equation}
Recursively applying inequality~\eqref{eq:remark:1_2} from step $k$ down to $1$ yields an upper bound on the learning error $d(P_k,Q_k)$. The approximate posteriors $Q_j, j \leq k$ only enter this bound through the incremental approximation errors $d(Q_j^*,Q_j), j \leq k$. This bound shows that if the incremental approximation errors $d(Q_j,Q_j^*)$ remain bounded for all $j \leq k$, then the overall learning error $d(P_k,Q_k)$ is also bounded. Further discussion of this bound is provided in Sections~\ref{sec:final_error} and~\ref{sec:discussion:stability}.

Similarly, applying inequality~\eqref{eq:remark:2} yields
\begin{equation}
d(P_k,Q_k) \leq K(Q_{k-1};y_k)d(P_{k-1},Q_{k-1})+d(Q_k^*,Q_k),
\label{eq:remark:2_2}
\end{equation}
and the recursive application of inequality~\eqref{eq:remark:2_2} from $k$ to $1$ leads to a second bound on $d(P_k,Q_k)$ that depends only on the approximate posteriors $Q_j, j \leq k$ and the incremental approximation errors. This bound can be computed in many practical scenarios. The learning error bounds derived from inequalities~\eqref{eq:remark:1_2} and~\eqref{eq:remark:2_2} are presented in detail in Section~\ref{sec:final_error}.

In contrast, learning error bounds derived from inequality~\eqref{eq:remark:other} depend on the incremental approximation errors as well as both the exact and approximate posteriors. These bounds are generally not computable, as the exact posteriors are typically intractable. Moreover, they do not guarantee that the learning error is bounded when the incremental approximation errors are bounded. Given bounded incremental approximation errors, these learning error bounds themselves may still be unbounded, due to their dependence on the approximate posteriors (the condition that the incremental approximation errors $d(Q_j^*,Q_j), j\leq k$ are bounded does not guarantee that a functional of the approximate posteriors $Q_j, j \leq k$ is bounded).

These results underscore the necessity of Theorem~\ref{thm:pointwise_lipschitz} for the establishment of a broadly applicable and practical analysis framework for approximate approaches of BSL.

\subsection{Learning error bounds of approximate approaches for BSL }\label{sec:final_error}
Theorem~\ref{thm:pointwise_lipschitz} uncovers the relation between the propagated error $d(P_k,Q_k^*)$ and the prior error $d(P_{k-1},Q_{k-1})$. As illustrated in Figure~\ref{fig:break_down_approx_sequential_learning}, this result can be combined with the triangle inequality $d(P_i, Q_i) \leq d(P_i, Q_i^*) + d(Q_i^*, Q_i)$ and applied recursively from $i = k$ down to $i = 1$ to derive upper bounds on the overall learning error $d(P_k, Q_k)$. 

Let $\Y_{m:n}$ denote the data received from step $m$ to step $n$, i.e., $(y_m,y_{m+1},\cdots,y_n)$. Theorem~\ref{thm:learning_error} presents two sets of upper bounds on the overall learning error $d(P_k,Q_k)$, with each set containing one bound under the TV distance, one under the Hellinger distance, and one under the $1$-Wasserstein distance. The first set of bounds is provided in Equation~\eqref{ieq:learning_error_bound}, while the second set is given in Equation~\eqref{ieq:learning_error_bound2}. 
\begin{theorem}\label{thm:learning_error}
Given a sequence of data $\Y_{1:k}$, suppose the assumptions in Theorem~\ref{thm:pointwise_lipschitz} hold for all steps up to step $k$. Assume $P_{i} \in \bar{\mathcal{P}}_{i+1}(\bar{\X})$ for all $i \in [0,k-1]$ and $Q_{i} \in \bar{\mathcal{P}}_{i+1}(\bar{\X})$ for all $i \in [1,k-1]$.
For any step $k\geq2$, the distance between the true posterior $P_k$ and the approximate posterior $Q_k$ is bounded by
\begin{equation}
d(P_k,Q_k) \leq \sum_{j=1}^{k-1}C(\Y_{j+1:k},p(\Y_{j+1:k} \mid \Y_{1:j}),j)d(Q_j^*,Q_j) + d(Q_k^*,Q_k),
\label{ieq:learning_error_bound}
\end{equation}
and 
\begin{equation}
d(P_k,Q_k) \leq \sum_{j=1}^{k-1}C^\prime(\Y_{j+1:k},Q_{j:k-1},j)d(Q_j^*,Q_j) + d(Q_k^*,Q_k),
\label{ieq:learning_error_bound2}
\end{equation}
where $Q_{j:k-1}$ represents $(Q_{j},Q_{j+1},\cdots,Q_{k-1})$.
The definitions of the real-valued functions $C$ and $C^\prime$ are provided in Table~\ref{table:C} and Table~\ref{table:C'}, respectively.
\begin{table}
\centering
\caption{Definition of function $C(\Y_{j+1:k},p(\Y_{j+1:k} \mid \Y_{1:j}),j)$}
\begin{tabularx}{1.0\linewidth}
{ 
  | >{\hsize=.2\hsize\centering\arraybackslash}X 
  | >{\hsize=1.3\hsize\centering\arraybackslash}X 
  | >{\hsize=1.1\hsize\centering\arraybackslash}X |>{\hsize=1.4\hsize\centering\arraybackslash}X | }
 \hline
$d$ / Pro-blem & Inverse problems & State estimation & Parameter-state estimation \\
 \hline
\rule{0pt}{18pt}$d_{TV}$  & $\frac{\prod_{i=j+1}^kC_h(y_i)}{p(\Y_{j+1:k}\mid \Y_{1:j})}$ & $\frac{\prod_{i=j+1}^kC_{Th}(y_i;i)}{p(\Y_{j+1:k}\mid \Y_{1:j})}$ & $\frac{\prod_{i=j+1}^k\tilde{C}_{Th}(y_i;i)}{p(\Y_{j+1:k}\mid \Y_{1:j})}$ \\
 \hline
\rule{0pt}{18pt}$d_H$ & $\frac{2^{k-j}\prod_{i=j+1}^k\sqrt{C_h(y_i)}}{\sqrt{p(\Y_{j+1:k}\mid \Y_{1:j})}}$ & $\frac{2^{k-j}\prod_{i=j+1}^k\sqrt{C_{Th}(y_i;i)}}{\sqrt{p(\Y_{j+1:k}\mid \Y_{1:j})}}$ & $\frac{2^{k-j}\prod_{i=j+1}^k\sqrt{\tilde{C}_{Th}(y_i;i)}}{\sqrt{p(\Y_{j+1:k}\mid \Y_{1:j})}}$  \\
 \hline
 \rule{0pt}{18pt}$W_1$ & $\frac{\prod_{i=j+1}^k(2D\lVert h \rVert_{\text{Lip}}(y_i)+C_h(y_i))}{p(\Y_{j+1:k}\mid \Y_{1:j})}$ & $\frac{(2D)^{k-j}\prod_{i=j+1}^k C^*_{Th}(y_i;i)}{p(\Y_{j+1:k}\mid \Y_{1:j})}$ & $\frac{\prod_{i=j+1}^k(2D\tilde{C}^*_{Th}(y_i;i)+\tilde{C}_{Th}(y_i;i))}{p(\Y_{j+1:k} \mid \Y_{1:j})}$ \\
\hline
\end{tabularx}
\label{table:C}
\end{table}

\begin{table}
\centering
\caption{Definition of function $C'(\Y_{j+1:k},Q_{j:k-1},j)$}
\begin{tabularx}{1.0\linewidth}{ 
  | >{\hsize=0.2\hsize\centering\arraybackslash}X 
  | >{\hsize=1.28\hsize\centering\arraybackslash}X 
  | >{\hsize=1.15\hsize\centering\arraybackslash}X
  | >{\hsize=1.37\hsize\centering\arraybackslash}X | }
 \hline
$d$ / Pro-blem & Inverse problems & State estimation & Parameter-state estimation \\
 \hline
\rule{0pt}{15pt}$d_{TV}$  & $\prod_{i=j+1}^k\frac{C_h(y_i)}{Z_i(Q_{i-1})}$ & $\prod_{i=j+1}^k\frac{C_{Th}(y_i;i)}{Z_i(Q_{i-1})}$ & $\prod_{i=j+1}^k\frac{\tilde{C}_{Th}(y_i;i)}{Z_i(Q_{i-1})}$ \\
 \hline
\rule{0pt}{18pt}$d_H$ & $2^{k-j}\prod_{i=j+1}^k\frac{\sqrt{C_h(y_i)}}{\sqrt{Z_i(Q_{i-1})}}$ & $2^{k-j}\prod_{i=j+1}^k\frac{\sqrt{C_{Th}(y_i;i)}}{\sqrt{Z_i(Q_{i-1})}}$ & $2^{k-j}\prod_{i=j+1}^k\frac{\sqrt{\tilde{C}_{Th}(y_i;i)}}{\sqrt{Z_i(Q_{i-1})}}$  \\
 \hline
\rule{0pt}{18pt}$W_1$ & $\prod_{i=j+1}^k\frac{2D\lVert h \rVert_{\text{Lip}}(y_i)+C_h(y_i)}{Z_i(Q_{i-1})}$ & $(2D)^{k-j}\prod_{i=j+1}^k \frac{C^*_{Th}(y_i;i)}{Z_i(Q_{i-1})}$ & $\prod_{i=j+1}^k\frac{2D\tilde{C}^*_{Th}(y_i;i)+\tilde{C}_{Th}(y_i;i)}{Z_i(Q_{i-1})}$ \\
\hline
\end{tabularx}
\label{table:C'}
\end{table}
\end{theorem}

\begin{proof}
See Appendix~\ref{sec:app:learning_error}.
\end{proof}

\begin{remark}
We note that, the bound provided in Equation~\eqref{ieq:learning_error_bound} is a linear function of the incremental approximation errors. This can significantly facilitate further stability analysis.
\end{remark}

\begin{remark}
For inverse problems and state estimation problems, Theorem~\ref{thm:learning_error} holds for approximate posteriors from arbitrary distribution families, without requiring them to be absolutely continuous with respect to the Lebesgue measure. This makes Theorem~\ref{thm:learning_error} applicable to a broad class of BSL methods, including those based on the Monte Carlo technique. 
\end{remark}

As shown in Equation~\eqref{ieq:learning_error_bound}, the first set of upper bounds depends only on the incremental approximation errors $d(Q_j^*,Q_j), j \leq k$ and the constant $C(\Y_{j+1:k},p(\Y_{j+1:k}\mid \Y_{1:j}),j)$. For a given problem, the probability density $p(\Y_{j+1:k}\mid \Y_{1:j})$ is fixed, and therefore the constant  $C(\Y_{j+1:k},p(\Y_{j+1:k}\mid \Y_{1:j}),j)$ is also fixed. These bounds thus demonstrate that, as long as the incremental approximation errors, $d(Q_j^*,Q_j), j \leq k$, remain bounded, the overall learning error $d(P_k,Q_k)$ is guaranteed to be bounded. Furthermore, if we are given another sequence of approximate posteriors $\tilde{Q}_j, j \leq k$, Equation~\eqref{ieq:learning_error_bound}, combined with the triangle inequality $d(Q_k,\tilde{Q}_k) \leq d(Q_k, P_k)+d(\tilde{Q}_k,P_k)$, shows that the distance between two approximate posteriors $Q_k$ and $\tilde{Q}_k$ is bounded, as long as both sequences of incremental approximation errors, $d(Q_j^*, Q_j)$ and $d(\tilde{Q}_j^*, \tilde{Q}_j), j \leq k$, are bounded. A detailed discussion of the stability of approximate BSL methods with respect to the incremental approximation process, which is demonstrated by Equation~\eqref{ieq:learning_error_bound}, is provided in Section~\ref{sec:discussion:stability}. However, the probability density $p(\Y_{j+1:k}\mid \Y_{1:j})$ in the first set of bounds is normally intractable. As a result, while these bounds provide a stability guarantee, they are generally not computable in practice. 

In contrast, the second set of upper bounds, given in Equation~\eqref{ieq:learning_error_bound2}, is independent of $p(\Y_{j+1:k}\mid \Y_{1:j})$. Computing these bounds only requires the approximate posteriors $Q_j,1\leq j \leq k-1$, which are the learning results, and the incremental approximate errors $d(Q_j^*,Q_j), j\leq k$. This makes the second set of bounds estimable in many real-world applications. Further discussion on the estimation of these bounds is provided in Section~\ref{sec:estimation_bound}.

In practical applications, it can sometimes be observed that the learning error $d(P_k,Q_k)$ decreases as the number of update steps $k$ increases. However, this phenomenon cannot be explained by the results presented so far in this paper. Although Theorem~\ref{thm:pointwise_lipschitz}, as discussed in Section~\ref{sec:related_work}, yields tighter upper bounds on $d(P_k,Q_k^*)$ than those obtainable from existing results, it can be easily proved that these bounds are still larger than $d(P_{k-1},Q_{k-1})$. Consequently, this leads to an increase in the derived upper bounds on $d(P_k,Q_k)$ as $k$ grows. In the following section, we will attempt to understand these observations by providing sufficient conditions under which the learning error decreases after data assimilation.

\subsection{Error reduction by data assimilation}\label{sec:er}
Understanding the phenomenon of learning error decay requires understanding what leads to $d(P_k,Q_k)\leq d(P_{k-1},Q_{k-1})$. Recall that $Q_k = \hat{F}_k(Q_{k-1})$. If the map $\hat{F}_k$ exhibits any specific structural properties, these properties can be used to help establish sufficient conditions ensuring learning error reduction. However, since our analysis targets general approximate methods for BSL, we do not impose any structural assumptions on $\hat{F}_k$. Instead, we leverage the structure of approximate BSL methods and establish sufficient conditions on the system, $P_{k-1}$, and $Q_{k-1}$, under which the inequality $d(P_k,Q_k^*) \leq d(P_{k-1},Q_{k-1})$ holds. If these conditions are met and the incremental approximation error satisfies $d(Q_k,Q_k^*) \leq d(P_{k-1},Q_{k-1}) - d(P_k,Q_k^*)$, then the learning error at step $k$ is smaller than that at step $k-1$.

To facilitate this analysis, we introduce a unified notation involving a variable $\bar{x}$, a system-dependent function $g$, and a space of reference measures, $\bar{\mathcal{M}}(\bar{\X})$, as summarized in Table~\ref{table:x_bar_g_M}.
\begin{table}
\centering
\caption{Definitions of variable $\bar{x}$, function $g$, and space $\bar{\mathcal{M}}(\bar{\X})$ in different problems. The symbol $\mathcal{M}(\X)$ denotes the set of all $\sigma$-finite measures defined on the space $(\X,\mathcal{B}(\X))$, and $\lambda$ is the Lebesgue measure defined on the space $(\X\times \W,\mathcal{B}(\X\times \W))$.}
\begin{tabularx}{1.0\linewidth}{ 
  | >{\hsize=1.5\hsize\centering\arraybackslash}X 
  | >{\hsize=.4\hsize\centering\arraybackslash}X
  | >{\hsize=1.7\hsize\centering\arraybackslash}X |>{\hsize=0.4\hsize\centering\arraybackslash}X | }
 \hline
\rule{0pt}{11pt} Problem & $\bar{x}$ & $g(\bar{x},y_k)$ & $\bar{\mathcal{M}}(\bar{\X})$ \\
 \hline
 Inverse problems  & $x$ & $h(y_k,x)$ & $\mathcal{M}(\X)$\\
 \hline
 State estimation & $x_{k-1}$ & $\int_{\X}h_k(y_k,x_k)T_k(x_k,x_{k-1})dx_k$ & $\mathcal{M}(\X)$ \\
 \hline
 Parameter-state estimation &$(x_{k-1},w)$ & $\int_{\X}h_k(y_k,x_k,w)T_k(x_k,x_{k-1},w)dx_k$ & $\{\lambda\}$ \\
\hline
\end{tabularx}
\label{table:x_bar_g_M}
\end{table}
We first present sufficient conditions that guarantee error reduction under the TV distance for all three learning problems considered in this paper.
\begin{theorem}\label{thm:er:all:tv}
Given the data $y_k$, assume that $P_{k-1} \in \bar{\mathcal{P}}_k(\bar{\X})$ and $Q_{k-1} \in \bar{\mathcal{P}}_k(\bar{\X})$. Suppose that there exists a measure $\nu \in \bar{\mathcal{M}}(\bar{\X})$ such that $P_{k-1} \ll \nu, Q_{k-1} \ll \nu$, and 
\begin{align*}
&\int_{\bar{\mathcal{X}}^*}g(y_k,\bar{x})\bigg \lvert \frac{dP_{k-1}}{d\nu}(\bar{x})-\frac{dQ_{k-1}}{d\nu}(\bar{x}) \bigg \rvert \nu(d\bar{x})  \nonumber \\
&\qquad \leq \int_{\bar{\mathcal{X}}^*}g(y_k,\bar{x})\nu(d\bar{x})\int_{\bar{\mathcal{X}}^*}\bigg \lvert \frac{dP_{k-1}}{d\nu}(\bar{x})-\frac{dQ_{k-1}}{d\nu}(\bar{x}) \bigg \rvert \nu(d\bar{x}),
\end{align*}
and 
\begin{equation*}
\int_{\bar{\mathcal{X}}^*}g(y_k,\bar{x})\nu(d\bar{x}) \leq  \int_{\bar{\X}}g(y_k,\bar{x})P_{k-1}(d\bar{x}) \vee \int_{\bar{\X}}g(y_k,\bar{x})Q_{k-1}(d\bar{x}),
\end{equation*}
where
\begin{equation*}
\bar{\mathcal{X}}^* := \begin{cases} \{\bar{x} \in \bar{\mathcal{X}}: \frac{dP_{k-1}}{d\nu}(\bar{x}) \geq \frac{dQ_{k-1}}{d\nu}(\bar{x})\}, \quad \text{if } Z_k(P_{k-1}) \geq Z_k(Q_{k-1}), \nonumber \\
\{\bar{x} \in \bar{\mathcal{X}}: \frac{dP_{k-1}}{d\nu}(\bar{x}) \leq \frac{dQ_{k-1}}{d\nu}(\bar{x})\}, \quad \text{otherwise}. \end{cases}
\end{equation*}
Then the distributions $P_k$ and $Q_k^*$ satisfy
\begin{equation*}
d_{TV}(P_k,Q_k^*) \leq d_{TV}(P_{k-1},Q_{k-1}).
\end{equation*}
\end{theorem}

\begin{proof}
See Appendix~\ref{sec:app:er:tv}.
\end{proof}

We next establish conditions under which error reduction occurs under the Hellinger distance. The first such condition is given below.
\begin{assumptionp}{ER.H.1}\label{assump:er:all:H:1}
Given the data $y_k$, there exists a measure $\nu \in \bar{\mathcal{M}}(\bar{\X})$ such that $P_{k-1} \ll \nu, Q_{k-1} \ll \nu$, and 
\begin{equation*}
\int_\mathcal{\bar{X}}g(y_k,\bar{x})\sqrt{\frac{dP_{k-1}}{d\nu}(\bar{x})\frac{dQ_{k-1}}{d\nu}(\bar{x})}\nu(d\bar{x}) \geq \int_{\bar{\X}}g(y_k,\bar{x})\nu(d\bar{x})\int_{\bar{\X}}\sqrt{\frac{dP_{k-1}}{d\nu}(\bar{x})\frac{dQ_{k-1}}{d\nu}(\bar{x})}\nu(d\bar{x}),
\end{equation*}
and
\begin{equation*}
\int_{\bar{\X}}g(y_k,\bar{x})\nu(d\bar{x}) \geq \sqrt{\int_{\bar{\X}}g(y_k, \bar{x})P_{k-1}(d\bar{x})\int_{\bar{\X}}g(y_k,\bar{x})Q_{k-1}(d\bar{x})}.
\end{equation*}
\end{assumptionp}

An alternative sufficient condition for Hellinger error reduction is provided in the following assumption.
\begin{assumptionp}{ER.H.2}\label{assump:er:all:H:2}
Given the data $y_k$, there exists a measure $\nu \in \bar{\mathcal{M}}(\bar{\X})$ such that $P_{k-1} \ll \nu, Q_{k-1} \ll \nu$, and 
\begin{align*}
&\int_{\bar{\X}}g(y_k,\bar{x})\left(\sqrt{\frac{dP_{k-1}}{d\nu}(\bar{x})}-\sqrt{\frac{dQ_{k-1}}{d\nu}(\bar{x})}\right)^2\nu(d\bar{x}) \\
&\qquad \leq \int_{\bar{\X}}g(y_k,\bar{x})\nu(d\bar{x})\int_{\bar{\X}}\left(\sqrt{\frac{dP_{k-1}}{d\nu}(\bar{x})}-\sqrt{\frac{dQ_{k-1}}{d\nu}(\bar{x})}\right)^2\nu(d\bar{x}),
\end{align*}
and 
\begin{equation*}
\int_{\bar{\X}}g(y_k,\bar{x})\nu(d\bar{x})  \leq \sqrt{\int_{\bar{\X}}g(y_k,\bar{x})P_{k-1}(d\bar{x})\int_{\bar{\X}}g(y_k,\bar{x})Q_{k-1}(d\bar{x})}.
\end{equation*}
\end{assumptionp}

\begin{theorem}\label{thm:er:all:H}
Assume that $P_{k-1} \in \bar{\mathcal{P}}_k(\bar{\X})$ and $Q_{k-1} \in \bar{\mathcal{P}}_k(\bar{\X})$. Suppose \textbf{either} Assumption~\ref{assump:er:all:H:1} or Assumption~\ref{assump:er:all:H:2} holds. Then the distributions $P_k$ and $Q_k^*$ satisfy
\begin{equation*}
d_H(P_k,Q_k^*) \leq d_H(P_{k-1},Q_{k-1}).
\end{equation*}
\end{theorem}

\begin{proof}
See Appendix~\ref{sec:app:er:H}.
\end{proof}

\begin{remark}
\label{remark:interpret_tv_H}
Theorems~\ref{thm:er:all:tv} and~\ref{thm:er:all:H} characterize a scenario where we may have $d(P_k,Q_k^*) \leq d(P_{k-1},Q_{k-1})$, with $d$ being the TV or Hellinger distance.

This scenario arises when $P_{k-1}$ and $Q_{k-1}$ are relatively similar where the function $g$ takes large values, and both assign high average probability to those regions. Additionally, $P_{k-1}$ and $Q_{k-1}$ should differ more significantly in regions where $g$ is small but both assign, on average, low probability to such regions. It should be noted that both $P_{k-1}$ and $Q_{k-1}$ assigning low \textit{average} probability to low-$g$ regions does not preclude them from being substantially different in those regions. 
\end{remark}

Let the symbol "$\otimes$" denote the product of two independent probability measures. Specifically, let $\nu_1$ and $\nu_2$ be two probability measures defined on measurable spaces $(\Z_1,\mathcal{B}(\Z_1))$ and $(\Z_2,\mathcal{B}(\Z_2))$, respectively. The probability measure $\nu_1\otimes\nu_2$ is defined as $$\nu_1\otimes\nu_2(A,B)=\nu_1(A)\nu_2(B), \quad \forall A \in \mathcal{B}(\Z_1), B \in \mathcal{B}(\Z_2).$$ We now present sufficient conditions for error reduction under the $1$-Wasserstein distance in inverse problems.

\begin{theorem}\label{thm:er:ip:W}
In inverse problems, given the data $y_k$, suppose that $P_{k-1}\in \bar{\mathcal{P}}_k(\X) \bigcap \mathcal{P}_1(\X)$ and $Q_{k-1}\in \bar{\mathcal{P}}_k(\X) \bigcap \mathcal{P}_1(\X)$. Assume the likelihood model $h$ and the distributions $P_{k-1}$ and $Q_{k-1}$ satisfy
\begin{align*}
&\frac{\E_{(X,X^\prime) \sim P_{k-1} \bigotimes Q_{k-1}}[d_{\X}(X,X^\prime)h(y_k,X)h(y_k,X^\prime)]}{\E_{X \sim P_{k-1}}[h(y_k,X)]\E_{X \sim Q_{k-1}}[h(y_k,X)]} \\
&\qquad \leq \sup_{x_0 \in \mathcal{X}}\bigg \lvert \E_{X\sim P_{k-1}}[d_{\X}(X,x_0)]-\E_{X\sim Q_{k-1}}[d_{\X}(X,x_0)]\bigg \rvert.
\end{align*}
Then the distributions $P_k$ and $Q_k^*$ satisfy $W_1(P_k,Q_k^*) \leq  W_1(P_{k-1},Q_{k-1})$.
\end{theorem}

\begin{proof}
See Appendix~\ref{sec:app:er:ip:W}.
\end{proof}

Finally, we outline the sufficient conditions for error reduction under the $1$-Wasserstein distance in state estimation and parameter-state estimation problems.

\begin{theorem}\label{thm:er:se&ps:W}
Given the data $y_k$, assume that the distributions $P_{k-1}$ and $Q_{k-1}$ belong to the space $\bar{\mathcal{P}}_k(\bar{\X}) \bigcap \mathcal{P}_1(\bar{\X})$. Suppose the state transition model $T_k$, the observation model $h_k$, and the distributions $P_{k-1}$ and $Q_{k-1}$ satisfy
\begin{align*}
&\E_{(\bar{X},\bar{X}^\prime)\sim P_k^- \bigotimes Q_k^{-}}[d_{\bar{\X}}(\bar{X},\bar{X}^\prime)h_k(y_k,\bar{X})h_k(y_k,\bar{X}^\prime)] \\
&\qquad \leq \left(\int_{\bar{\X}}h_k(y_k, \bar{x})d\bar{x}\right)^2\E_{(\bar{X},\bar{X}^\prime)\sim P_k^- \bigotimes Q_k^{-}}[d_{\bar{\X}}(\bar{X},\bar{X}^\prime)],
\end{align*}
and 
\begin{equation*}
\E_{(\bar{X},\bar{X}^\prime)\sim P_k^- \bigotimes Q_k^{-}}[d_{\bar{\X}}(\bar{X},\bar{X}^\prime)] \leq \sup_{\bar{x}_0 \in \bar{\X}}\bigg \lvert \E_{\bar{X}\sim P_{k-1}}[d_{\bar{\X}}(\bar{X},\bar{x}_0)] -\E_{\bar{X}\sim Q_{k-1}}[d_{\bar{\X}}(\bar{X},\bar{x}_0)]\bigg \rvert,
\end{equation*}
and 
\begin{equation*}
\left(\int_{\bar{\X}}h_k(y_k, \bar{x})d\bar{x}\right)^2 \leq \E_{\bar{X}\sim P_k^-}[h_k(y_k,\bar{X})]\E_{\bar{X}\sim Q_k^{-}}[h_k(y_k,\bar{X})],
\end{equation*}
where $P_k^-$ and $Q_k^{-}$ are probability measures with Lebesgue densities $p_k^-$ and $q_k^{-}$, respectively, defined as: 
\begin{itemize}
\item in state estimation problems, 
$$p_k^-(x):= \E_{X\sim P_{k-1}}[T_k(x,X)], \quad
q_k^{-}(x):= \E_{X\sim Q_{k-1}}[T_k(x,X)];$$ 
\item in parameter-state estimation problems, 
\begin{align*}
p_k^-(x,w) &:= \int_{\X}T_k(x,x_{k-1},w)p_{k-1}(x_{k-1},w)dx_{k-1}, \\
q_k^{-}(x,w) &:= \int_{\X}T_k(x,x_{k-1},w)q_{k-1}(x_{k-1},w)dx_{k-1},
\end{align*}
where $p_{k-1}$ and $q_{k-1}$ are the Lebesgue densities of $P_{k-1}$ and $Q_{k-1}$, respectively.
\end{itemize}
Then, the distributions $P_k$ and $Q_k^*$ satisfy $$W_1(P_k,Q_k^*) \leq W_1(P_{k-1},Q_{k-1}).$$
\end{theorem}

\begin{proof}
See Appendix~\ref{sec:app:er:se&ps:W}.
\end{proof}

\begin{remark}
The results presented in Theorems~\ref{thm:er:all:tv},~\ref{thm:er:all:H},~\ref{thm:er:ip:W}, and ~\ref{thm:er:se&ps:W} apply to any pair of distributions $P_{k-1}$ and $Q_{k-1}$ that satisfy the sufficient conditions specified therein.
\end{remark}

\section{Discussion}
\label{sec:discussion}
In this section, we provide discussions of the main theorems. These discussions are carried out under the assumption that the conditions of Theorems~\ref{thm:pointwise_lipschitz} and~\ref{thm:learning_error} are satisfied. 

\subsection{Stability of approximate approaches in BSL with respect to the incremental approximation process}
\label{sec:discussion:stability}
In this section, we discuss the stability of approximate approaches for BSL with respect to the incremental approximation process, based on Theorems~\ref{thm:learning_error} and~\ref{thm:pointwise_lipschitz}.
\subsubsection{Bounded input, bounded output stability of learning error evolution}
\label{section:bibo}
The evolution of the learning error $d(Q_k,P_k)$ over the step $k$ can be viewed as a dynamical system with input $d(Q_k,Q_k^*)$. According to Theorems~\ref{thm:pointwise_lipschitz} and~\ref{thm:learning_error}, $d(Q_k,P_k)$ is upper bounded by 
\begin{align*}
d(Q_k,P_k) &\leq \tilde{C}(y_k,p(y_k \mid \Y_{1:k-1}),k)d(Q_{k-1},P_{k-1})+d(Q_k,Q_k^*) \\
&\leq \sum_{i=1}^{k-1}C(\Y_{i+1:k}, p(\Y_{i+1:k}\mid \Y_{1:i}),i)d(Q_i,Q_i^*)+d(Q_k,Q_k^*),
\end{align*}
where $\tilde{C}(y_k,p(y_k \mid \Y_{1:k-1}),k)$ is a constant given $y_k$, $p(y_k \mid \Y_{1:k-1})$, and $k$. Function $C$ is defined in Table~\ref{table:C}. This implies that the learning error system exhibits bounded input, bounded output (BIBO) stability: starting from a distribution $Q_{k_0}$ at an arbitrary step $k_0<k$, if the input $d(Q_i,Q_i^*)$ remains upper bounded over steps $i \in (k_0,k]$, then the output $d(Q_k,P_k)$ is upper bounded.

\subsubsection{Distance between different outputs of an approximate BSL approach}
Consider two different sequences of approximate posteriors $Q_i$ and $\tilde{Q}_i, i \leq k$. By Theorem~\ref{thm:learning_error} and the triangle inequality satisfied by the distance $d$, the distance between $Q_k$ and $\tilde{Q}_k$ is bounded by:
\begin{align}
&d(Q_k, \tilde{Q}_k) \leq d(Q_k,P_k)+d(\tilde{Q}_k,P_k) \nonumber \\
&\qquad \leq \sum_{i=1}^{k-1}C(\Y_{i+1:k}, p(\Y_{i+1:k}\mid \Y_{1:i}),i)\left(d(Q_i,Q_i^*)+d(\tilde{Q}_i,\tilde{Q}_i^*)\right)+d(Q_k,Q_k^*)+d(\tilde{Q}_k,\tilde{Q}_k^*). 
\label{ieq:stability:different_output}
\end{align}
Given the data $\Y_{1:k}$, the coefficients $C(\Y_{i+1:k}, p(\Y_{i+1:k}\mid \Y_{1:i}),i)$ are fixed for each $i \leq k-1$. Inequality~\eqref{ieq:stability:different_output} thus indicates that, as long as both sequences of incremental approximation errors, $d(Q_i,Q_i^*)$ and $d(\tilde{Q}_i,\tilde{Q}_i^*), i \leq k$ are upper bounded, the distance between the approximate posteriors $Q_k$ and $\tilde{Q}_k$---representing two different outputs of an approximate BSL method---remains bounded.

\subsection{Estimation of learning error bounds}
\label{sec:estimation_bound}
As shown in Equation~\eqref{ieq:learning_error_bound2}, estimating the learning error bound $d(P_k,Q_k)$ requires evaluating the evidence terms $Z_i(Q_{i-1})$ for $i=2,3,\cdots, k$, where $Q_{i-1}$ are the outputs of the learning methods. Since each $Z_i(Q_{i-1})$ involves only a single data point $y_i$, computing these terms is generally much easier than evaluating the offline learning evidence, which involves the entire batch of data. 

In practice, $Z_i(Q_{i-1})$ can often be estimated, for instance, using Monte Carlo methods with samples from $Q_{i-1}$. This is particularly convenient for methods that represent the approximate posteriors with empirical distributions. While the ability to estimate $Z_i(Q_{i-1})$ implies that one could, in principle, construct $Q_i^*$, practical BSL methods do not typically do so. Setting $Q_i=Q_i^*$ would lead to increasingly complex representations of $Q_i$ over time, significantly increasing computational costs and reducing inference speed-both of which are critical considerations in BSL applications. For example, one may want to retain a Gaussian form for $Q_i$ throughout the learning process. 

In addition to evidence terms, estimating the learning error bounds under the TV and Hellinger distances requires evaluating $C_h(y_i)$ for inverse problems, $C_{Th}(y_i;i)$ for state estimation problems, and $\tilde{C}_{Th}(y_i;i)$ for parameter-state estimation problems.

Given the data $y_i$, it is usually straightforward to compute $C_h(y_i)$, which is the upper bound of the likelihood. Although $C_{Th}(y_i;i)$ and $\tilde{C}_{Th}(y_i;i)$ may be more difficult to compute directly, they are upper bounded by their respective observation model bounds. These bounds can therefore be used as conservative estimates for $C_{Th}(y_i;i)$ and $\tilde{C}_{Th}(y_i;i)$, although this may result in looser learning error bounds. 

In summary, the learning error bounds under the TV and Hellinger distances are estimable in many practical settings. However, estimating the learning error bound under the $1$-Wasserstein distance using Equation~\eqref{ieq:learning_error_bound2} is generally more challenging, as it requires computing $\lVert h\rVert_{\text{Lip}}(y_i)$, $C^*_{Th}(y_i;i)$, and $\tilde{C}^*_{Th}(y_i;i)$, with the latter two being particularly difficult to evaluate in practice. Nevertheless, the inequality $W_1(P_k,Q_k) \leq Dd_{TV}(P_k,Q_k)$ (\cite{lipschitz_stability_ip,choose_metric}), where $D$ is defined in Assumption~\ref{assump:bounded_space}, provides an alternative way to upper bound $1$-Wasserstein learning error using the estimated TV bound.

Finally, it is worth noticing that the bound in Equation~\eqref{ieq:learning_error_bound2} can be computed recursively, which is especially suitable for sequential learning applications.
\subsection{Inaccurate initial prior} 
\label{section:inaccurate_prior}
Based on Theorem~\ref{thm:pointwise_lipschitz}, Theorem~\ref{thm:learning_error} can be easily extended to the scenarios in which the true initial prior $P_0$ is unknown, and an estimated initial prior $P_0^\prime$ is used for inference. In this case, if we have $P_0^\prime \in \bar{\mathcal{P}}_1(\bar{\X})$, the learning error is bounded by
\begin{align*}
d(Q_k,P_k) \leq& \sum_{i=1}^{k-1}C(\Y_{i+1:k}, p(\Y_{i+1:k}\mid \Y_{1:i}),i)d(Q_i,Q_i^*)+d(Q_k,Q_k^*) \\
&\qquad + C(\Y_{2:k},p(\Y_{2:k} \mid y_1),1)\hat{C}(y_1,p(y_1))d(P_0^\prime,P_0),
\end{align*}
where $Q_1^*=F_1(P_0^\prime)$ and $Q_i^*=F_i(Q_{i-1})$ for $i\geq 2$. Here $\hat{C}(y_1,p(y_1))$ is a constant depending on data $y_1$ and probability density $p(y_1)$. This upper bound grows linearly with the initial prior error $d(P_0^\prime,P_0)$. Assuming the incremental approximation errors $d(Q_i,Q_i^*), i\leq k$ are bounded, the learning error remains bounded as long as the initial prior error $d(P_0^\prime,P_0)$ is bounded. 

\subsection{Choice of initial prior and approximate posteriors}
\label{section:choice}
In this section, we combine Theorems~\ref{thm:pointwise_lipschitz} and~\ref{thm:learning_error} with existing results on long-term posterior behavior to provide insights into the choice of initial prior and approximate posteriors in state estimation problems. 

Let $F_{m:n}$ denote the composition of the prior-to-posterior maps from step $m$ to $n$:
\begin{equation*}
F_{m:n} := F_n\circ F_{n-1} \circ \cdots \circ F_{m}, \qquad m<n,
\end{equation*}
where "$\circ$" denotes function composition. For instance, $F_n\circ F_{n-1}(\mu) := F_n(F_{n-1}(\mu))$ for a probability measure $\mu$. Using this notation, the true posterior at step $k$ can be written as $P_k=F_{1:k}(P_0)$ and $P_k=F_{j+1:k}(P_j)$ for any $j<k$. According to \cite{stability_longtime_Hilbert_metric}, given the data $\Y_{j:k}$, if the product of state transition and observation models, $T_i(x_i,x_{i-1})h_i(y_i,x_i)$, is bounded above and below by positive constants for all $i \in [j,k]$, then 
 $d_{TV}(F_{j:k}(\mu),F_{j:k}(\mu^\prime))$ can become arbitrarily small for any $\mu$ and $\mu^\prime$, when $k-j$ is sufficiently large. 
\subsubsection{Choice of initial prior}
Let $P_0^\prime$ denote an alternative initial prior. Suppose we use $P_0^\prime$ as the true initial prior and approximate the corresponding posterior $P_k^\prime=F_{1:k}(P_0^\prime)$ with $Q_k$ at each time step $k$. Assume $P_i^\prime \in \bar{\mathcal{P}}_{i+1}(\bar{\X})$ for all $i \in [0,k-1]$. Then the learning error under the TV distance satisfies: 
\begin{align}
&d_{TV}(Q_k,P_k) \leq d_{TV}(Q_k,P_k^\prime)+d_{TV}(P_k^\prime,P_k) \nonumber \\
&\quad \leq \sum_{i=1}^{k-1}C(\Y_{i+1:k}, p^\prime(\Y_{i+1:k}\mid \Y_{1:i}),i)d_{TV}(Q_i,Q_i^*)+d_{TV}(Q_k,Q_k^*) + d_{TV}(P_k^\prime,P_k),
\label{ieq:different_prior}
\end{align}
where $Q_1^*=F_1(P_0^\prime)$ and $Q_i^*=F_i(Q_{i-1})$ for $i\geq2$. Here the probability density $p^\prime(\Y_{i+1:k}\mid \Y_{1:i})$ is computed under $P_0^\prime$. The constant $C$ is defined in Table~\ref{table:C} and is proportional to $1/p^\prime(\Y_{i+1:k}\mid \Y_{1:i})$. 

If $T_i(x_i,x_{i-1})h_i(y_i,x_i)$ is bounded below and above by positive constants for all $i \in [1,k]$, then as $k$ increases, $d_{TV}(P_k^\prime,P_k)$ will become negligible and the upper bound given in Equation~\eqref{ieq:different_prior} will be dominated by the first two terms.

Comparing Equations~\eqref{ieq:different_prior} and~\eqref{ieq:learning_error_bound}, we see that the bound in Equation~\eqref{ieq:different_prior} can be significantly smaller if $P_0^\prime$ leads to either lower incremental approximation errors or higher probability densities $p^\prime(\Y_{i+1:k}\mid \Y_{1:i})$. For instance, smaller incremental approximation errors may arise when $P_0^\prime$ yields a simpler optimization problem in practice.
This highlights the following insight: {\it when the product of the state transition and observation models is bounded above and below, using an alternative, well-chosen initial prior $P_0^\prime$, even if the true initial prior $P_0$ is available, may improve learning accuracy.} 

\subsubsection{Choice of approximate posterior $Q$}
By applying Theorem~\ref{thm:pointwise_lipschitz} and triangle inequality iteratively, we obtain the following upper bound on the learning error: 
\begin{align}
&d_{TV}(Q_k,P_k) \leq d_{TV}(Q_k,F_{j+1:k}(Q_j))+d_{TV}(F_{j+1:k}(Q_j),P_k) \nonumber \\
&\leq\sum_{i=j+1}^{k-1}C(\Y_{i+1:k}, p_q(\Y_{i+1:k}\mid \Y_{1:i}),i)d_{TV}(Q_i,Q_i^*)+d_{TV}(Q_k,Q_k^*) + d_{TV}(F_{j+1:k}(Q_j),P_k),
\label{ieq:choice_approx_post}
\end{align}
where $p_q(\Y_{i+1:k}\mid \Y_{1:i})$ is the probability density computed using $Q_j$ as the true posterior at step $j$. Function $C$ is defined in Table~\ref{table:C}. Since $P_k=F_{j+1:k}(P_j)$, the bound in Equation~\eqref{ieq:choice_approx_post} will be gradually dominated by the first two terms as $k$ grows, provided that the product of the state transition and observation models is bounded above and below by positive constants.

When selecting $Q_j$ at an earlier step $j$, we may compare two candidates. One offers a smaller immediate incremental approximation error $d_{TV}(Q_j,Q_j^*)$. The other has a larger immediate incremental approximation error but may facilitate more accurate approximations at future steps, resulting in smaller incremental approximation errors $d_{TV}(Q_i,Q_i^*), i > j$. This can occur, for instance, if the candidate belongs to a simpler distribution family that is easier to propagate in subsequent steps. 
 
In such cases, inequality~\eqref{ieq:choice_approx_post} suggests that, in the case where the product of state transition and observation models is bounded above and below, the second choice may yield a lower learning error over time (note that the summation in the upper bound given in inequality~\eqref{ieq:choice_approx_post} starts from $j+1$). 

\section{Application example}\label{sec:application_example}
In this section, we consider a specific class of algorithms for BSL: online variational inference for joint state and parameter estimation. Online VI has emerged in recent years as one of the most popular and effective approaches for sequential inference. However, theoretical analysis of their learning error remains sparse. We leverage Theorem~\ref{thm:learning_error} to derive learning error bounds for these methods in specific system settings.

Online VI approaches for joint state and parameter estimation can be categorized into two types: 
\begin{itemize}
\item Type 1 methods (e.g., \cite{online_VI_linear,fbovi}) approximate the joint posterior of state $X_k$ and parameter $W$, solving the parameter-state estimation problem considered in this paper.
\item Type 2 methods (e.g., \cite{onlinevi_nonlinear,onlinevi_realtime}) estimate the posterior distribution of state $X_k$, and compute a point estimate $\hat{w}_k$ of the system parameter.
\end{itemize}
We refer to these as Type 1 and Type 2 methods, respectively.
\subsection{Type 1 online VI methods}
Consider the following discrete-time Markov system:
\begin{align}
\begin{split}
p(x_k \mid x_{0:k-1},w) &= p(x_k \mid x_{k-1},w) =T_k(x_{k},x_{k-1},w), \\
p(y_k \mid x_{0:k},\Y_{1:k-1},w) &= p(y_k \mid x_k,w) =p_{\mathcal{N}}(y_k \mid \Phi_k(x_k,w),\Gamma),
\end{split}
\label{sys:application:1}
\end{align}
where $y_k \in \reals^r$ and $p_{\mathcal{N}}(y_k \mid \Phi_k(x_k,w),\Gamma)$ denotes the value of the probability density function the normal distribution with mean $\Phi_k(x_k,w)$ and covariance $\Gamma$ evaluated at $y_k$. Assume the initial prior distribution $P_0(X_0,W)$ is absolutely continuous w.r.t. the Lebesgue measure. 

Standard Type 1 methods (e.g., \cite{online_VI_linear}) compute the approximate posterior $Q_k$ by maximizing the evidence lower bound (ELBO). Let $q_k$ and $q_{k-1}$ denote the pdfs of $Q_k$ and $Q_{k-1}$, respectively. The ELBO at step $k$ is given by:
\begin{equation*}
\mathcal{L}_k(Q_k) := \E_{q_k(x_k,w)}[\log \left(p_{\mathcal{N}}(y_k \mid \Phi_k(x_k,w),\Gamma)q_k^{*-}(x_k,w)\right)] - \E_{q_k(x_k,w)}[\log q_k(x_k,w)],
\end{equation*}
where $q_k^{*-}(x_k,w):=\int_{\X}T_k(x_k,x_{k-1},w)q_{k-1}(x_{k-1},w)dx_{k-1}$. The following corollary applies Theorem~\ref{thm:learning_error} to bound the learning error $d(P_k,Q_k)$ for standard Type 1 methods applied to system~\eqref{sys:application:1}.
\begin{corollary}\label{coro:example1}
Suppose the ELBO satisfies $\mathcal{L}_i(Q_i) \geq \epsilon_i$\footnote{Note that neither the ELBO $\mathcal{L}_i(Q_i)$ nor the threshold $\epsilon_i$ is necessarily positive.} for all $i \in [1,k]$. Assume $P_i \in \bar{\mathcal{P}}_{i+1}(\X\times\W)$ for all $i \in [0,k-1]$ and $Q_i \in \bar{\mathcal{P}}_{i+1}(\X\times\W)$ for all $i \in [1,k-1]$. Also assume that $Q_i$ is absolutely continuous with respect to the Lebesgue measure for all $i \in [1,k]$. Then the learning error $d(P_k, Q_k)$ is bounded as:
\begin{align}
d(P_k,Q_k) &\leq \sum_{j=1}^{k-1}C_{VI}(\Y_{j+1:k},p(\Y_{j+1:k} \mid \Y_{1:j}),j)\sqrt{-\frac{r}{2}\log(2\pi)-\frac{1}{2}\log(\det(\Gamma)) - \epsilon_j} \nonumber \\
&\qquad  + \alpha \sqrt{-\frac{r}{2}\log(2\pi)-\frac{1}{2}\log(\det(\Gamma)) - \epsilon_k},
\label{ieq:learning_error_bound1_vi}
\end{align}
and 
\begin{align}
d(P_k,Q_k) &\leq \sum_{j=1}^{k-1}C_{VI}^\prime(\Y_{j+1:k},Q_{j:k-1},j)\sqrt{-\frac{r}{2}\log(2\pi)-\frac{1}{2}\log(\det(\Gamma)) - \epsilon_j} \nonumber \\
&\qquad  + \alpha \sqrt{-\frac{r}{2}\log(2\pi)-\frac{1}{2}\log(\det(\Gamma)) - \epsilon_k},
\label{ieq:learning_error_bound2_vi}
\end{align}
for $d$ being the total variation or Hellinger distance.

If in addition the metric space $(\X\times \W, d_{\X\times\W})$ satisfies $$\sup_{(x,w),(x^\prime,w^\prime) \in \X\times\W}d_{\X\times \W}((x,w),(x^\prime,w^\prime))=D<\infty,$$ then inequalities~\eqref{ieq:learning_error_bound1_vi} and ~\eqref{ieq:learning_error_bound2_vi} also hold with $d$ being the $1$-Wasserstein distance.

 The definitions of functions $C_{VI}$, $C_{VI}^\prime$, and the constant $\alpha$ for each distance are summarized in Table~\ref{table:C_C'_VI}.
\begin{table}
\centering
\caption{Definitions of $C_{VI}$, $C_{VI}^\prime$, and $\alpha$ for different distances of probability measures.}
\begin{tabularx}{1.0\linewidth}
{ 
  | >{\hsize=0.5\hsize\centering\arraybackslash}X 
  | >{\hsize=1.2\hsize\centering\arraybackslash}X 
  | >{\hsize=1.2\hsize\centering\arraybackslash}X
  |>{\hsize=0.1\hsize\centering\arraybackslash}X | }
 \hline
Distance $d$ & $C_{VI}(\Y_{j+1:k},p(\Y_{j+1:k}\mid \Y_{1:j}),j)$ & $C_{VI}^\prime(\Y_{j+1:k},Q_{j:k-1},j)$ & $\alpha$ \\
 \hline
 \rule{0pt}{18pt}$d_{TV}$  & $\frac{(2\pi)^{-\frac{r(k-j)}{2}}\det(\Gamma)^{-\frac{k-j}{2}}}{\sqrt{2}p(\Y_{j+1:k}\mid \Y_{1:j})}$ & $\frac{(2\pi)^{-\frac{r(k-j)}{2}}\det(\Gamma)^{-\frac{k-j}{2}}}{\sqrt{2}\prod_{i=j+1}^k Z_i(Q_{i-1})}$ & $\frac{1}{\sqrt{2}}$\\
 \hline
\rule{0pt}{20pt}$d_H$  & $\frac{2^{(k-j)}(2\pi)^{-\frac{r(k-j)}{4}}\det(\Gamma)^{-\frac{k-j}{4}}}{\sqrt{2p(\Y_{j+1:k}\mid \Y_{1:j})}}$ & $\frac{2^{(k-j)}(2\pi)^{-\frac{r(k-j)}{4}}\det(\Gamma)^{-\frac{k-j}{4}}}{\sqrt{2}\prod_{i=j+1}^k\sqrt{Z_i(Q_{i-1})}}$   & $\frac{1}{\sqrt{2}}$ \\
 \hline
 \rule{0pt}{18pt} $W_1$  & $\frac{D(2\pi)^{-\frac{r(k-j)}{2}}\det(\Gamma)^{-\frac{k-j}{2}}}{\sqrt{2}p(\Y_{j+1:k}\mid \Y_{1:j})}$ & $\frac{D(2\pi)^{-\frac{r(k-j)}{2}}\det(\Gamma)^{-\frac{k-j}{2}}}{\sqrt{2}\prod_{i=j+1}^kZ_i(Q_{i-1})}$   & $\frac{D}{\sqrt{2}}$ \\
 \hline
\end{tabularx}
\label{table:C_C'_VI}
\end{table}
\end{corollary}
\begin{proof}
See Appendix~\ref{sec:app:coro1}.
\end{proof}

\subsection{Type 2 online VI methods}
Type 2 methods solve a different problem: they assume the existence of a true but unknown system parameter $\bar{w}$ and aim to estimate the true posterior $P_{k}$, defined as the distribution of the state $X_k$ given all the received data $\Y_{1:k}$ under the underlying system with parameter $\bar{w}$, and compute a point estimate $\hat{w}_k$ of the system parameter.

Some Type 2 methods, such as \cite{onlinevi_realtime}, assume that the observation model is known and only the state transition model is unknown. We adopt this setting and consider the following system:
\begin{align}
\begin{split}
p(x_k \mid x_{0:k-1},w) &= p(x_k \mid x_{k-1},w) =T_k(x_{k},x_{k-1},w), \\
p(y_k \mid x_{0:k},\Y_{1:k-1},w) &= p(y_k \mid x_k) = p_{\mathcal{N}}(y_k \mid \Phi_k(x_k),\Gamma),
\end{split}
\label{sys:application:2}
\end{align}
where $y_k \in \reals^r$. Assume that the initial prior distribution $P_0(X_0)$ is absolutely continuous w.r.t. the Lebesgue measure. Let $q_k$ denote the pdf of $Q_k$. Standard Type 2 methods obtain $Q_k$ and $\hat{w}_k$ by maximizing the following ELBO:
\begin{equation*}
\mathcal{L}_k(Q_k,\hat{w}_k) := \E_{q_k(x_k)}[\log \left(p_{\mathcal{N}}(y_k \mid \Phi_k(x_k),\Gamma)q_{k,\hat{w}_k}^{*-}(x_k)\right)] - \E_{q_k(x_k)}[\log q_k(x_k)],
\end{equation*}
where $q_{k,\hat{w}_k}^{*-}(x_k):= \int_{\X}T_k(x_k,x_{k-1},\hat{w}_k)Q_{k-1}(dx_{k-1})$.

By applying Theorem~\ref{thm:learning_error} to the state estimation setting considered here, we obtain Corollary~\ref{coro:example2}, which provides an upper bound on the state estimation error $d(P_k, Q_k)$.
\begin{assumptionp}{VI.1}\label{assump:vi}
At time step $k$, assume the state transition and observation model of  system~\eqref{sys:application:2} satisfies
\begin{equation*}
g_{\text{Lip}}(x;k) := \sup_{\forall x_{k-1}\in \X, w, w^\prime \in \W, w \neq w^\prime}\frac{\lvert T_k(x,x_{k-1},w)-T_k(x,x_{k-1},w^\prime)\rvert }{\lVert w - w^\prime \rVert_m} < \infty, \qquad \forall x \in \X,
\end{equation*}
where $\lVert \cdot \rVert_m$ is a metric on the parameter space $\W$, \\
and 
$$\tilde{C}_{VI}(y_k;k):=\int_{\X}p_{\mathcal{N}}(y_k \mid \Phi_k(x),\Gamma)g_{\text{Lip}}(x;k)dx < \infty.$$
\end{assumptionp}
\begin{corollary}\label{coro:example2}
Suppose Assumption~\ref{assump:vi} holds for all time steps up to $k$. Assume the ELBO satisfies $\mathcal{L}_i(Q_i,\hat{w}_i) \geq \epsilon_i$\footnote{Neither the ELBO $\mathcal{L}_i(Q_i,\hat{w}_i)$ nor $\epsilon_i$ is necessarily positive.} for all $i \in [1,k]$. Assume $P_i \in \bar{\mathcal{P}}_{i+1}(\X)$\footnote{The set $\bar{\mathcal{P}}_{i+1}(\X)$ corresponds to the system with the true parameter $\bar{w}$.} for all $i \in [0,k-1]$ and $Q_i \in \bar{\mathcal{P}}_{i+1}(\X)$ for all $i \in [1,k-1]$. Also assume that $Q_i$ is absolutely continuous with respect to the Lebesgue measure for all $i \in [1,k]$. Then the state estimation error satisfies
\begin{align}
&d(P_{k},Q_k) 
\leq \sum_{j=1}^{k-1}C_{VI}^\prime(\Y_{j+1:k},Q_{j:k-1},j)\left({\sqrt{-\frac{r}{2}\log(2\pi)-\frac{1}{2}\log(\det(\Gamma)) - \epsilon_j}}\right. \nonumber \\
&\qquad \left.{+\beta\left(Q_{j-1},y_j,\hat{w}_j,\lVert \hat{w}_j-\bar{w}\rVert_m,j\right)}\right) \nonumber \\
&\qquad \qquad + \alpha \left(\sqrt{-\frac{r}{2}\log(2\pi)-\frac{1}{2}\log(\det(\Gamma)) - \epsilon_k}+\beta\left(Q_{k-1},y_k,\hat{w}_k,\lVert \hat{w}_k-\bar{w}\rVert_m,k\right)\right),
\label{ieq:learning_error_bound_vi:type2}
\end{align}
with $d$ being the total variation or Hellinger distance.

Assuming further that the metric space $(\X,d_{\X})$ satisfies $\sup_{x,x^\prime \in \X}d_{\X}(x,x^\prime)=D<\infty$, then inequality~\eqref{ieq:learning_error_bound_vi:type2} holds with $d$ being the $1$-Wasserstein distance. The functions $C_{VI}^\prime$ and $\beta$ are defined in Tables~\ref{table:C_C'_VI} and~\ref{table:beta}, respectively.
\begin{table}
\centering
\caption{Definitions of function $\beta$ for different distances of probability measures. The function $Z_{k,\hat{w}_k}$is defined as $Z_{k,\hat{w}_k}(Q_{k-1}):= \int_{\X}p_{\mathcal{N}}(y_k \mid H_k(x),\Gamma)\left(\int_{\X}T_k(x,x_{k-1},\hat{w}_k)Q_{k-1}(dx_{k-1})\right)dx$.}
\begin{tabularx}{1.0\linewidth}
{ 
  | >{\hsize=1.2\hsize\centering\arraybackslash}X 
  | >{\hsize=0.9\hsize\centering\arraybackslash}X 
  |>{\hsize=0.9\hsize\centering\arraybackslash}X | }
 \hline
Function $\beta$ & $d_{TV}$, $W_1$ & $d_H$ \\
 \hline
 \rule{0pt}{17pt} $\beta(Q_{k-1},y_k,\hat{w}_k,\lVert \hat{w}_k-\bar{w}\rVert_m,k)$  & $\frac{\sqrt{2}\tilde{C}_{VI}(y_k;k)\lVert \hat{w}_k - \bar{w} \rVert_m }{Z_{k,\hat{w}_k}(Q_{k-1})}$ & $2\sqrt{\frac{\tilde{C}_{VI}(y_k;k)\lVert \hat{w}_k - \bar{w} \rVert_m}{Z_{k,\hat{w}_k}(Q_{k-1})}}$\\
 \hline
\end{tabularx}
\label{table:beta}
\end{table}
\end{corollary}

\begin{proof}
See Appendix~\ref{sec:app:coro2}.
\end{proof}

\section{Illustrative numerical example}\label{sec:illustrative_example}
In this section, we present a numerical example to validate the upper bounds established in Theorem~\ref{thm:pointwise_lipschitz} for inverse problems and to investigate the following questions:
\begin{itemize}
\item How tight are the derived bounds?
\item How does the tightness of the derived bound on the posterior distance $d(F_k(\mu),F_k(\mu^\prime))$ change when the priors $\mu$ and $\mu^\prime$ vary?
\item As the priors $\mu$ and $\mu^\prime$ vary, do $d(F_k(\mu),F_k(\mu^\prime))$ and the derived bound on it exhibit similar trends?
\end{itemize}

As will be shown, the derived bounds can be remarkably tight. To minimize the influence of computational and estimation errors on the validation of the theorem, we consider a simple one-dimensional inverse problem with a linear Gaussian likelihood model and Gaussian initial prior distributions. Under this setting, both the posterior distributions and the Hellinger distance between them admit closed-form expressions. The model is chosen to be one-dimensional to minimize the estimation error of the TV distance.

Consider the following model:
\begin{equation*}
Y = aX + \eta, \quad \eta \sim \mathcal{N}(0,3),
\end{equation*}
where $a=1.1$. We generate a single observation $Y=y$ and evaluate three cases with three different sets of initial distributions $\mu_0$ and $\mu^\prime_0$. The first two sets are $\mu_0 =\mathcal{N}(-10,5)$ and $\mu^\prime_0 =\mathcal{N}(8,5)$, and $\mu_0 =\mathcal{N}(0,1)$ and $\mu^\prime_0 =\mathcal{N}(2,1)$, representing cases where the initial distributions are significantly different and relatively close, respectively. The third set of initial distributions consists of Gaussian distributions whose means are uniformly sampled from the interval $[-10, 10]$ and variances uniformly sampled from the interval $[0, 5]$. In this case, the variances of the two initial distributions are not necessarily equal. We refer to the cases using the first, second, and third sets of initial distributions as Case 1, Case 2, and Case 3, respectively. We update the corresponding posteriors sequentially over 20 iterations. For each update step $k$, the updates follow $\mu_k = F_k(\mu_{k-1})$ and $\mu^\prime_k = F_k(\mu^\prime_{k-1})$. The same data $y$ is used at every iteration, i.e., $y_k = y$ for all $k$.

In this example, the assumptions in Theorem~\ref{thm:pointwise_lipschitz} for the TV and Hellinger distances are satisfied. The TV distance between $\mu_k$ and $\mu_k^\prime$, $d_{TV}(\mu_k,\mu_k^\prime)$, and the derived bound on it are shown in Figure~\ref{fig:tv}. The Hellinger distance between $\mu_k$ and $\mu_k^\prime$, $d_H(\mu_k,\mu_k^\prime)$, and its corresponding bound are shown in Figure~\ref{fig:H}.  
\begin{figure}
     \centering
     \begin{subfigure}[b]{0.31\linewidth}
         \centering
         \includegraphics[width=\linewidth]{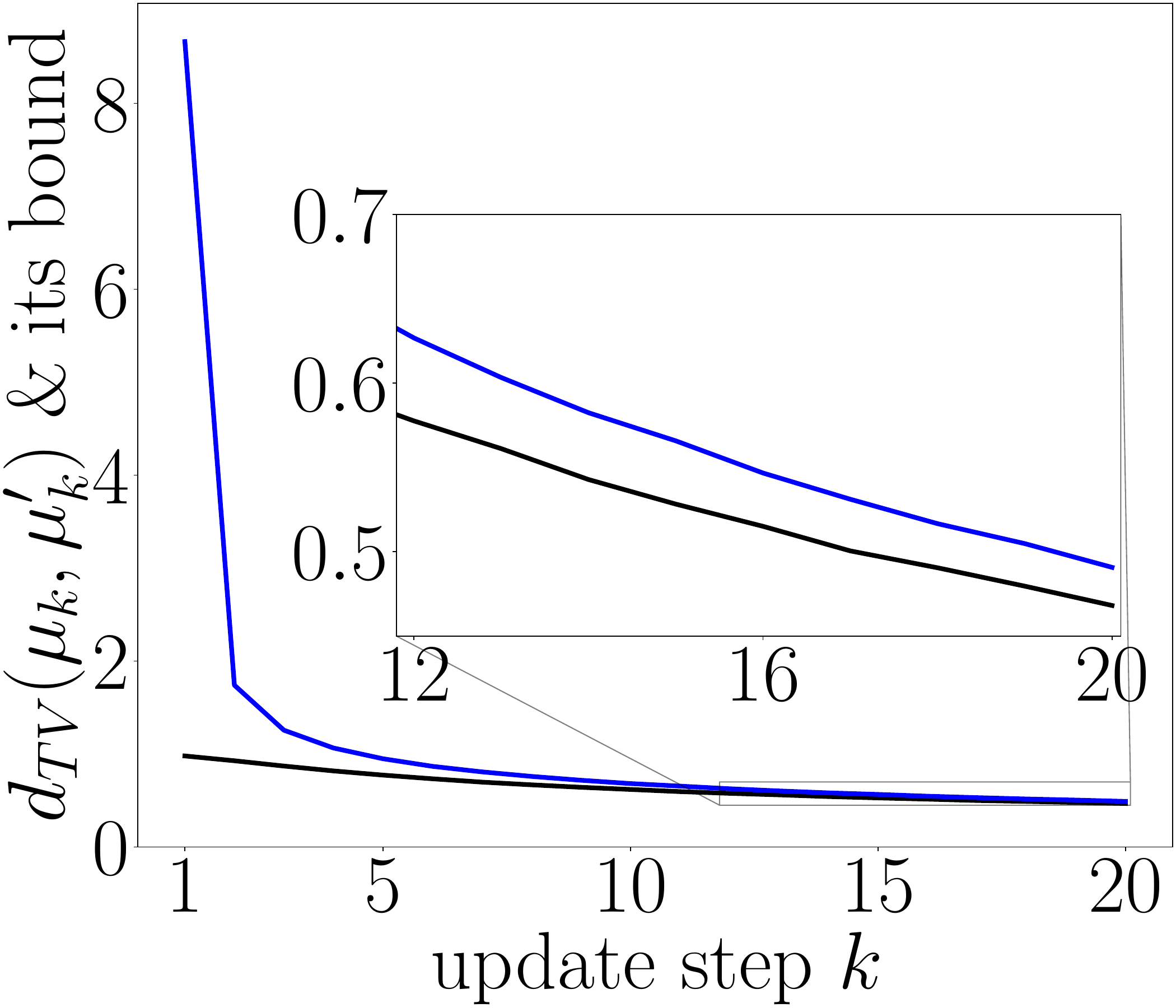}
         \caption{Case 1}
         \label{fig:case1_tv}
     \end{subfigure}
     \hfill
     \begin{subfigure}[b]{0.31\linewidth}
         \centering
         \includegraphics[width=\linewidth]{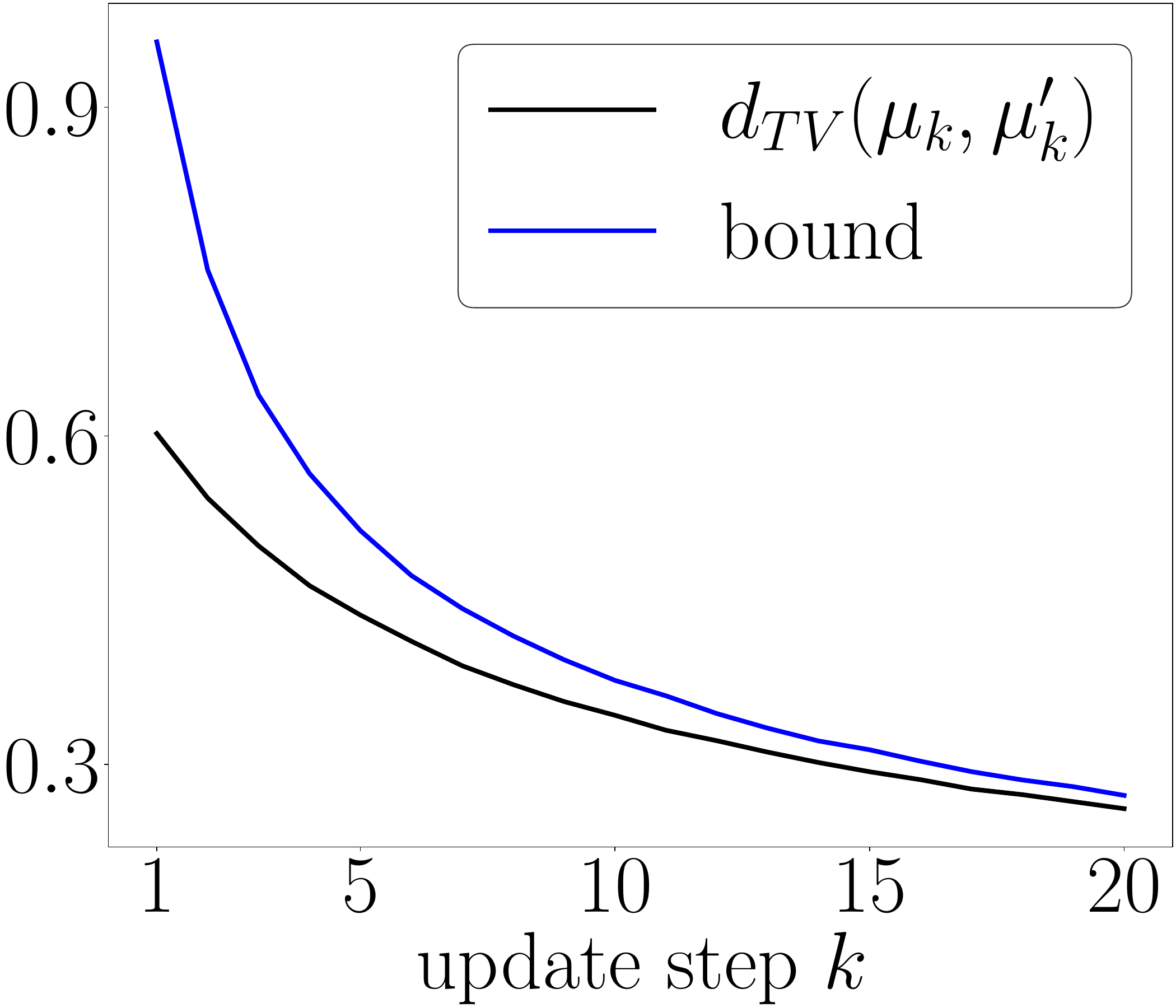}
         \caption{Case 2}
         \label{fig:case2_tv}
     \end{subfigure}
     \hfill
     \begin{subfigure}[b]{0.31\linewidth}
         \centering
         \includegraphics[width=\linewidth]{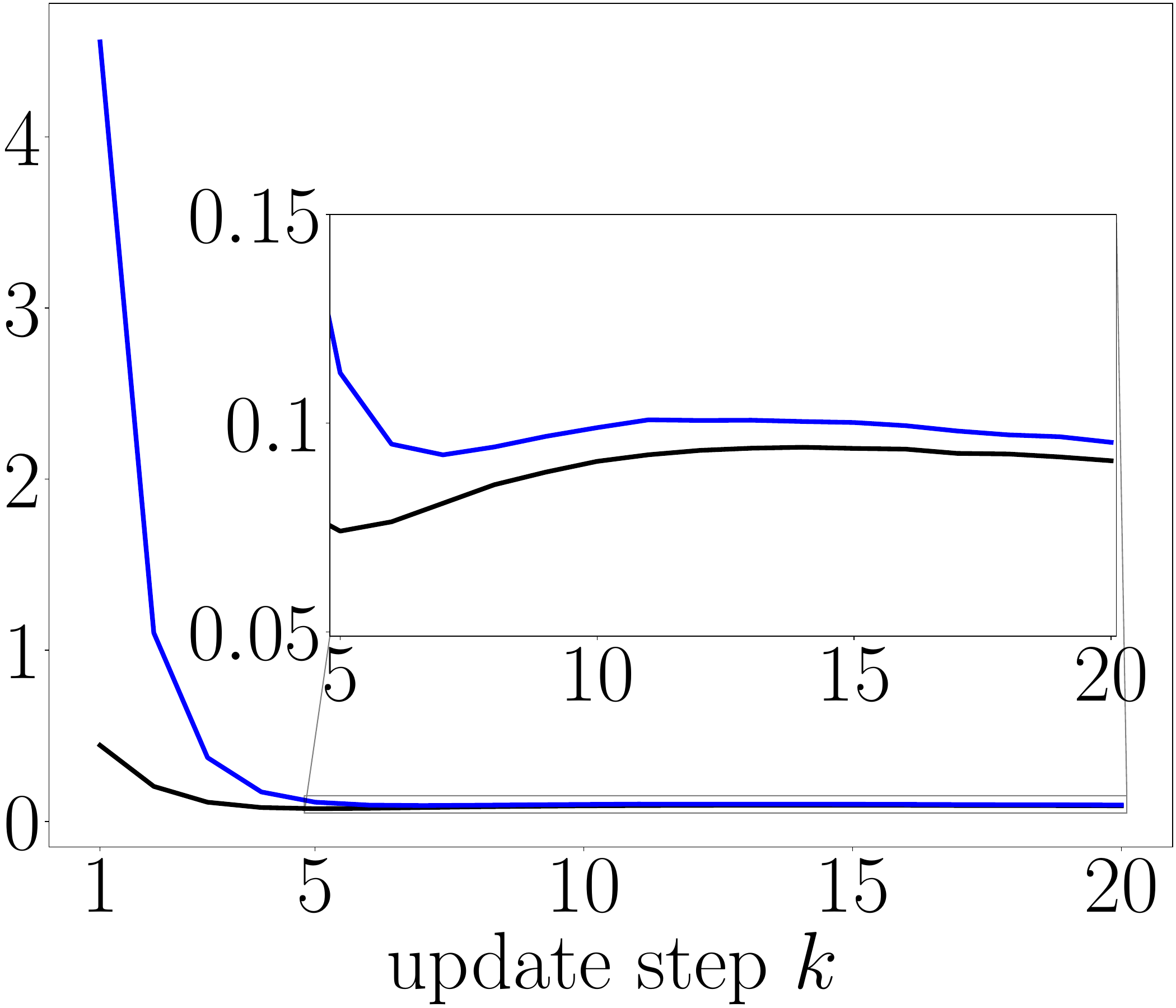}
         \caption{Case 3}
         \label{fig:case3_tv}
     \end{subfigure}
         \caption{TV distance between posteriors, $d_{TV}(\mu_k,\mu_k^\prime)$, and its upper bound provided by Theorem~\ref{thm:pointwise_lipschitz}. Because of the symmetry of the TV distance, Theorem~\ref{thm:pointwise_lipschitz} yields the bound $\frac{C_h(y)}{Z_k(\mu_{k-1})\vee Z_k(\mu^\prime_{k-1})}d_{TV}(\mu_{k-1},\mu_{k-1}^\prime)$. A zoom-in view of steps $12$ through $20$ is shown in panel (a), and a zoom-in view of steps $5$ through $20$ is shown in panel (c). The posterior distance is always below the derived bound. In Cases 1 and Case 2, when the posterior distance decreases, the derived bound also declines. In Case 3, the trends of changes in the posterior distance and the bound are similar most of the time. }
         \label{fig:tv}
\end{figure}
\begin{figure}
     \centering
     \begin{subfigure}[b]{0.31\linewidth}
         \centering
         \includegraphics[width=\linewidth]{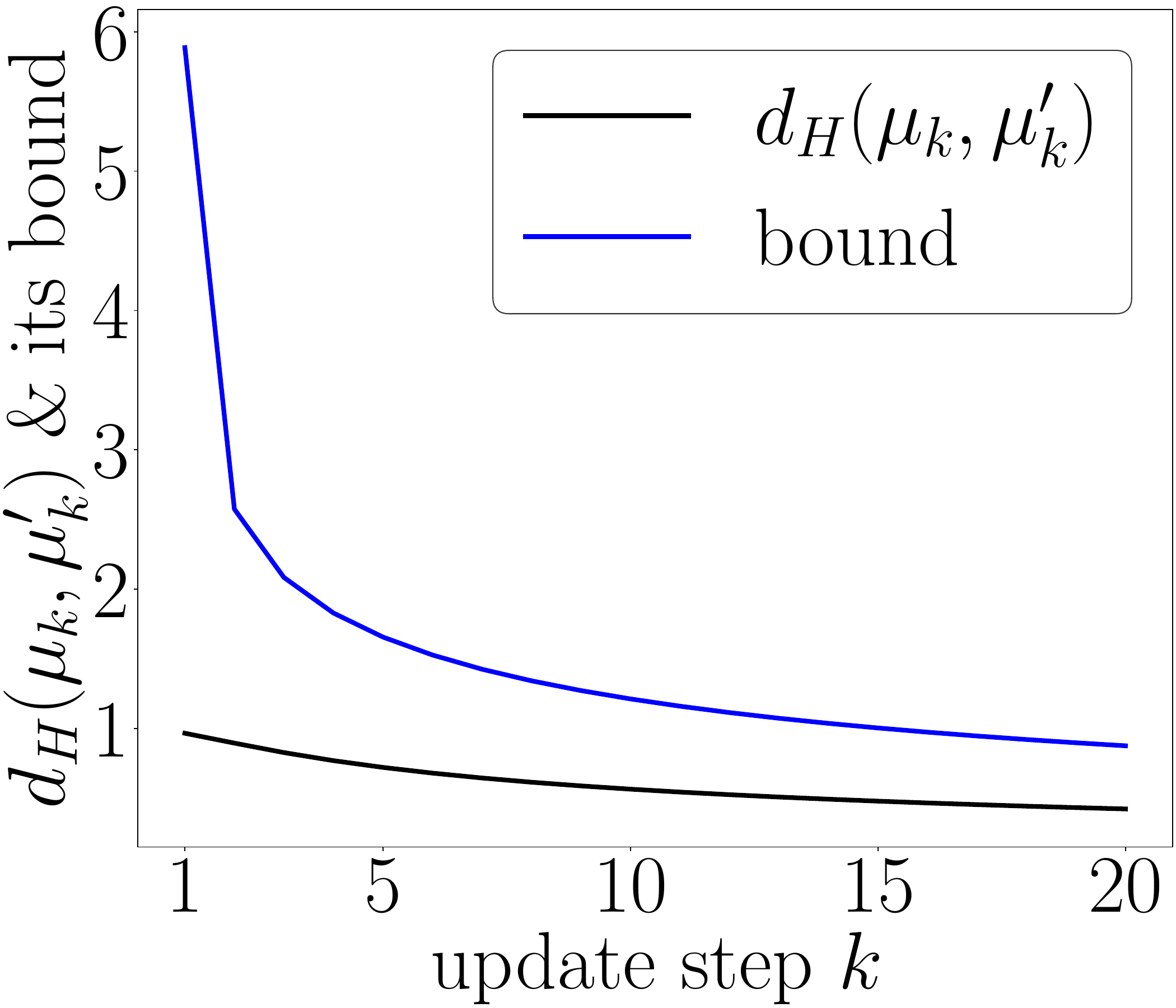}
         \caption{Case 1}
         \label{fig:case1_Hellinger}
     \end{subfigure}
     \hfill
     \begin{subfigure}[b]{0.31\linewidth}
         \centering
         \includegraphics[width=\linewidth]{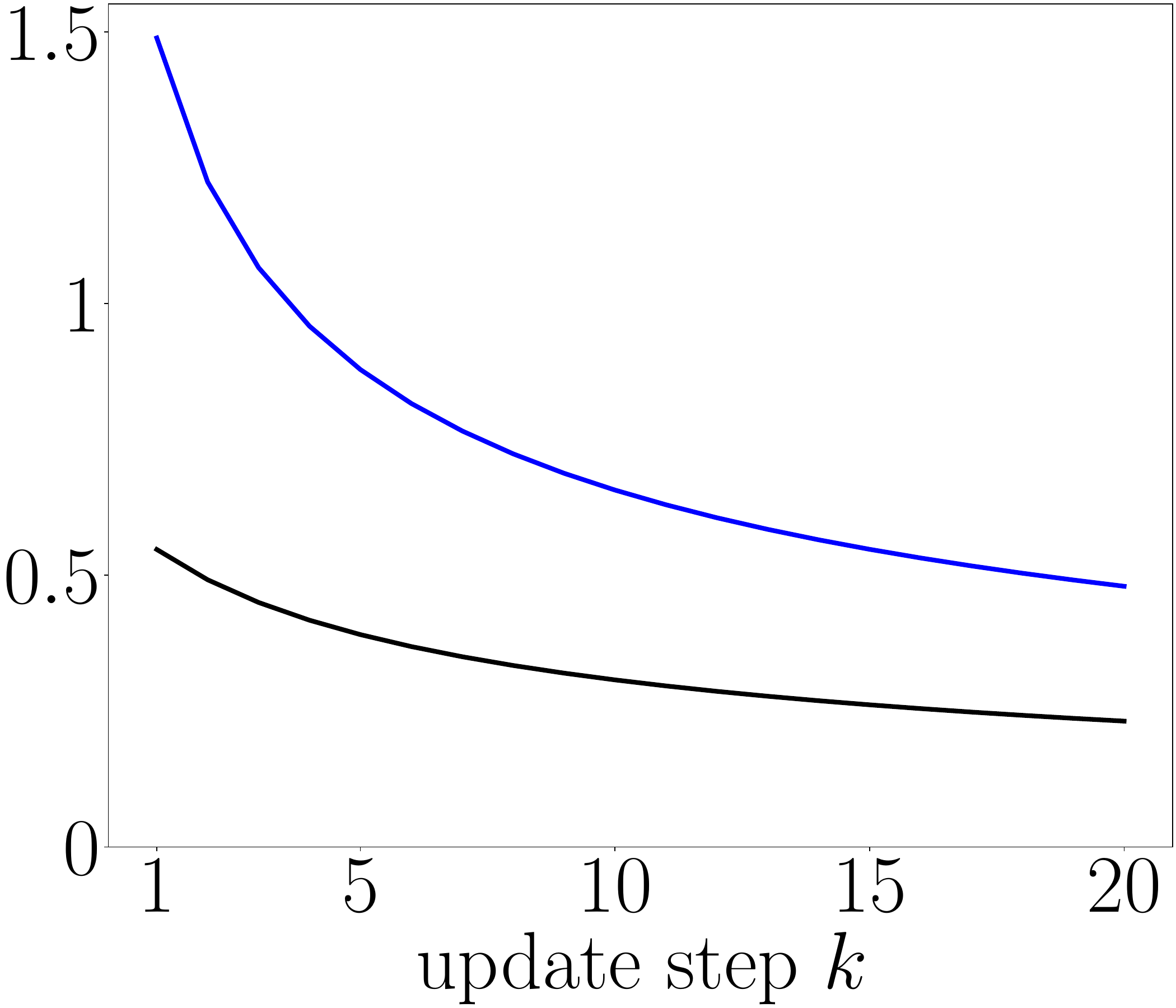}
         \caption{Case 2}
         \label{fig:case2_Hellinger}
     \end{subfigure}
     \hfill
     \begin{subfigure}[b]{0.31\linewidth}
         \centering
         \includegraphics[width=\linewidth]{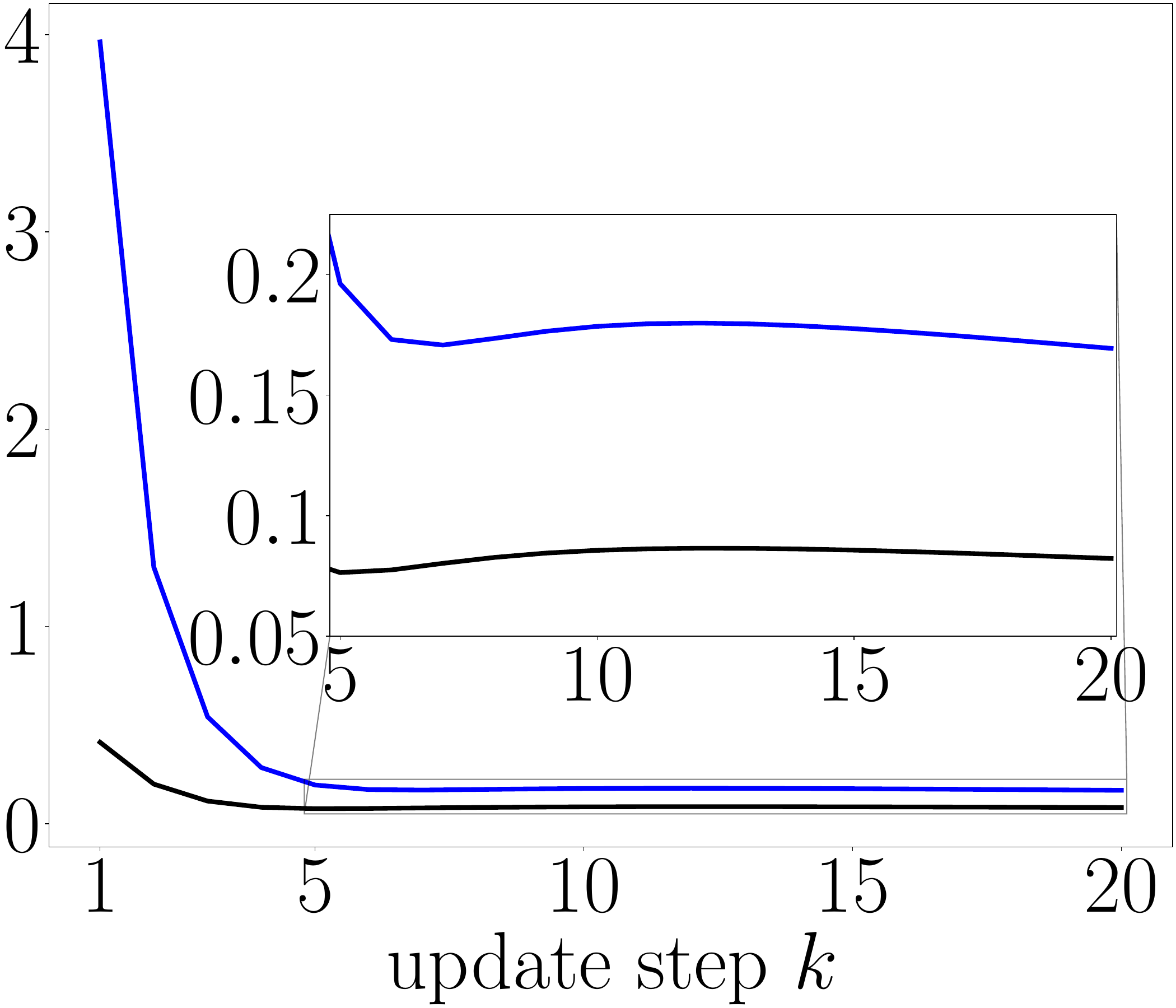}
         \caption{Case 3}
         \label{fig:case3_Hellinger}
     \end{subfigure}
         \caption{Hellinger distance between posteriors, $d_{H}(\mu_k,\mu_k^\prime)$, and its upper bound from Theorem~\ref{thm:pointwise_lipschitz}. Because the Hellinger distance is symmetric, Theorem~\ref{thm:pointwise_lipschitz} provides the bound $2\sqrt{\frac{C_h(y)}{Z_k(\mu_{k-1})\vee Z_k(\mu^\prime_{k-1})}}d_H(\mu_{k-1},\mu^\prime_{k-1})$. A zoom-in view of steps $5$ through $20$ is shown in panel (c). The posterior distance remains below the derived bound throughout. In Cases 1 and 2, both the posterior distance and its bound decrease over the update steps. In Case 3, the trends of the posterior distance and the bound are generally aligned.}
         \label{fig:H}
\end{figure}

Figures~\ref{fig:tv} and~\ref{fig:H} validate the results under the TV and Hellinger distances for inverse problems in Theorem~\ref{thm:pointwise_lipschitz}. Both the TV and Hellinger bounds become progressively tighter over update steps as more data are assimilated. Notably, the TV bound can be quite sharp. Furthermore, both the TV and Hellinger bounds closely track the actual posterior distances, as the trends of the bound and the posterior distance align most of the time. 

\section{Conclusions}
\label{sec:conclusion}
In this work, we have presented a rigorous theoretical analysis of general approximate approaches for Bayesian sequential learning (BSL) in the context of inverse problems, state estimation, and parameter-state estimation. By establishing pointwise global Lipschitz continuity of the prior-to-posterior map under the total variation, Hellinger, and $1$-Wasserstein distances, we derived two sets of upper bounds on the learning error associated with these methods. The derived bounds demonstrate the stability of these approaches with respect to the incremental approximation process and offer estimable guidance in real-world applications.

Our results mark the first establishment of global Lipschitz stability of the posterior with respect to the prior under the Hellinger and Wasserstein distances, and provide the first general error analysis framework for approximate BSL methods. Moreover, we identified sufficient conditions under which learning error decays, a phenomenon often observed empirically but not previously explored theoretically for general approximate methods in BSL.

To demonstrate the applicability of our framework, we apply it to two different online variational inference methods in the context of joint state and parameter learning.

Our theoretical results yield practical insights. In particular, for state estimation problems, under certain conditions, using a well-designed alternative initial prior—rather than the true one—can lead to improved learning accuracy. Moreover, under the same conditions, our results suggest that selecting an approximate posterior $Q_k$ with a larger immediate incremental approximation error $d(Q_k,Q_k^*)$ may be advantageous if it enables more accurate incremental approximations in subsequent steps, compared to a choice that minimizes the immediate incremental approximation error but leads to less accurate incremental approximations over time.

\acks{This work was supported by an NSF CAREER Award, grant number CMMI-2238913.}

\newpage

\appendix
\section{Comparison with the results in existing literature}
\subsection{Comparison with the results in \cite{lipschitz_stability_ip}}
\label{sec:app:compare1}
We compare the posterior distance bounds provided by Theorem~\ref{thm:pointwise_lipschitz} for inverse problems with those presented in \cite{lipschitz_stability_ip}. When the assumptions required by the results in \cite{lipschitz_stability_ip} are satisfied, the assumptions for Theorem~\ref{thm:pointwise_lipschitz} are also satisfied. Under these \textit{same} assumptions, the corresponding upper bounds are summarized in Table~\ref{table:related_work} below.
\begin{table}[H]
\centering
\caption{Upper bounds on the distance between posteriors, $d(F_k(\mu),F_k(\mu^\prime))$. The Lipschitz constant $\lVert h \rVert_{\text{Lip}}(y_k)$, metric space diameter $D$, and evidence function $Z_k$ are defined in Assumption IP.1, Assumption IP-SE-PS.4, and Definition 2.4, respectively. Under the assumptions in~\cite{lipschitz_stability_ip}, it can be readily shown that we have $Z_k(\mu) \vee Z_k(\mu^\prime) \leq 1$.}
\begin{tabularx}{0.9\linewidth}{ 
  | >{\hsize=.7\hsize\centering\arraybackslash}X 
  | >{\hsize=1.15\hsize\centering\arraybackslash}X 
  | >{\hsize=1.15\hsize\centering\arraybackslash}X | }
\hline
Distance & Ours & \cite{lipschitz_stability_ip} \\
 \hline
 Total variation  & $\bm{\frac{1}{Z_k(\mu) \vee Z_k(\mu^\prime)}d_{TV}(\mu,\mu^\prime)}$ & $\frac{2}{Z_k(\mu) \vee Z_k(\mu^\prime)}d_{TV}(\mu,\mu^\prime)$ \\
 \hline
 Hellinger & $\bm{2\sqrt{\frac{1}{Z_k(\mu) \vee Z_k(\mu^\prime)}}d_{H}(\mu,\mu^\prime)}$ & $\frac{2}{Z_k(\mu) \wedge Z_k(\mu^\prime)}d_{H}(\mu,\mu^\prime)$  \\
 \hline
 $1$-Wasserstein &$\frac{1+2D\lVert h \rVert _{\text{Lip}}(y_k)}{Z_k(\mu) \vee Z_k(\mu^\prime)}W_1(\mu,\mu^\prime)$ & $\frac{\left(1+D\lVert h \rVert _{\text{Lip}}(y_k)\right)^2}{\left(Z_k(\mu) \wedge Z_k(\mu^\prime)\right)^2}W_1(\mu,\mu^\prime)$, $\tilde{K}(\mu,\mu^\prime)W_1(\mu,\mu^\prime)$  \\
\hline
\end{tabularx}
\label{table:related_work}
\end{table}
 where $$\tilde{K}(\mu,\mu^\prime)=(1+D\lVert h \rVert _{\text{Lip}}(y_k))\left(\frac{1+\lVert h \rVert _{\text{Lip}}(y_k)\frac{\lvert \mu \rvert_{\mathcal{P}_1}}{Z_k(\mu)}}{Z_k(\mu^\prime)} \wedge \frac{1+\lVert h \rVert _{\text{Lip}}(y_k)\frac{\lvert \mu^\prime \rvert_{\mathcal{P}_1}}{Z_k(\mu^\prime)}}{Z_k(\mu)}\right),$$
with $\lvert \mu \rvert_{\mathcal{P}_1}$ defined under the metric space $(\X,d_{\X})$ as
\begin{equation*}
\lvert \mu \rvert_{\mathcal{P}^1} := \inf_{x_0 \in \X}\int_{\X}d_{\X}(x,x_0)\mu(dx).
\end{equation*}
 
 Table~\ref{table:related_work} shows that our bounds are tighter than those presented in \cite{lipschitz_stability_ip} under the total variation and Hellinger distances. As we have $Z_k(\mu) \vee Z_k(\mu^\prime) \leq 1$, the ratio between our bound for the Hellinger distance between the posteriors and that in \cite{lipschitz_stability_ip} is less than $\frac{Z_k(\mu) \wedge Z_k(\mu^\prime)}{Z_k(\mu) \vee Z_k(\mu^\prime)}$.  This ratio can be particularly small when the priors differ significantly. For $1$-Wasserstein distance, \cite{lipschitz_stability_ip} presents two upper bounds. It can be readily proven that our bound is tighter than $\frac{\left(1+D\lVert h \rVert _{\text{Lip}}(y_k)\right)^2}{\left(Z_k(\mu) \wedge Z_k(\mu^\prime)\right)^2}W_1(\mu,\mu^\prime)$. However, for two arbitrary priors $\mu$ and $\mu^\prime$, a direct comparison between our bound and $\tilde{K}(\mu,\mu^\prime)W_1(\mu,\mu^\prime)$ is not possible; the relative tightness depends on the specific choice of $\mu$ and $\mu^\prime$. It is noteworthy that the bounds in \cite{lipschitz_stability_ip} only establish pointwise \textit{local} Lipschitz continuity of the prior-to-posterior map $F_k$ under the Hellinger or $1$-Wasserstein distance. In contrast, our bounds enable the establishment of pointwise \textit{global} Lipschitz continuity of $F_k$ under all three distance metrics. It is also noteworthy that although both results prove the global Lipschitz stability of posterior with respect to prior under total variation distance, our upper bound is tighter by a factor of $1/2$.

\subsection{Comparison with results in \cite{local_pf_dim,ds,ip_ds}}
\label{sec:app:compare2}
For state estimation problems, \cite{local_pf_dim,ds,ip_ds} prove that given the data $y_k$, if there exists $\kappa \in (0,1)$ such that the observation model $h_k(y_k,x)$ satisfies
\begin{equation}
\kappa \leq \alpha h_k(y_k,x) \leq \frac{1}{\kappa}, \quad \forall x \in \X,
\label{ieq:assump:lit:se}
\end{equation}
where $\alpha >0$ is a scaling factor, then the prior-to-posterior map $F_k$ satisfies
\begin{equation}
d_{TV}(F_k(\mu),F_k(\mu^\prime)) \leq \frac{2}{\kappa^2}d_{TV}(\mu,\mu^\prime), \quad \forall \mu,\mu^\prime \in \bar{\mathcal{P}}_k(\X),
\label{ieq:ub:lit:se}
\end{equation}
This result establishes global Lipschitz stability under the assumption that the observation model $h_k(y_k,x)$ is bounded above and below by positive constants. While the observation function is often upper bounded in practice, the lower boundedness assumption is much stronger and does not hold in many problems. For example, consider a common setting: a scalar system with observation $Y= aX+V$, where $V$ is a Gaussian noise. In this case, the likelihood is not bounded below by a positive constant.

Next, we show that under the assumption in Equation~\eqref{ieq:assump:lit:se}, the upper bound provided by Theorem~\ref{thm:pointwise_lipschitz} is half of that given in Equation~\eqref{ieq:ub:lit:se}. When Equation~\eqref{ieq:assump:lit:se} holds, the assumption for Theorem~\ref{thm:pointwise_lipschitz} under the total variation distance for state estimation problems is satisfied. Theorem~\ref{thm:pointwise_lipschitz} yields:
\begin{equation}
d_{TV}(F_k(\mu),F_k(\mu^\prime)) \leq \frac{C_{Th}(y_k;k)}{Z_k(\mu)\vee Z_k(\mu^\prime)}d_{TV}(\mu,\mu^\prime), \quad \forall \mu,\mu^\prime \in \bar{\mathcal{P}}_k(\X),
\label{ieq:ub:our:se}
\end{equation}
Under the assumption in Equation~\eqref{ieq:assump:lit:se}, $C_{Th}(y_k;k)$ and $Z_k(\mu)$ satisfy
\begin{equation}
C_{Th}(y_k;k) := \sup_{x_{k-1} \in \X}\int_{\X}h_k(y_k,x)T_k(x,x_{k-1})dx \leq \sup_{x \in \X}h_k(y_k,x) \leq \frac{1}{\alpha \kappa},
\label{ieq:ub:our:se:1}
\end{equation}
and 
\begin{equation}
Z_k(\mu) = \int_{\X}h(y_k,x)F_k^-(\mu)(dx) \geq \frac{\kappa}{\alpha}\int_{\X}F_k^-(\mu)(dx)=\frac{\kappa}{\alpha},
\label{ieq:ub:our:se:2}
\end{equation}
where $F_k^-(\mu)$ represents the probability measure obtained by propagating the prior $\mu$ through the dynamics. Substituting Equations~\eqref{ieq:ub:our:se:1} and ~\eqref{ieq:ub:our:se:2} into Equation~\eqref{ieq:ub:our:se} gives
\begin{equation}
d_{TV}(F_k(\mu),F_k(\mu^\prime)) \leq \frac{1}{\kappa ^2}d_{TV}(\mu,\mu^\prime), \quad \forall \mu,\mu^\prime \in \bar{\mathcal{P}}_k(\X).
\label{ieq:ub:our:se:final}
\end{equation}
Thus, the bound given in Equation~\eqref{ieq:ub:our:se:final} is exactly $\frac{1}{2}$ of the bound provided by Equation~\eqref{ieq:ub:lit:se}.

\section{Proof of Theorem~\ref{thm:pointwise_lipschitz}}
\label{sec:app:pointwise_lipschitz}

Theorem~\ref{thm:pointwise_lipschitz} is a conclusion of Propositions~\ref{prop:ip:normalized_bound_tv},~\ref{prop:se:normalized_bound_tv},~\ref{prop:ps:normalized_bound_tv},~\ref{prop:ip:normalized_bound_H},~\ref{prop:se:normalized_bound_H},~\ref{prop:ps:normalized_bound_H},~\ref{prop:ip:normalized_bound_W},~\ref{prop:se:normalized_bound_W}, and~\ref{prop:ps:normalized_bound_W}.

\subsection{Proof of Theorem~\ref{thm:pointwise_lipschitz} under the total variation distance}
Theorem~\ref{thm:pointwise_lipschitz} under the total variation distance is a conclusion of Propositions~\ref{prop:ip:normalized_bound_tv}---\ref{prop:ps:normalized_bound_tv}.
\begin{prop}
In inverse problems, given the data $y_k$, under Assumption IP.1, $\forall \mu, \mu^\prime \in \bar{\mathcal{P}}_k(\X )$, the posteriors $F_k(\mu)$ and $F_k(\mu^\prime)$ satisfy
\begin{equation*}
d_{TV}(F_k(\mu),F_k(\mu^\prime)) \leq \frac{C_h(y_k)}{Z_k(\mu)}d_{TV}(\mu,\mu^\prime).
\end{equation*}
\label{prop:ip:normalized_bound_tv}
\end{prop}

\begin{proof}
Let $\nu$ be a $\sigma$-finite measure defined on the measurable space $(\X,\mathcal{B}(\X))$. Assume that $\mu \ll \nu$ and $\mu^\prime \ll \nu$. As we have $F_k(\mu) \ll \mu$ and $F_k(\mu^\prime) \ll \mu^\prime$, $F_k(\mu)$ and $F_k(\mu^\prime)$ are both absolutely continuous w.r.t. $\nu$.

By the definition of total variation distance, we have
\begin{align}
&d_{TV}(F_k(\mu),F_k(\mu^\prime)) \nonumber \\
&\qquad =\frac{1}{2}\int_\mathcal{X} \bigg \lvert \frac{h(y_k,x)}{Z_k(\mu)}\cdot \frac{d\mu}{d\nu}(x) - \frac{h(y_k,x)}{Z_k(\mu^\prime)}\cdot \frac{d\mu^\prime}{d\nu}(x) \bigg \rvert \nu(dx) \nonumber \\
& \qquad \leq  \frac{1}{2}\int_\mathcal{X} \bigg \lvert \frac{h(y_k,x)}{Z_k(\mu)}\cdot \frac{d\mu}{d\nu}(x) - \frac{h(y_k,x)}{Z_k(\mu)}\cdot \frac{d\mu^\prime}{d\nu}(x) \bigg \rvert \nu(dx) \nonumber \\
&\qquad \qquad + \frac{1}{2}\int_\mathcal{X} \bigg \lvert \frac{h(y_k,x)}{Z_k(\mu)}\cdot \frac{d\mu^\prime}{d\nu}(x) - \frac{h(y_k,x)}{Z_k(\mu^\prime)}\cdot \frac{d\mu^\prime}{d\nu}(x) \bigg \rvert \nu(dx). 
\label{ieq:ip:tv:1}
\end{align}
The first term on the right-hand side in Equation~\eqref{ieq:ip:tv:1} can be expressed as:
\begin{align*}
&\frac{1}{2}\int_\mathcal{X} \bigg \lvert \frac{h(y_k,x)}{Z_k(\mu)}\cdot \frac{d\mu}{d\nu}(x) - \frac{h(y_k,x)}{Z_k(\mu)}\cdot \frac{d\mu^\prime}{d\nu}(x) \bigg \rvert \nu(dx) \\
&\qquad =\frac{1}{2Z_k(\mu)}\int_\mathcal{X} h(y_k,x) \bigg \lvert \frac{d\mu}{d\nu}(x) - \frac{d\mu^\prime}{d\nu}(x) \bigg \rvert \nu(dx).
\end{align*}
We can write the second term on the right-hand side in Equation~\eqref{ieq:ip:tv:1} as:
\begin{align*}
&\frac{1}{2}\int_\mathcal{X} \bigg \lvert \frac{h(y_k,x)}{Z_k(\mu)}\cdot \frac{d\mu^\prime}{d\nu}(x) - \frac{h(y_k,x)}{Z_k(\mu^\prime)}\cdot \frac{d\mu^\prime}{d\nu}(x) \bigg \rvert \nu(dx) \\
&\qquad =\frac{1}{2}\bigg \lvert \frac{1}{Z_k(\mu)} - \frac{1}{Z_k(\mu^\prime)} \bigg \rvert \int_\mathcal{X}  h(y_k,x)\cdot \frac{d\mu^\prime}{d\nu}(x)\nu(dx).
\end{align*}
Note that 
\begin{equation*}
\int_\mathcal{X}  h(y_k,x)\cdot \frac{d\mu^\prime}{d\nu}(x)\nu(dx) = Z_k(\mu^\prime).
\end{equation*}
Therefore the second term in Equation~\eqref{ieq:ip:tv:1} can be reformulated as:
\begin{align*}
\frac{1}{2}\int_\mathcal{X} \bigg \lvert \frac{h(y_k,x)}{Z_k(\mu)}\cdot \frac{d\mu^\prime}{d\nu}(x) - \frac{h(y_k,x)}{Z_k(\mu^\prime)}\cdot \frac{d\mu^\prime}{d\nu}(x) \bigg \rvert \nu(dx) = \frac{\lvert Z_k(\mu) - Z_k(\mu^\prime) \rvert }{2Z_k(\mu)}.
\end{align*}
It follows that
\begin{align}
&d_{TV}(F_k(\mu),F_k(\mu^\prime)) \nonumber \\
&\qquad \leq \frac{1}{2Z_k(\mu)}\left(\int_\mathcal{X} h(y_k,x) \bigg \lvert \frac{d\mu}{d\nu}(x) - \frac{d\mu^\prime}{d\nu}(x) \bigg \rvert \nu(dx) + \bigg \lvert Z_k(\mu) - Z_k(\mu^\prime) \bigg \rvert \right).
\label{ieq:ip:tv:2}
\end{align}
Similarly, we can prove that 
\begin{align}
&d_{TV}(F_k(\mu),F_k(\mu^\prime)) \nonumber \\
&\qquad \leq \frac{1}{2Z_k(\mu^\prime)}\left(\int_\mathcal{X} h(y_k,x) \bigg \lvert \frac{d\mu}{d\nu}(x) - \frac{d\mu^\prime}{d\nu}(x) \bigg \rvert \nu(dx) + \bigg \lvert Z_k(\mu) - Z_k(\mu^\prime) \bigg \rvert \right).
\label{ieq:ip:tv:2_new}
\end{align}
Combining Equations~\eqref{ieq:ip:tv:2} and~\eqref{ieq:ip:tv:2_new} gives
\begin{align}
d_{TV}(F_k(\mu),F_k(\mu^\prime)) 
&\leq \frac{1}{2(Z_k(\mu) \vee Z_k(\mu^\prime))}\left({\int_\mathcal{X} h(y_k,x) \bigg \lvert \frac{d\mu}{d\nu}(x) - \frac{d\mu^\prime}{d\nu}(x) \bigg \rvert \nu(dx)}\right. \nonumber \\
&\qquad \left.{ + \bigg \lvert Z_k(\mu) - Z_k(\mu^\prime) \bigg \rvert }\right).
\label{ieq:ip:tv:3_new}
\end{align}
Define a space $\mathcal{X}^*$ as:
\begin{equation*}
\mathcal{X}^* := \begin{cases} \{x \in \X: \frac{d\mu}{d\nu}(x) \geq \frac{d\mu^\prime}{d\nu}(x)\}, \quad \text{if } Z_k(\mu) \geq Z_k(\mu^\prime), \nonumber \\
\{x \in \X: \frac{d\mu}{d\nu}(x) \leq \frac{d\mu^\prime}{d\nu}(x)\}, \quad \text{otherwise}. \end{cases}
\end{equation*}
Suppose we have $Z_k(\mu) \geq Z_k(\mu^\prime)$. Then the space $\X^*$ is given by
\begin{equation*}
\mathcal{X}^* = \{x \in \X: \frac{d\mu}{d\nu}(x)-\frac{d\mu^\prime}{d\nu}(x) \geq 0\}.
\end{equation*} \\
Equation~\eqref{ieq:ip:tv:3_new} becomes
\begin{align}
d_{TV}(F_k(\mu),F_k(\mu^\prime)) \leq& \frac{1}{2(Z_k(\mu)\vee Z_k(\mu^\prime))}\left({\int_\mathcal{X} h(y_k,x) \bigg \lvert \frac{d\mu}{d\nu}(x) - \frac{d\mu^\prime}{d\nu}(x) \bigg \rvert \nu(dx)}\right. \nonumber \\
&\qquad \left.{+ Z_k(\mu) - Z_k(\mu^\prime) }\right).
\label{ieq:ip:tv:5}
\end{align}
According to the Radon-Nikodym theorem, $\frac{d\mu}{d\nu}$ and $\frac{d\mu^\prime}{d\nu}$, as Radon-Nikodym derivatives, are $(\mathcal{B}(\X),\mathcal{B}(\reals))$-measurable. Therefore, the function $\frac{d\mu}{d\nu}-\frac{d\mu^\prime}{d\nu}$ is also $(\mathcal{B}(\X),\mathcal{B}(\reals))$-measurable. 
Since $[0,+\infty)$ belongs to $\mathcal{B}(\reals)$, the set $\X^*$ belongs to $\mathcal{B}(\X)$. 
Then the term $Z_k(\mu) - Z_k(\mu^\prime)$ can be written as
\begin{align}
Z_k(\mu) - Z_k(\mu^\prime) 
=& \int_{\X^*}h(y_k,x)\left(\frac{d\mu}{d\nu}(x) - \frac{d\mu^\prime}{d\nu}(x)\right)\nu(dx) \nonumber \\
&\qquad +\int_{\mathcal{X} \setminus \X^*} h(y_k,x)\left(\frac{d\mu}{d\nu}(x) - \frac{d\mu^\prime}{d\nu}(x)\right)\nu(dx).
\label{eq:ip:tv:1}
\end{align} \\
The term $\int_\mathcal{X} h(y_k,x) \bigg \lvert \frac{d\mu}{d\nu}(x) - \frac{d\mu^\prime}{d\nu}(x) \bigg \rvert \nu(dx)$ in Equation~\eqref{ieq:ip:tv:5} can be expressed as
\begin{align}
\int_\mathcal{X} h(y_k,x) \bigg \lvert \frac{d\mu}{d\nu}(x) - \frac{d\mu^\prime}{d\nu}(x) \bigg \rvert \nu(dx) 
&=\int_{\X^*}h(y_k,x) \left( \frac{d\mu}{d\nu}(x) - \frac{d\mu^\prime}{d\nu}(x) \right) \nu(dx) \nonumber \\
&\quad + \int_{\mathcal{X}\setminus \X^*}h(y_k,x) \left(\frac{d\mu^\prime}{d\nu}(x) - \frac{d\mu}{d\nu}(x)  \right) \nu(dx).
\label{eq:ip:tv:2}
\end{align}
Substitute Equation~\eqref{eq:ip:tv:1} and~\eqref{eq:ip:tv:2} into Equation~\eqref{ieq:ip:tv:5}. We obtain
\begin{align}
d_{TV}(F_k(\mu),F_k(\mu^\prime)) \leq & 
\frac{1}{Z_k(\mu)\vee Z_k(\mu^\prime)}\int_{\X^*} h(y_k,x) \bigg \lvert  \frac{d\mu}{d\nu}(x) - \frac{d\mu^\prime}{d\nu}(x) \bigg \rvert \nu(dx).
\label{ieq:intermediate:tv:ip}
\end{align}
We can write the total variation distance between the two priors $\mu$ and $\mu^\prime$ as
\begin{align*}
d_{TV}(\mu,\mu^\prime)
=& \frac{1}{2}\left(\int_{\X^*} \left(\frac{d\mu}{d\nu}(x) - \frac{d\mu^\prime}{d\nu}(x) \right)\nu(dx) + \int_{\mathcal{X}\setminus \X^*}\left( \frac{d\mu^\prime}{d\nu}(x) -\frac{d\mu}{d\nu}(x) \right)\nu(dx) \right).
\end{align*}
Note that
\begin{equation*}
\int_\mathcal{X}\frac{d\mu}{d\nu}(x)\nu(dx) = 1 = \int_\mathcal{X}\frac{d\mu^\prime}{d\nu}(x)\nu(dx).
\end{equation*}
Therefore, we have
\begin{equation*}
\int_{\X^*}\frac{d\mu}{d\nu}(x)\nu(dx) + \int_{\mathcal{X} \setminus \X^*}\frac{d\mu}{d\nu}(x)\nu(dx) = \int_{\X^*}\frac{d\mu^\prime}{d\nu}(x)\nu(dx) + \int_{\mathcal{X} \setminus \X^*}\frac{d\mu^\prime}{d\nu}(x)\nu(dx).
\end{equation*}
It follows that
\begin{equation*}
\int_{\X^*}\left(\frac{d\mu}{d\nu}(x) - \frac{d\mu^\prime}{d\nu}(x)\right)\nu(dx) = \int_{\mathcal{X} \setminus \X^*}\left(\frac{d\mu^\prime}{d\nu}(x) -\frac{d\mu}{d\nu}(x)\right)\nu(dx),
\end{equation*}
and we can re-write $d_{TV}(\mu,\mu^\prime)$ as
\begin{equation}
d_{TV}(\mu,\mu^\prime) = \int_{\X^*}\bigg \lvert \frac{d\mu}{d\nu}(x) - \frac{d\mu^\prime}{d\nu}(x) \bigg \rvert \nu(dx).
\label{eq:ip:normalized:prior_diff:rewrite}
\end{equation}
Similarly, we can prove that Equations~\eqref{ieq:intermediate:tv:ip} and ~\eqref{eq:ip:normalized:prior_diff:rewrite} hold when $Z_k(\mu)<Z_k(\mu^\prime)$. \\
Given Equations~\eqref{ieq:intermediate:tv:ip} and~\eqref{eq:ip:normalized:prior_diff:rewrite}, under Assumption IP.1, we have
\begin{align*}
d_{TV}(F_k(\mu),F_k(\mu^\prime)) 
\leq& \frac{C_h(y_k)}{Z_k(\mu)\vee Z_k(\mu^\prime)}\int_{\X^*} \left( \frac{d\mu}{d\nu}(x) - \frac{d\mu^\prime}{d\nu}(x) \right) \nu(dx) \\
=& \frac{C_h(y_k)}{Z_k(\mu)\vee Z_k(\mu^\prime)}d_{TV}(\mu,\mu^\prime) \\
\leq & \frac{C_h(y_k)}{Z_k(\mu)}d_{TV}(\mu,\mu^\prime)
\end{align*}
\end{proof}

\begin{prop}
In state estimation problems, given the data $y_k$, under Assumption SE.1, $\forall \mu, \mu^\prime \in \bar{\mathcal{P}}_k(\X )$, the posteriors $F_k(\mu)$ and $F_k(\mu^\prime)$ satisfy
\begin{equation*}
d_{TV}(F_k(\mu),F_k(\mu^\prime)) \leq \frac{C_{Th}(y_k;k)}{Z_k(\mu)}d_{TV}(\mu,\mu^\prime).
\end{equation*}
\label{prop:se:normalized_bound_tv}
\end{prop}

\begin{proof}
Let $\nu$ be a $\sigma$-finite measure defined on the measurable space $(\X,\mathcal{B}(\X))$ such that $\mu \ll \nu$ and $\mu^\prime \ll \nu$. Let $\pi$ and $\pi^\prime$ denote the Radon-Nikodym derivatives $\frac{d\mu}{d\nu}$ and $\frac{d\mu^\prime}{d\nu}$, respectively. By the definition of $F_k$ for state estimation problems, $d_{TV}(F_k(\mu),F_k(\mu^\prime))$ can be written as 
\begin{align}
&d_{TV}(F_k(\mu),F_k(\mu^\prime)) = \frac{1}{2}\int_\mathcal{X}\bigg \lvert \frac{h_k(y_k,x)}{Z_k(\mu)}\int_{\X}T_k(x,x_{k-1})\mu(dx_{k-1}) \nonumber \\
&\quad \quad + \frac{h_k(y_k,x)}{Z_k(\mu)}\int_{\X}T_k(x,x_{k-1})\mu^\prime(dx_{k-1}) - \frac{h_k(y_k,x)}{Z_k(\mu^\prime)}\int_{\X}T_k(x,x_{k-1})\mu^\prime(dx_{k-1}) \bigg \rvert dx \nonumber \\
& \quad \leq \frac{1}{2}\int_\mathcal{X}\bigg \lvert \frac{h_k(y_k,x)}{Z_k(\mu)}\int_{\X}T_k(x,x_{k-1})\mu(dx_{k-1}) - \frac{h_k(y_k,x)}{Z_k(\mu)}\int_{\X}T_k(x,x_{k-1})\mu^\prime(dx_{k-1}) \bigg \rvert dx \nonumber \\
&\qquad \qquad + \frac{1}{2}\int_\mathcal{X}\bigg \lvert \frac{h_k(y_k,x)}{Z_k(\mu)}\int_{\X}T_k(x,x_{k-1})\mu^\prime(dx_{k-1}) \nonumber \\
&\qquad \qquad \qquad \qquad - \frac{h_k(y_k,x)}{Z_k(\mu^\prime)}\int_{\X}T_k(x,x_{k-1})\mu^\prime(dx_{k-1}) \bigg \rvert dx \nonumber \\
& \quad = \frac{1}{2Z_k(\mu)}\int_\mathcal{X}\bigg \lvert h_k(y_k,x)\int_{\X}T_k(x,x_{k-1})\left(\mu(dx_{k-1}) - \mu^\prime(dx_{k-1})\right) \bigg \rvert dx\nonumber \\
&\qquad \qquad + \frac{1}{2}\bigg \lvert \frac{1}{Z_k(\mu)} - \frac{1}{Z_k(\mu^\prime)}\bigg \rvert \int_\mathcal{X}h_k(y_k,x)\left(\int_{\X}T_k(x,x_{k-1})\mu^\prime(dx_{k-1})\right) dx.
\label{ieq:normalized:TV:se:1}
\end{align}
The first term on the right-hand side of Equation~\eqref{ieq:normalized:TV:se:1} satisfies
\begin{align}
&\frac{1}{2Z_k(\mu)}\int_\mathcal{X}\bigg \lvert h_k(y_k,x)\int_{\X}T_k(x,x_{k-1})\left(\mu(dx_{k-1}) - \mu^\prime(dx_{k-1})\right) \bigg \rvert dx \nonumber \\
& \qquad \leq \frac{1}{2Z_k(\mu)}\int_\mathcal{X} h_k(y_k,x)\left(\int_{\X}T_k(x,x_{k-1})\bigg \lvert \pi(x_{k-1}) - \pi^\prime(x_{k-1})\bigg \rvert \nu(dx_{k-1}) \right) dx \nonumber \\
& \qquad = \frac{1}{2Z_k(\mu)}\int_\mathcal{X} \left(\int_{\X}h_k(y_k,x)T_k(x,x_{k-1})dx\right) \bigg \lvert \pi(x_{k-1}) - \pi^\prime(x_{k-1})\bigg \rvert \nu(dx_{k-1}),
\label{ieq:normalized:TV:se:2}
\end{align}
where first equality is obtained using Tonelli's theorem. \\
Note that we have
\begin{equation*}
\int_\mathcal{X}h_k(y_k,x)\left(\int_{\X}T_k(x,x_{k-1})\mu^\prime(dx_{k-1})\right) dx = Z_k(\mu^\prime).
\end{equation*}
Hence the second term on the right-hand side of Equation~\eqref{ieq:normalized:TV:se:1} can be written as
\begin{align}
&\frac{1}{2}\bigg \lvert \frac{1}{Z_k(\mu)} - \frac{1}{Z_k(\mu^\prime)}\bigg \rvert \int_\mathcal{X}h_k(y_k,x)\left(\int_{\X}T_k(x,x_{k-1})\mu^\prime(dx_{k-1})\right) dx \nonumber \\
&\qquad = \frac{\lvert Z_k(\mu) - Z_k(\mu^\prime) \rvert}{2Z_k(\mu)Z_k(\mu^\prime)}Z_k(\mu^\prime) =\frac{\lvert Z_k(\mu) - Z_k(\mu^\prime) \rvert}{2Z_k(\mu)}.
\label{eq:normalized:TV:se:1}
\end{align}
Plug Equations~\eqref{ieq:normalized:TV:se:2} and ~\eqref{eq:normalized:TV:se:1} into ~\eqref{ieq:normalized:TV:se:1}. We obtain
\begin{align}
&d_{TV}(F_k(\mu),F_k(\mu^\prime)) \nonumber \\
&\qquad \leq \frac{1}{2Z_k(\mu)}\left({\int_\mathcal{X} \left(\int_{\X}h_k(y_k,x)T_k(x,x_{k-1})dx\right)\bigg \lvert \pi(x_{k-1}) - \pi^\prime(x_{k-1})\bigg \rvert \nu(dx_{k-1}) }\right. \nonumber \\
&\qquad \qquad \left.{+\lvert Z_k(\mu) - Z_k(\mu^\prime) \rvert}\right) .
\label{ieq:se:tv:total}
\end{align}
Similarly, we can show that $d_{TV}(F_k(\mu),F_k(\mu^\prime))$ also satisfies
\begin{align}
&d_{TV}(F_k(\mu),F_k(\mu^\prime)) \nonumber \\
&\qquad \leq \frac{1}{2Z_k(\mu^\prime)}\left({\int_\mathcal{X} \left(\int_{\X}h_k(y_k,x)T_k(x,x_{k-1})dx\right)\bigg \lvert \pi(x_{k-1}) - \pi^\prime(x_{k-1})\bigg \rvert \nu(dx_{k-1}) }\right. \nonumber \\
&\qquad \qquad \left.{+\lvert Z_k(\mu) - Z_k(\mu^\prime) \rvert}\right) .
\label{ieq:se:tv:total_new}
\end{align}
Combining Equation~\eqref{ieq:se:tv:total} and~\eqref{ieq:se:tv:total_new} gives
\begin{align}
d_{TV}(F_k(\mu),F_k(\mu^\prime)) &\leq \frac{1}{2(Z_k(\mu)\vee Z_k(\mu^\prime))}\left({\int_\mathcal{X} \left(\int_{\X}h_k(y_k,x)T_k(x,x_{k-1})dx\right)}\right. \nonumber \\
&\qquad \left.{\cdot \bigg \lvert \pi(x_{k-1}) - \pi^\prime(x_{k-1})\bigg \rvert \nu(dx_{k-1}) +\lvert Z_k(\mu) - Z_k(\mu^\prime) \rvert}\right) .
\label{ieq:se:tv:total_new_2}
\end{align}
By Tonelli's theorem, we can write the evidence terms $Z_k(\mu)$ and $Z_k(\mu^\prime)$ as:
\begin{align}
Z_k(\mu) 
&= \int_{\X}h_k(y_k,x)\left(\int_{\X}T_k(x,x_{k-1})\pi(x_{k-1})\nu(dx_{k-1})\right)dx_k \nonumber \\
&= \int_{\X}\left(\int_{\X}h_k(y_k,x)T_k(x,x_{k-1})dx \right)\pi(x_{k-1})\nu(dx_{k-1}),
\label{eq:se:tv:def:evidence_mu}
\end{align}
and
\begin{align}
Z_k(\mu^\prime) 
&= \int_{\X}h_k(y_k,x)\left(\int_{\X}T_k(x,x_{k-1})\pi^\prime(x_{k-1})\nu(dx_{k-1})\right)dx \nonumber \\
&= \int_{\X}\left(\int_{\X}h_k(y_k,x)T_k(x,x_{k-1})dx \right)\pi^\prime(x_{k-1})\nu(dx_{k-1}). 
\label{eq:se:tv:def:evidence_mu_prime}
\end{align}
Define a space $\X^*$ as:
\begin{equation*}
\mathcal{X}^* := \begin{cases} \{x \in \X: \pi(x) \geq \pi^\prime(x)\}, \quad \text{if } Z_k(\mu) \geq Z_k(\mu^\prime), \nonumber \\
\{x \in \X: \pi(x) \leq \pi^\prime(x)\}, \quad \text{otherwise}. \end{cases}
\end{equation*}
Following an argument analogous to that in the proof of Proposition~\ref{prop:ip:normalized_bound_tv}, we obtain
\begin{align}
d_{TV}(F_k(\mu),F_k(\mu^\prime)) 
&\leq \frac{1}{Z_k(\mu)\vee Z_k(\mu^\prime)}\int_{\X^*}\left(\int_{\X}h_k(y_k,x)T_k(x,x_{k-1})dx \right)\nonumber \\
&\qquad \cdot \lvert \pi(x_{k-1}) - \pi^\prime(x_{k-1}) \rvert \nu(dx_{k-1}),
\label{ieq:intermediate:tv:se}
\end{align}
and 
\begin{equation}
d_{TV}(\mu,\mu^\prime) = \int_{\X^*}\lvert \pi(x_{k-1}) - \pi^\prime(x_{k-1}) \rvert \nu(dx_{k-1}).
\label{ieq:intermediate:tv:se_prior}
\end{equation}
Combining Equations~\eqref{ieq:intermediate:tv:se} and~\eqref{ieq:intermediate:tv:se_prior} with Assumption~SE.1 gives
\begin{align*}
d_{TV}(F_k(\mu),F_k(\mu^\prime)) \leq& \frac{C_{Th}(y_k;k)}{Z_k(\mu)\vee Z_k(\mu^\prime)}d_{TV}(\mu,\mu^\prime) \\
\leq& \frac{C_{Th}(y_k;k)}{Z_k(\mu)}d_{TV}(\mu,\mu^\prime).
\end{align*}
\end{proof}

\begin{prop}
In parameter-state estimation problems, given the data $y_k$, under Assumption PS.1, $\forall \mu, \mu^\prime \in \bar{\mathcal{P}}_k(\X \times \W)$, the posteriors $F_k(\mu)$ and $F_k(\mu^\prime)$ satisfy
\begin{equation*}
d_{TV}(F_k(\mu),F_k(\mu^\prime)) \leq \frac{\tilde{C}_{Th}(y_k;k)}{Z_k(\mu)}d_{TV}(\mu,\mu^\prime).
\end{equation*}
\label{prop:ps:normalized_bound_tv}
\end{prop}
\begin{proof}
Let $\pi$ and $\pi^\prime$ denote the densities of $\mu$ and $\mu^\prime$ with respect to the Lebesgue measure, respectively. By the definition of $F_k$ for parameter-state estimation problems, $d_{TV}(F_k(\mu),F_k(\mu^\prime))$ can be written as
\begin{align}
&d_{TV}(F_k(\mu),F_k(\mu^\prime)) \nonumber \\ 
&\qquad = \frac{1}{2}\int_{\mathcal{X}\times \W} \bigg \lvert \frac{h_k(y_k,x,w)}{Z_k(\mu)}\int_{\X}T_k(x,x_{k-1},w)\pi(x_{k-1},w)dx_{k-1} \nonumber \\
&\qquad \qquad - \frac{h_k(y_k,x,w)}{Z_k(\mu^\prime)}\int_{\X}T_k(x,x_{k-1},w)\pi^\prime(x_{k-1},w)dx_{k-1} \bigg \rvert dxdw \nonumber \\
& \qquad \leq \frac{1}{2}\int_{\mathcal{X}\times \W}\bigg \lvert \frac{h_k(y_k,x,w)}{Z_k(\mu)}\int_{\X}T_k(x,x_{k-1},w)\pi(x_{k-1},w)dx_{k-1} \nonumber \\
& \qquad \qquad - \frac{h_k(y_k,x,w)}{Z_k(\mu)}\int_{\X}T_k(x,x_{k-1},w)\pi^\prime(x_{k-1},w)dx_{k-1} \bigg \rvert dxdw \nonumber \\
&\qquad \qquad + \frac{1}{2}\int_{\mathcal{X}\times \W}\bigg \lvert \frac{h_k(y_k,x,w)}{Z_k(\mu)}\int_{\X}T_k(x,x_{k-1},w)\pi^\prime(x_{k-1},w)dx_{k-1} \nonumber \\
& \qquad \qquad - \frac{h_k(y_k,x,w)}{Z_k(\mu^\prime)}\int_{\X}T_k(x,x_{k-1},w)\pi^\prime(x_{k-1},w)dx_{k-1} \bigg \rvert dxdw 
\label{ieq:normalized:tv:ps:1}
\end{align}
The first integral on the right-hand side of Equation~\eqref{ieq:normalized:tv:ps:1} satisfies
\begin{align*}
&\frac{1}{2}\int_{\mathcal{X}\times \W}\bigg \lvert \frac{h_k(y_k,x,w)}{Z_k(\mu)}\int_{\X}T_k(x,x_{k-1},w)\pi(x_{k-1},w)dx_{k-1} \nonumber \\
& \qquad \qquad - \frac{h_k(y_k,x,w)}{Z_k(\mu)}\int_{\X}T_k(x,x_{k-1},w)\pi^\prime(x_{k-1},w)dx_{k-1} \bigg \rvert dxdw \nonumber \\
&\quad \leq \frac{1}{2Z_k(\mu)}\int_{\mathcal{X}\times \W}h_k(y_k,x,w) \nonumber \\
&\qquad \qquad \cdot \left( \int_{\X}T_k(x,x_{k-1},w)\bigg \lvert \pi(x_{k-1},w) - \pi^\prime(x_{k-1},w)\bigg \rvert dx_{k-1} \right) dxdw.
\end{align*}
According Tonelli's theorem, we have
\begin{align*}
&\int_{\mathcal{X}\times \W}h_k(y_k,x,w)\left( \int_{\X}T_k(x,x_{k-1},w)\bigg \lvert \pi(x_{k-1},w) - \pi^\prime(x_{k-1},w)\bigg \rvert dx_{k-1} \right)dxdw \\
& \quad = \int_{\X\times \W}\left( \int_{\X}h_k(y_k,x,w)T_k(x,x_{k-1},w)dx \right) \bigg \lvert \pi(x_{k-1},w) - \pi^\prime(x_{k-1},w)\bigg \rvert dx_{k-1}dw.
\end{align*} \\
The second integral on the right-hand side of Equation~\eqref{ieq:normalized:tv:ps:1} can be written as
\begin{align*}
&\frac{1}{2}\int_{\mathcal{X}\times \W}\bigg \lvert \frac{h_k(y_k,x,w)}{Z_k(\mu)}\int_{\X}T_k(x,x_{k-1},w)\pi^\prime(x_{k-1},w)dx_{k-1} \nonumber \\
& \qquad \qquad - \frac{h_k(y_k,x,w)}{Z_k(\mu^\prime)}\int_{\X}T_k(x,x_{k-1},w)\pi^\prime(x_{k-1},w)dx_{k-1} \bigg \rvert dxdw \nonumber \\
&\quad = \frac{1}{2}\bigg \lvert \frac{1}{Z_k(\mu)} - \frac{1}{Z_k(\mu^\prime)}\bigg \rvert \nonumber \\
&\qquad \qquad \cdot \int_{\mathcal{X}\times \W}h_k(y_k,x,w)\left(\int_{\X}T_k(x,x_{k-1},w)\pi^\prime(x_{k-1},w)dx_{k-1}\right)  dxdw \nonumber \\
&\quad = \frac{\lvert Z_k(\mu) - Z_k(\mu^\prime) \rvert}{2Z_k(\mu)Z_k(\mu^\prime)}Z_k(\mu^\prime) =\frac{\lvert Z_k(\mu) - Z_k(\mu^\prime) \rvert}{2Z_k(\mu)}.
\end{align*}
By Tonelli's theorem, we can write the evidence term $Z_k(\mu)$ as:
\begin{align}
Z_k(\mu) 
&= \int_{\mathcal{X}\times \W}h_k(y_k,x,w)\left( \int_{\X}T_k(x,x_{k-1},w)  \pi(x_{k-1},w) dx_{k-1} \right) dxdw \nonumber \\
& = \int_{\X\times \W}\left( \int_{\X}h_k(y_k,x,w)T_k(x,x_{k-1},w)dx \right) \pi(x_{k-1},w) dx_{k-1}dw.
\label{eq:ps:tv:def:evidence_mu}
\end{align}
Similarly, the evidence term $Z_k(\mu^\prime)$ can be expressed as:
\begin{align}
Z_k(\mu^\prime) 
& = \int_{\X\times \W}\left( \int_{\X}h_k(y_k,x,w)T_k(x,x_{k-1},w)dx \right) \pi^\prime(x_{k-1},w) dx_{k-1}dw.
\label{eq:ps:tv:def:evidence_mu_prime}
\end{align}
Define a space $\mathcal{S}^*$ as:
\begin{equation*}
\mathcal{S}^* := \begin{cases} \{(x,w) \in \X\times \W: \pi(x,w) \geq \pi^\prime(x,w)\}, \quad \text{if } Z_k(\mu) \geq Z_k(\mu^\prime), \nonumber \\
\{(x,w) \in \X\times \W: \pi(x,w) \leq \pi^\prime(x,w)\}, \quad \text{otherwise}. \end{cases}
\end{equation*}
Analogous to the proof of Proposition~\ref{prop:ip:normalized_bound_tv}, the following can be proved:
\begin{align}
d_{TV}(F_k(\mu),F_k(\mu^\prime)) 
&\leq \frac{1}{Z_k(\mu)\vee Z_k(\mu^\prime)}\int_{\mathcal{S}^*}\left(\int_{\X}h_k(y_k,x,w)T_k(x,x_{k-1},w)dx\right)\nonumber \\
&\qquad \cdot \lvert \pi(x_{k-1},w) - \pi^\prime(x_{k-1},w) \rvert dx_{k-1}dw,
\label{ieq:intermediate:tv:ps}
\end{align}
and 
\begin{equation}
d_{TV}(\mu,\mu^\prime) = \int_{\mathcal{S}^*}\lvert \pi(x_{k-1},w) - \pi^\prime(x_{k-1},w) \rvert dx_{k-1}dw.
\label{ieq:intermediate:tv:ps_prior}
\end{equation}
Combining Equations~\eqref{ieq:intermediate:tv:ps} and ~\eqref{ieq:intermediate:tv:ps_prior} with Assumption PS.1 yields:
\begin{align*}
d_{TV}(F_k(\mu),F_k(\mu^\prime)) \leq& \frac{\tilde{C}_{Th}(y_k;k)}{Z_k(\mu)\vee Z_k(\mu^\prime)}d_{TV}(\mu,\mu^\prime) \\
\leq& \frac{\tilde{C}_{Th}(y_k;k)}{Z_k(\mu)}d_{TV}(\mu,\mu^\prime).
\end{align*}
\end{proof}

\subsection{Proof of Theorem~\ref{thm:pointwise_lipschitz} under Hellinger distance}
Theorem~\ref{thm:pointwise_lipschitz} under the Hellinger distance is a conclusion of Propositions~\ref{prop:ip:normalized_bound_H}---\ref{prop:ps:normalized_bound_H}.
\begin{prop}
In inverse problems, given the data $y_k$, under Assumption IP.1, $\forall \mu, \mu^\prime \in \bar{\mathcal{P}}_k(\X )$, the posteriors $F_k(\mu)$ and $F_k(\mu^\prime)$ satisfy
\begin{equation*}
d_{H}(F_k(\mu),F_k(\mu^\prime)) \leq 2\sqrt{\frac{C_h(y_k)}{Z_k(\mu) }}d_{H}(\mu,\mu^\prime).
\end{equation*}
\label{prop:ip:normalized_bound_H}
\end{prop}

\begin{prop}
In state estimation problems, given the data $y_k$, under Assumption SE.1, $\forall \mu, \mu^\prime \in \bar{\mathcal{P}}_k(\X )$, the posteriors $F_k(\mu)$ and $F_k(\mu^\prime)$ satisfy
\begin{equation*}
d_{H}(F_k(\mu),F_k(\mu^\prime)) \leq 2\sqrt{\frac{C_{Th}(y_k;k)}{Z_k(\mu) }}d_{H}(\mu,\mu^\prime).
\end{equation*}
\label{prop:se:normalized_bound_H}
\end{prop}

\begin{prop}
In parameter-state estimation problems, given the data $y_k$, under Assumption PS.1, $\forall \mu, \mu^\prime \in \bar{\mathcal{P}}_k(\X \times \W)$, the posteriors $F_k(\mu)$ and $F_k(\mu^\prime)$ satisfy
\begin{equation*}
d_{H}(F_k(\mu),F_k(\mu^\prime)) \leq 2\sqrt{\frac{\tilde{C}_{Th}(y_k;k)}{Z_k(\mu)}}d_{H}(\mu,\mu^\prime).
\end{equation*}
\label{prop:ps:normalized_bound_H}
\end{prop}

To prove Propositions~\ref{prop:ip:normalized_bound_H}---~\ref{prop:ps:normalized_bound_H}, we first present a supporting proposition, Proposition~\ref{thm:lipschitz}, which aids in the proof of Propositions~\ref{prop:ip:normalized_bound_H}---~\ref{prop:ps:normalized_bound_H}. 
\subsubsection{Supporting proposition for the proof of Theorem~\ref{thm:pointwise_lipschitz} under the Hellinger distance: Proposition~\ref{thm:lipschitz}}
Proposition~\ref{thm:lipschitz} is used to prove Propositions~\ref{prop:ip:normalized_bound_H}---~\ref{prop:ps:normalized_bound_H}. In this section, we provide Proposition~\ref{thm:lipschitz} and its proof.

Before presenting Proposition~\ref{thm:lipschitz}, we first define distances for scaled probability measures. Let $\tilde{\mathcal{P}}(\Z)$ denote the set of all scaled probability measures defined on $(\Z, \mathcal{B}(\Z))$, i.e., $$\tilde{\mathcal{P}}(\Z)=\{\tilde{\mu} \mid \tilde{\mu}(\epsilon) = c\mu(\epsilon), c \in \reals_+, c < +\infty, \mu \in \mathcal{P}(\Z), \epsilon \in \mathcal{B}(\Z)\}.$$ Let $\tilde{\mu}$ and $\tilde{\mu}^\prime$ be two measures in $\tilde{\mathcal{P}}(\Z)$. Suppose both $\tilde{\mu}$ and $\tilde{\mu}^\prime$ are absolutely continuous with respect to a common reference measure $\nu$ (such $\nu$ always exists). The distances $\tilde{d}_{TV}$ and $\tilde{d}_{H}$ between $\tilde{\mu}$ and $\tilde{\mu}^\prime$ are defined as follows.

\begin{definition}[\textbf{Distances for scaled probability measures}]
\label{def:general distance}
Distance $\tilde{d}_{TV}$ between $\tilde{\mu}$ and $\tilde{\mu}^\prime$ is given by
\begin{equation*}
\tilde{d}_{TV}(\tilde{\mu},\tilde{\mu}^\prime) := \frac{1}{2}\int_{\Z}\bigg \lvert \frac{d\tilde{\mu}}{d\nu}(z) - \frac{d\tilde{\mu}^\prime}{d\nu}(z) \bigg \rvert \nu(dz), 
\end{equation*}
and the distance $\tilde{d}_{H}$ between $\tilde{\mu}$ and $\tilde{\mu}^\prime$ is given by
\begin{equation*}
\tilde{d}_{H}(\tilde{\mu},\tilde{\mu}^\prime) := \sqrt{\frac{1}{2}\int_{\Z}\left( \sqrt{\frac{d\tilde{\mu}}{d\nu}(z) } - \sqrt{\frac{d\tilde{\mu}^\prime}{d\nu}(z)} \right)^2 \nu(dz)}.
\end{equation*}
\end{definition}
Note that the values of $\tilde{d}_{TV}(\tilde{\mu},\tilde{\mu}^\prime)$ and $\tilde{d}_{H}(\tilde{\mu},\tilde{\mu}^\prime)$ are invariant to the choice of the reference measure $\nu$. The total variation distance $d_{TV}$ and Hellinger distance $d_H$ can be viewed as specific instances of $\tilde{d}_{TV}$ and $\tilde{d}_{H}$, respectively, with the inputs restricted to probability measures. Building on Definition ~\ref{def:general distance}, we now present the following lemma.
\begin{lemma}\label{lemma:general metric}
If $(\Z,d_{\Z})$ is a metric space, then $(\tilde{\mathcal{P}}(\Z),\tilde{d}_{TV})$ and $(\tilde{\mathcal{P}}(\Z),\tilde{d}_{H})$ are metric spaces.
\end{lemma}
Let $\Z$ be a set and $d_{\Z}: \Z \times \Z \rightarrow \reals$ be a function. The pair $(\Z,d_{\Z})$ is a metric space if and only if the following conditions are satisfied for any $z1,z2,z3 \in \Z$.
\begin{enumerate}
\item $d_{\Z}(z_1,z_2)=0$ $\Leftrightarrow$ $z_1=z_2$.
\item \textbf{Non-negativity}: $d_{\Z}(z_1,z_2)\geq 0$.
\item \textbf{Symmetry}: $d_{\Z}(z_1,z_2)=d_{\Z}(z_2,z_1)$.
\item \textbf{Triangle inequality}: $d_{\Z}(z_1,z_3) \leq d_{\Z}(z_1,z_2)+d_{\Z}(z_2,z_3)$.
\end{enumerate}
It is straightforward to prove that $(\tilde{\mathcal{P}}(\Z),\tilde{d}_{TV})$ is a metric space and the pair $(\tilde{\mathcal{P}}(\Z),\tilde{d}_{H})$ satisfies the first three requirements for being a metric space. The corresponding proofs are omitted for brevity. 

Below, we provide a proof of the triangle inequality of $\tilde{d}_{H}$ in space $\tilde{\mathcal{P}}(\Z)$.
\begin{proof}
Assume $\mu_1, \mu_2, \mu_3 \in \tilde{\mathcal{P}}(\Z)$. Let the measure $\nu$ satisfies $\mu_1 \ll \nu, \quad \mu_2 \ll \nu$, and $\mu_3 \ll \nu$ --- note that such $\nu$ can always exist. Denote the Radon-Nikodym derivatives $\frac{d\mu_1}{d\nu}$, $\frac{d\mu_2}{d\nu}$, and $\frac{d\mu_3}{d\nu}$ by $g_1$, $g_2$, and $g_3$, respectively. By the definition of the distance $\tilde{d}_H$, $\tilde{d}_H(\mu_1,\mu_3)$ satisfies
\begin{align}
\tilde{d}_H^2(\mu_1,\mu_3) =& \frac{1}{2}\int_{\Z}\left(\sqrt{g_1(z)}-\sqrt{g_3(z)}\right)^2\nu(dz) = \frac{1}{2}\int_{\Z}\bigg \lvert \sqrt{g_1(z)}-\sqrt{g_3(z)}\bigg \rvert ^2\nu(dz).
\label{eq:dL2:1}
\end{align}
The square of $\tilde{d}_H(\mu_1,\mu_2)+\tilde{d}_H(\mu_2,\mu_3)$ can be written as
\begin{align}
&\left(\sqrt{\frac{1}{2}\int_{\Z}\left(\sqrt{g_1(z)}-\sqrt{g_2(z)}\right)^2\nu(dz)}+\sqrt{\frac{1}{2}\int_{\Z}\left(\sqrt{g_2(z)}-\sqrt{g_3(z)}\right)^2\nu(dz)}\right)^2 \nonumber \\
&\qquad = \frac{1}{2}\int_{\Z}\left(\sqrt{g_1(z)}-\sqrt{g_2(z)}\right)^2\nu(dz) + \frac{1}{2}\int_{\Z}\left(\sqrt{g_2(z)}-\sqrt{g_3(z)}\right)^2\nu(dz) \nonumber \\
& \qquad \qquad + \sqrt{\int_{\Z}\left(\sqrt{g_1(z)}-\sqrt{g_2(z)}\right)^2\nu(dz)}\sqrt{\int_{\Z}\left(\sqrt{g_2(z)}-\sqrt{g_3(z)}\right)^2\nu(dz)}.
\label{eq:dL2:2}
\end{align}
By Hölder's inequality, the third term in Equation ~\eqref{eq:dL2:2} satisfies
\begin{align}
&\sqrt{\int_{\Z}\left(\sqrt{g_1(z)}-\sqrt{g_2(z)}\right)^2\nu(dz)}\sqrt{\int_{\Z}\left(\sqrt{g_2(z)}-\sqrt{g_3(z)}\right)^2\nu(dz)} \nonumber \\
&\qquad  \geq \int_{\Z}\bigg \lvert \left(\sqrt{g_1(z)}-\sqrt{g_2(z)}\right)\left(\sqrt{g_2(z)}-\sqrt{g_3(z)}\right) \bigg \rvert \nu(dz) \nonumber \\
&\qquad = \int_{\Z}\bigg \lvert \sqrt{g_1(z)}-\sqrt{g_2(z)}\bigg \rvert \bigg \lvert \sqrt{g_2(z)}-\sqrt{g_3(z)} \bigg \rvert \nu(dz)
\label{ieq:dL2:1}
\end{align}
Plugging Equation ~\eqref{ieq:dL2:1} into Equation ~\eqref{eq:dL2:2} gives
\begin{align}
&\left(\sqrt{\int_{\Z}\left(\sqrt{g_1(z)}-\sqrt{g_2(z)}\right)^2\nu(dz)}+\sqrt{\int_{\Z}\left(\sqrt{g_2(z)}-\sqrt{g_3(z)}\right)^2\nu(dz)}\right)^2 \nonumber \\
&\qquad \geq  \frac{1}{2}\int_{\Z}\left(\sqrt{g_1(z)}-\sqrt{g_2(z)}\right)^2\nu(dz) + \frac{1}{2}\int_{\Z}\left(\sqrt{g_2(z)}-\sqrt{g_3(z)}\right)^2\nu(dz) \nonumber \\
&\qquad \qquad + \int_{\Z}\bigg \lvert \sqrt{g_1(z)}-\sqrt{g_2(z)}\bigg \rvert \bigg \lvert \sqrt{g_2(z)}-\sqrt{g_3(z)} \bigg \rvert \nu(dz) \nonumber \\
&\qquad  = \frac{1}{2}\int_{\Z}\left(\bigg \lvert \sqrt{g_1(z)}-\sqrt{g_2(z)}\bigg \rvert + \bigg \lvert \sqrt{g_2(z)}-\sqrt{g_3(z)} \bigg \rvert \right)^2\nu(dz).
\label{ieq:dL2:2}
\end{align}
According to the triangle inequality of the Euclidean distance on $\reals$, for every $z \in \Z$, we have
\begin{equation}
0 \leq \bigg \lvert \sqrt{g_1(z)}-\sqrt{g_3(z)} \bigg \rvert \leq \bigg \lvert \sqrt{g_1(z)}-\sqrt{g_2(z)}\bigg \rvert + \bigg \lvert \sqrt{g_2(z)}-\sqrt{g_3(z)} \bigg \rvert.
\label{ieq:dL2:3}
\end{equation}
Combining Equations~\eqref{eq:dL2:1},~\eqref{ieq:dL2:2} and~\eqref{ieq:dL2:3} completes the proof.
\end{proof}

Next, we state Proposition~\ref{thm:lipschitz}, which establishes the global Lipschitz continuity of the map $\tilde{F}_k$.
\begin{prop}\label{thm:lipschitz}
(\textbf{Global Lipschitz continuity of $\tilde{F}_k$}) Given data $y_k$, the function $\tilde{F}_k: (\bar{\mathcal{P}}_k(\bar{\X}), d) \rightarrow (\tilde{\mathcal{P}}(\bar{\X}), \tilde{d})$ is globally Lipschitz continuous, i.e., $\tilde{F}_k$ satisfies
\begin{equation*}
\tilde{d}(\tilde{F}_k(\mu),\tilde{F}_k(\mu^\prime)) \leq \tilde{K}(y_k)d(\mu,\mu^\prime), \quad \forall \mu,\mu^\prime \in \bar{\mathcal{P}}_k(\bar{\X}),
\end{equation*}
when $d$ is $d_{TV}$ and $\tilde{d}$ is $\tilde{d}_{TV}$, or when $d$ is $d_H$ and $\tilde{d}$ is $\tilde{d}_H$, under the assumption for the total variation and Hellinger distance in Theorem~\ref{thm:pointwise_lipschitz}.
\end{prop} 

The Lipschitz constants $\tilde{K}(y_k)$ corresponding to different distances and different problems are provided in Table~\ref{table:unnormalized}.

\begin{table}
\caption{(Global) Lipschitz constant $\tilde{K}(y_k)$ of function $\tilde{F}_k$ for different problems and distances of probability measures.}
\centering
\begin{tabularx}{0.9\linewidth}{ 
  | >{\centering\arraybackslash}X 
  | >{\centering\arraybackslash}X 
  | >{\centering\arraybackslash}X |>{\centering\arraybackslash}X | }
 \hline
 Problem name or distance $d$ and $\tilde{d}$& Inverse problems & State estimation &Parameter-state learning \\
 \hline
 \rule{0pt}{11pt} $(d_{TV},\tilde{d}_{TV})$ & $C_h(y_k)$ & $C_{Th}(y_k;k)$ & $\tilde{C}_{Th}(y_k;k)$\\
 \hline
 \rule{0pt}{16pt}$(d_{H},\tilde{d}_{H})$ & $\sqrt{C_h(y_k)}$ & $\sqrt{C_{Th}(y_k;k)}$ &$\sqrt{\tilde{C}_{Th}(y_k;k)}$ \\
\hline
\end{tabularx}
\label{table:unnormalized}
\end{table}

\begin{remark}
The conclusion in Proposition~\ref{thm:lipschitz} under the total variation distance is not used for the proof of Theorem~\ref{thm:pointwise_lipschitz}. We include it in Proposition~\ref{thm:lipschitz} for completeness.
\end{remark}

Proposition~\ref{thm:lipschitz} is a conclusion of Lemmas~\ref{prop:ip:unnormalized_bound_tv}---\ref{prop:ps:unnormalized_bound_H}. 
\begin{lemma}
In inverse problems, under Assumption IP.1, $\forall \mu, \mu^\prime \in \bar{\mathcal{P}}_k(\X)$, $\tilde{F}_k(\mu)$ and $\tilde{F}_k(\mu^\prime)$ satisfy
\begin{equation*}
\tilde{d}_{TV}(\tilde{F}_k(\mu),\tilde{F}_k(\mu^\prime)) \leq C_h(y_k)d_{TV}(\mu,\mu^\prime).
\end{equation*}
\label{prop:ip:unnormalized_bound_tv}
\end{lemma}

\begin{proof}
Let $\nu$ be a $\sigma$-finite measure defined on the measurable space $(\X,\mathcal{B}(\X))$. Assume that $\mu \ll \nu$ and $\mu^\prime \ll \nu$. According to the definition of $\tilde{F}_k$, $\tilde{F}_k(\mu)$ and $\tilde{F}_k(\mu^\prime)$ are equivalent to $F_k(\mu)$ and $F_k(\mu^\prime)$, respectively. Since $F_k(\mu)$ and $F_k(\mu^\prime)$ satisfy $F_k(\mu) \ll \mu$ and $F_k(\mu^\prime) \ll \mu^\prime$, we have $F_k(\mu) \ll \nu, F_k(\mu^\prime)\ll \nu$ and $\tilde{F}_k(\mu) \ll \nu, \tilde{F}_k(\mu^\prime)\ll \nu$. The Radon–Nikodym
derivatives $\frac{d\tilde{F}_k\mu}{d\nu}$ and $\frac{d\tilde{F}_k\mu^\prime}{d\nu}$ can be written as
\begin{align*}
   \frac{d\tilde{F}_k\mu}{d\nu}(x) &= \frac{d\tilde{F}_k\mu}{d\mu}(x) \cdot \frac{d\mu}{d\nu}(x) = h(y_k,x)\cdot \frac{d\mu}{d\nu}(x), \\
   \frac{d\tilde{F}_k\mu^\prime}{d\nu}(x) &= \frac{d\tilde{F}_k\mu^\prime}{d\mu^\prime}(x) \cdot \frac{d\mu^\prime}{d\nu}(x) = h(y_k,x)\cdot \frac{d\mu^\prime}{d\nu}(x).
\end{align*}
By the definition of distance $\tilde{d}_{TV}$, $\tilde{d}_{TV}(\tilde{F}_k(\mu),\tilde{F}_k(\mu^\prime))$ can be written as
\begin{align*}
\tilde{d}_{TV}(\tilde{F}_k(\mu),\tilde{F}_k(\mu^\prime)) 
=& \frac{1}{2}\int_{\X}\bigg \lvert h(y_k,x)\frac{d\mu}{d\nu}(x) - h(y_k,x)\frac{d\mu^\prime}{d\nu}(x) \bigg \rvert \nu(dx) \\
\leq & \frac{1}{2}C_h(y_k)\int_{\X}\bigg \lvert \frac{d\mu}{d\nu}(x) - \frac{d\mu^\prime}{d\nu}(x) \bigg \rvert \nu(dx) \\
=&C_h(y_k)d_{TV}(\mu,\mu^\prime).
\end{align*}
\end{proof}

\begin{lemma}
In state estimation problems, under Assumption SE.1, $\forall \mu, \mu^\prime \in \bar{\mathcal{P}}_k(\X)$, $\tilde{F}_k(\mu)$ and $\tilde{F}_k(\mu^\prime)$ satisfy
\begin{equation*}
\tilde{d}_{TV}(\tilde{F}_k(\mu),\tilde{F}_k(\mu^\prime)) \leq C_{Th}(y_k;k)d_{TV}(\mu,\mu^\prime).
\end{equation*}
\label{prop:se:unnormalized_bound_tv}
\end{lemma}
\begin{proof}
Let $\nu$ be a $\sigma$-finite measure defined on the measurable space $(\X,\mathcal{B}(\X))$ such that $\mu \ll \nu$ and $\mu^\prime \ll \nu$. Let $\pi(x)$ and $\pi^\prime(x)$ denote the Radon-Nikodym derivatives $\frac{d\mu}{d\nu}(x)$ and $\frac{d\mu^\prime}{d\nu}(x)$, respectively. By the definition of $\tilde{F}_k$ for state estimation problems, we can write $\tilde{d}_{TV}(\tilde{F}_k(\mu),\tilde{F}_k(\mu^\prime))$ as:
\begin{align}
&\tilde{d}_{TV}(\tilde{F}_k(\mu),\tilde{F}_k(\mu^\prime)) \nonumber \\
&\qquad=\frac{1}{2}\int_\mathcal{X}\bigg \lvert h_k(y_k, x)\int_\mathcal{X}T_k(x,x_{k-1})\mu(dx_{k-1}) - h_k(y_k ,x)\int_\mathcal{X}T_k(x,x_{k-1})\mu^\prime(dx_{k-1}) \bigg \rvert dx \nonumber \\
&\qquad=\frac{1}{2}\int_\mathcal{X}\bigg \lvert h_k(y_k, x)\int_\mathcal{X}T_k(x,x_{k-1})\pi(x_{k-1})\nu(dx_{k-1}) \nonumber \\
&\qquad \qquad \qquad - h_k(y_k ,x)\int_\mathcal{X}T_k(x,x_{k-1})\pi^\prime(x_{k-1})\nu(dx_{k-1}) \bigg \rvert dx \nonumber \\
& \qquad \leq \frac{1}{2}\int_\mathcal{X}h_k(y_k,x) \left(\int_\mathcal{X}T_k(x,x_{k-1})\lvert \pi(x_{k-1})-\pi^\prime(x_{k-1})\rvert \nu(dx_{k-1})\right)dx.
\label{ieq:unnormalized:tv:se:1}
\end{align}
By Tonelli's theorem, the right-hand side of Equation~\eqref{ieq:unnormalized:tv:se:1} can be written as
\begin{align*}
&\frac{1}{2}\int_\mathcal{X}h_k(y_k,x) \left(\int_\mathcal{X}T_k(x,x_{k-1})\lvert \pi(x_{k-1})-\pi^\prime(x_{k-1})\rvert \nu(dx_{k-1})\right)dx \nonumber \\
&\qquad = \frac{1}{2}\int_\mathcal{X}\left(\int_\mathcal{X}h_k(y_k,x)T_k(x, x_{k-1})dx\right)\lvert \pi(x_{k-1})-\pi^\prime(x_{k-1})\rvert \nu(dx_{k-1}) \nonumber \\
&\qquad \leq \frac{1}{2}C_{Th}(y_k;k)\int_\mathcal{X}\lvert \pi(x_{k-1})-\pi^\prime(x_{k-1})\rvert \nu(dx_{k-1}) \nonumber \\
&\qquad =C_{Th}(y_k;k)d_{TV}(\mu,\mu^\prime).
\end{align*}
\end{proof}

\begin{lemma}
In parameter-state estimation problems, under Assumption PS.1, $\forall \mu, \mu^\prime \in \bar{\mathcal{P}}_k(\X \times \W)$, $\tilde{F}_k(\mu)$ and $\tilde{F}_k(\mu^\prime)$ satisfy
\begin{equation*}
\tilde{d}_{TV}(\tilde{F}_k(\mu),\tilde{F}_k(\mu^\prime)) \leq \tilde{C}_{Th}(y_k;k)d_{TV}(\mu,\mu^\prime).
\end{equation*}
\label{prop:ps:unnormalized_bound_tv}
\end{lemma}

\begin{proof}
Let $\pi$ and $\pi^\prime$ denote the Lebesgue densities of $\mu$ and $\mu^\prime$, respectively. By the definition of $\tilde{F}_k$ for parameter-state estimation problems, we can write $\tilde{d}_{TV}(\tilde{F}_k(\mu),\tilde{F}_k(\mu^\prime))$ as
\begin{align*}
&\tilde{d}_{TV}(\tilde{F}_k(\mu),\tilde{F}_k(\mu^\prime)) \\
& \quad  = \frac{1}{2}\int_{\X \times \W} \Bigg \lvert h_k(y_k, x, w)\int_{\X} T_k(x, x_{k-1},w )\pi(x_{k-1},w) dx_{k-1} \\
&\qquad \qquad - h_k(y_k,x,w)\int_{\X} T_k(x, x_{k-1},w)\pi^\prime(x_{k-1},w)dx_{k-1}  \Bigg \rvert dxdw \nonumber \\
& \quad = \frac{1}{2}\int_{\X \times \W} h_k(y_k,x,w)\Bigg \lvert \int_{\X} T_k(x, x_{k-1},w)\left(\pi(x_{k-1},w)-\pi^\prime(x_{k-1},w)\right) dx_{k-1} \Bigg \rvert dxdw \\
& \quad \leq \frac{1}{2}\int_{\X \times \W} h_k(y_k,x,w) \left(\int_{\X} T_k(x, x_{k-1},w)\big \lvert \pi(x_{k-1},w)-\pi^\prime(x_{k-1},w)\big \rvert dx_{k-1} \right)dxdw.
\end{align*}
By Tonelli's theorem, we have
\begin{align*}
& \frac{1}{2}\int_{\X \times \W} h_k(y_k,x,w) \left(\int_{\X} T_k(x, x_{k-1},w)\big \lvert \pi(x_{k-1},w)-\pi^\prime(x_{k-1},w)\big \rvert dx_{k-1} \right)dxdw \\
& \quad= \frac{1}{2}\int_{\X\times\W} \left(\int_{\X} h_k(y_k,x,w) T_k(x, x_{k-1},w)dx\right)\big \lvert \pi(x_{k-1},w)-\pi^\prime(x_{k-1},w)\big \rvert dx_{k-1}dw \\
& \quad \leq \tilde{C}_{Th}(y_k;k)d_{TV}(\mu,\mu^\prime).
\end{align*}
\end{proof}

\begin{lemma}
In inverse problems, under Assumption IP.1, $\forall \mu, \mu^\prime \in \bar{\mathcal{P}}_k(\X)$, $\tilde{F}_k(\mu)$ and $\tilde{F}_k(\mu^\prime)$ satisfy
\begin{equation*}
\tilde{d}_{H}(\tilde{F}_k(\mu),\tilde{F}_k(\mu^\prime)) \leq \sqrt{C_{h}(y_k)}d_{H}(\mu,\mu^\prime).
\end{equation*}
\label{prop:ip:unnormalized_bound_H}
\end{lemma}
The proof of Lemma~\ref{prop:ip:unnormalized_bound_H} is analogous to that of Lemma~\ref{prop:ip:unnormalized_bound_tv}. 

\begin{lemma}
In state estimation problems, under Assumption SE.1, $\forall \mu, \mu^\prime \in \bar{\mathcal{P}}_k(\X)$, $\tilde{F}_k(\mu)$ and $\tilde{F}_k(\mu^\prime)$ satisfy
\begin{equation*}
\tilde{d}_{H}(\tilde{F}_k(\mu),\tilde{F}_k(\mu^\prime)) \leq \sqrt{C_{Th}(y_k;k)}d_{H}(\mu,\mu^\prime).
\end{equation*}
\label{prop:se:unnormalized_bound_H}
\end{lemma}

\begin{proof}
Let $\nu$ be a $\sigma$-finite measure defined on the measurable space $(\X,\mathcal{B}(\X))$ such that $\mu \ll \nu$ and $\mu^\prime \ll \nu$. Let $\pi(x)$ and $\pi^\prime(x)$ denote the Radon-Nikodym derivatives $\frac{d\mu}{d\nu}(x)$ and $\frac{d\mu^\prime}{d\nu}(x)$, respectively. By the definition of $\tilde{d}_H$ distance, we have
\begin{align}
&\tilde{d}_{H}^2(\tilde{F}_k(\mu),\tilde{F}_k(\mu^\prime)) \nonumber \\
&\qquad = \frac{1}{2}\int_{\X}\left({\sqrt{h_k(y_k,x)\int_{\X}T_k(x,x_{k-1})\mu(dx_{k-1})}}\right.\nonumber \\
&\qquad \qquad \qquad \qquad \qquad\left.{- \sqrt{h_k(y_k,x)\int_{\X}T_k(x,x_{k-1})\mu^\prime(dx_{k-1})}}\right)^2dx \nonumber \\
&\qquad = \frac{1}{2}\int_{\X}h_k(y_k,x)\left({\int_{\X}T_k(x,x_{k-1})\left(\pi(x_{k-1})+\pi^\prime(x_{k-1})\right)\nu(dx_{k-1})}\right. \nonumber \\
&\qquad \quad \left.{-2\sqrt{\int_{\X}T_k(x,x_{k-1})\pi(x_{k-1}) \nu(dx_{k-1})}\sqrt{\int_{\X}T_k(x,x_{k-1})\pi^\prime(x_{k-1}) \nu(dx_{k-1})}}\right)dx.
\label{eq:se:unnormalized:H:def}
\end{align}
By Hölder's inequality, we have
\begin{align}
&\sqrt{\int_{\X}T_k(x,x_{k-1})\pi(x_{k-1}) \nu(dx_{k-1})}\sqrt{\int_{\X}T_k(x,x_{k-1})\pi^\prime(x_{k-1}) \nu(dx_{k-1})}\nonumber \\
& \qquad \geq \int_{\X}T_k(x,x_{k-1})\sqrt{\pi(x_{k-1})} \sqrt{\pi^\prime(x_{k-1})}\nu(dx_{k-1}).
\label{ieq:se:unnormalized:H:1}
\end{align}
Substituting Equation~\eqref{ieq:se:unnormalized:H:1} into Equation~\eqref{eq:se:unnormalized:H:def} gives
\begin{align}
&\tilde{d}_{H}^2(\tilde{F}_k(\mu),\tilde{F}_k(\mu^\prime)) \leq \frac{1}{2}\int_{\X}h_k(y_k,x)\left({\int_{\X}T_k(x,x_{k-1})\left(\pi(x_{k-1})+\pi^\prime(x_{k-1})\right)\nu(dx_{k-1})}\right. \nonumber \\
&\qquad \qquad \qquad \qquad \qquad\left.{-2\int_{\X}T_k(x,x_{k-1})\sqrt{\pi(x_{k-1})}\sqrt{\pi^\prime(x_{k-1})} \nu(dx_{k-1})}\right)dx\nonumber \\
&\qquad \quad = \frac{1}{2}\int_{\X}h_k(y_k,x)\left(\int_{\X}T_k(x,x_{k-1})\left(\sqrt{\pi(x_{k-1})}-\sqrt{\pi^\prime(x_{k-1})}\right)^2\nu(dx_{k-1})\right)dx.
\label{ieq:se:unnormalized:H:new:1}
\end{align}
According to Tonelli's theorem, the right-hand side of Equation~\eqref{ieq:se:unnormalized:H:new:1} satisfies
\begin{align*}
&\frac{1}{2}\int_{\X}h_k(y_k,x)\left(\int_{\X}T_k(x,x_{k-1})\left(\sqrt{\pi(x_{k-1})}-\sqrt{\pi^\prime(x_{k-1})}\right)^2\nu(dx_{k-1})\right)dx \\
& \qquad  =\frac{1}{2}\int_{\X}\left(\int_{\X}h_k(y_k,x)T_k(x,x_{k-1})dx\right)\left(\sqrt{\pi(x_{k-1})}-\sqrt{\pi^\prime(x_{k-1})}\right)^2\nu(dx_{k-1}) \nonumber \\
& \qquad \leq \frac{1}{2}C_{Th}(y_k;k)\int_{\X}\left(\sqrt{\pi(x_{k-1})}-\sqrt{\pi^\prime(x_{k-1})}\right)^2\nu(dx_{k-1}) \nonumber \\
& \qquad = C_{Th}(y_k;k)d_H^2(\mu,\mu^\prime).
\end{align*}
\end{proof}

\begin{lemma}
In parameter-state estimation problems, under Assumption PS.1, $\forall \mu, \mu^\prime \in \bar{\mathcal{P}}_k(\X \times \W)$, $\tilde{F}_k(\mu)$ and $\tilde{F}_k(\mu^\prime))$ satisfy
\begin{equation*}
\tilde{d}_{H}(\tilde{F}_k(\mu),\tilde{F}_k(\mu^\prime)) \leq \sqrt{\tilde{C}_{Th}(y_k;k)}d_{H}(\mu,\mu^\prime).
\end{equation*}
\label{prop:ps:unnormalized_bound_H}
\end{lemma}

\begin{proof}
Let $\pi$ and $\pi^\prime$ denote the Lebesgue densities of $\mu$ and $\mu^\prime$, respectively. Following an argument analogous to that used in the proof of Lemma~\ref{prop:se:unnormalized_bound_H}, we can show that
\begin{align}
&\tilde{d}^2_{H}(\tilde{F}_k(\mu), \tilde{F}_k(\mu^\prime)) \nonumber \\
& \leq \quad \frac{1}{2}\int_{\X \times \W} h_k(y_k,x,w) \left(\int_{\X} T_k(x,x_{k-1},w)\left(\sqrt{\pi(x_{k-1},w)} - \sqrt{\pi^\prime(x_{k-1},w)} \right)^2 dx_{k-1}\right) \nonumber \\
&\qquad \qquad \cdot dxdw.
\label{ieq:ps:unnormalized:H:new:1}
\end{align}
An argument similar to that in Lemma~\ref{prop:ps:unnormalized_bound_tv} yields
\begin{align*}
&\int_{\X \times \W} h_k(y_k,x,w) \left(\int_{\X} T_k(x,x_{k-1},w)\left(\sqrt{\pi(x_{k-1},w)} - \sqrt{\pi^\prime(x_{k-1},w)} \right)^2 dx_{k-1}\right) \\
&\qquad \cdot dxdw \nonumber \\
& \quad = \int_{\X \times \W} \left(\int_{\X} h_k(y_k,x,w) T_k(x,x_{k-1},w) dx \right) \left(\sqrt{\pi(x_{k-1},w) } - \sqrt{\pi^\prime(x_{k-1},w)} \right)^2  \\ &\qquad \qquad \cdot dx_{k-1}dw \nonumber \\
&\qquad \leq \tilde{C}_{Th}(y_k;k)\int_{\X \times \W} \left(\sqrt{\pi(x_{k-1},w)} - \sqrt{\pi^\prime(x_{k-1},w)} \right)^2 dx_{k-1}dw\nonumber \\
&\qquad = 2\tilde{C}_{Th}(y_k;k)d_H^2(\mu,\mu^\prime).
\end{align*}
\end{proof}

\subsubsection{Proofs of Propositions~\ref{prop:ip:normalized_bound_H}---\ref{prop:ps:normalized_bound_H}}

Given Proposition~\ref{thm:lipschitz}, to prove Propositions~\ref{prop:ip:normalized_bound_H}---\ref{prop:ps:normalized_bound_H}, we first prove the following lemma.

\begin{lemma}
Let $\tilde{\mu}, \tilde{\mu}^\prime \in \tilde{\mathcal{P}}(\Z)$ be two measures such that $\forall A \in \mathcal{B}(\Z)$,
\begin{equation*}
\tilde{\mu}(A) = k\mu(A), \quad \tilde{\mu}^\prime(A) = k^\prime\mu^\prime(A), \quad \mu,\mu^\prime \in \mathcal{P}(\Z), \quad k,k^\prime \in \reals_{+}.
\end{equation*}
 Then $\lvert \sqrt{k} -\sqrt{k^\prime} \rvert$ satisfies
\begin{equation}
\bigg \lvert \sqrt{k} -\sqrt{k^\prime} \bigg \rvert \leq \sqrt{2}\tilde{d}_H(\tilde{\mu},\tilde{\mu}^\prime).
\end{equation}
The Hellinger distance between the two probability measures $\mu$ and $\mu^\prime$ satisfies
\begin{equation}
d_H\left(\mu,\mu^\prime \right) \leq \frac{2}{\sqrt{k}}\tilde{d}_H(\tilde{\mu},\tilde{\mu}^\prime).
\end{equation}
\label{lemma:unnormalized2normalized:H}
\end{lemma}

\begin{proof}
We first prove the statement $\lvert \sqrt{k} -\sqrt{k^\prime} \rvert \leq \sqrt{2}\tilde{d}_H(\tilde{\mu},\tilde{\mu}^\prime)$.
Let $\nu$ be a $\sigma$-finite measure defined on $(\Z,\mathcal{B}(\Z))$ such that $\mu \ll \nu$ and $\mu^\prime \ll \nu$. Then we have $\tilde{\mu} \ll \nu$ and $\tilde{\mu}^\prime \ll \nu$.
Given the definition of $\tilde{\mu}$ and $\tilde{\mu}^\prime$, we have
\begin{align*}
\int_{\Z}\frac{d\tilde{\mu}}{d\nu}(z)\nu(dz) =& \int_{\Z}\tilde{\mu}(dz) = k\int_{\Z}\mu(dz)=k, \\
\int_{\Z}\frac{d\tilde{\mu}^\prime}{d\nu}(z)\nu(dz) =& \int_{\Z}\tilde{\mu}^\prime(dz) = k^\prime\int_{\Z}\mu^\prime(dz)=k^\prime.
\end{align*}
Therefore, $(\sqrt{k} -\sqrt{k^\prime})^2$ can be written as
\begin{align}
&(\sqrt{k} -\sqrt{k^\prime})^2 =k + k ^\prime -2\sqrt{k}\sqrt{k^\prime} \nonumber \\
&\qquad \qquad =\int_{\Z}\frac{d\tilde{\mu}}{d\nu}(z)\nu(dz) +  \int_{\Z}\frac{d\tilde{\mu}^\prime}{d\nu}(z)\nu(dz) -2\sqrt{\int_{\Z}\frac{d\tilde{\mu}}{d\nu}(z)\nu(dz)}\sqrt{\int_{\Z}\frac{d\tilde{\mu}^\prime}{d\nu}(z)\nu(dz)}.
\label{eq:unnornalized2normalized:H:total}
\end{align}
By Hölder's inequality, the term $\sqrt{\int_{\Z}\frac{d\tilde{\mu}}{d\nu}(z)\nu(dz)}\sqrt{\int_{\Z}\frac{d\tilde{\mu}^\prime}{d\nu}(z)\nu(dz)}$ satisfies
\begin{equation}
\sqrt{\int_{\Z}\frac{d\tilde{\mu}}{d\nu}(z)\nu(dz)}\sqrt{\int_{\Z}\frac{d\tilde{\mu}^\prime}{d\nu}(z)\nu(dz)} \geq \int_{\Z}\sqrt{\frac{d\tilde{\mu}}{d\nu}(z)}\sqrt{\frac{d\tilde{\mu}^\prime}{d\nu}(z)}\nu(dz).
\label{ieq:unnornalized2normalized:H:1}
\end{equation}
Plugging Equation~\eqref{ieq:unnornalized2normalized:H:1} into Equation~\eqref{eq:unnornalized2normalized:H:total} gives
\begin{align*}
(\sqrt{k} -\sqrt{k^\prime})^2 \leq & \int_{\Z}\frac{d\tilde{\mu}}{d\nu}(z)\nu(dz) +  \int_{\Z}\frac{d\tilde{\mu}^\prime}{d\nu}(z)\nu(dz) -2\int_{\Z}\sqrt{\frac{d\tilde{\mu}}{d\nu}(z)}\sqrt{\frac{d\tilde{\mu}^\prime}{d\nu}(z)}\nu(dz) \\
=& \int_{\Z}\left(\sqrt{\frac{d\tilde{\mu}}{d\nu}(z)}-\sqrt{\frac{d\tilde{\mu}^\prime}{d\nu}(z)}\right)^2\nu(dz) \\
=& 2\tilde{d}_{H}^2(\tilde{\mu},\tilde{\mu}^\prime).
\end{align*}

Next, we prove the statement $d_H\left(\mu,\mu^\prime \right) \leq \frac{2}{\sqrt{k}}\tilde{d}_H(\tilde{\mu},\tilde{\mu}^\prime)$. \\
By the definition of the Hellinger distance, $d_H^2\left(\mu,\mu^\prime \right)$ can be written as
\begin{align*}
&d_H^2\left(\mu,\mu^\prime \right) = \frac{1}{2}\int_\mathcal{Z} \left(\sqrt{\frac{d\mu}{d\nu}(z)} - \sqrt{\frac{d\mu^\prime}{d\nu}(z)}\right)^2\nu(dz) \nonumber \\
& \quad =\frac{1}{2}\int_\mathcal{Z} \left(\sqrt{\frac{1}{k}}\sqrt{\frac{d\tilde{\mu}}{d\nu}(z)} - \sqrt{\frac{1}{k^\prime}}\sqrt{\frac{d\tilde{\mu}^\prime}{d\nu}(z)}\right)^2\nu(dz) \nonumber \\
&\quad =\frac{1}{2}\int_\mathcal{Z} \left(\sqrt{\frac{1}{k}}\sqrt{\frac{d\tilde{\mu}}{d\nu}(z)} - \sqrt{\frac{1}{k}}\sqrt{\frac{d\tilde{\mu}^\prime}{d\nu}(z)} + \sqrt{\frac{1}{k}}\sqrt{\frac{d\tilde{\mu}^\prime}{d\nu}(z)} -\sqrt{\frac{1}{k^\prime}}\sqrt{\frac{d\tilde{\mu}^\prime}{d\nu}(z)}\right)^2\nu(dz) \nonumber \\
&\quad \leq \int_\mathcal{Z} \left({\left(\sqrt{\frac{1}{k}}\sqrt{\frac{d\tilde{\mu}}{d\nu}(z)} - \sqrt{\frac{1}{k}}\sqrt{\frac{d\tilde{\mu}^\prime}{d\nu}(z)}\right)^2}\right. \\
&\qquad \qquad \left.{+ \left(\sqrt{\frac{1}{k}}\sqrt{\frac{d\tilde{\mu}^\prime}{d\nu}(z)} -\sqrt{\frac{1}{k^\prime}}\sqrt{\frac{d\tilde{\mu}^\prime}{d\nu}(z)}\right)^2}\right)\nu(dz) \nonumber \\
&\quad =\frac{1}{k}\int_\mathcal{Z} \left(\sqrt{\frac{d\tilde{\mu}}{d\nu}(z)} - \sqrt{\frac{d\tilde{\mu}^\prime}{d\nu}(z)}\right)^2\nu(dz) + \left(\sqrt{\frac{1}{k}}-\sqrt{\frac{1}{k^\prime}}\right)^2\int_{\Z}\frac{d\tilde{\mu}^\prime}{d\nu}(z)\nu(dz) \nonumber \\
& \quad =\frac{2}{k}\tilde{d}^2_H(\tilde{\mu},\tilde{\mu}^\prime)+ \frac{(\sqrt{k}-\sqrt{k^\prime})^2}{k} \leq \frac{4}{k}\tilde{d}^2_H(\tilde{\mu},\tilde{\mu}^\prime).
\end{align*}
\end{proof}

Lemma~\ref{lemma:unnormalized2normalized:H} is a generalization of Lemma~\ref{lemma:unnormalized2normalized:H:density} stated below. Lemma~\ref{lemma:unnormalized2normalized:H:density} is presented in Lemma 1 in \cite{tt_sequential} and the proofs of Proposition 5 and Theorem 1 in \cite{deep_composition}.

\begin{lemma}[
\cite{tt_sequential}, \cite{deep_composition}]
Let $\tilde{\mu}, \tilde{\mu}^\prime \in \tilde{\mathcal{P}}(\Z)$ be two measures, both absolutely continuous with respect to the Lebesgue measure. 
Let $\tilde{\pi}$ and $\tilde{\pi}^\prime$ represent the densities of $\tilde{\mu}$ and $\tilde{\mu}^\prime $ with respect to the Lebesgue measure, respectively. Assume that $\forall z \in \mathcal{Z}$, $\tilde{\pi}$ and $\tilde{\pi}^\prime$ satisfy 
\begin{equation*}
\tilde{\pi}(z) = k\pi(z), \quad \tilde{\pi}^\prime(z) = k^\prime\pi^\prime(z), \quad k,k^\prime \in \reals_{+},
\end{equation*}
where $\pi$ and $\pi^\prime$ are the densities of the probability measures $\mu$ and $\mu^\prime$ with respect to the Lebesgue measure, respectively.
Then $\lvert \sqrt{k} -\sqrt{k^\prime} \rvert$ satisfies
\begin{equation*}
\bigg \lvert \sqrt{k} -\sqrt{k^\prime} \bigg \rvert \leq \sqrt{2}\tilde{d}_H(\tilde{\mu},\tilde{\mu}^\prime).
\end{equation*}
The Hellinger distance between the two probability measures $\mu$ and $\mu^\prime$ satisfies
\begin{equation}
d_H\left(\mu,\mu^\prime \right) \leq \frac{2}{\sqrt{k}}\tilde{d}_H(\tilde{\mu},\tilde{\mu}^\prime).
\end{equation}
\label{lemma:unnormalized2normalized:H:density}
\end{lemma}

As shown in the proof of Lemma~\ref{lemma:unnormalized2normalized:H}, the proof of the conclusion $d_H\left(\mu,\mu^\prime \right) \leq \frac{2}{\sqrt{k}}\tilde{d}_H(\tilde{\mu},\tilde{\mu}^\prime)$
is based on the conclusion $\bigg \lvert \sqrt{k} -\sqrt{k^\prime} \bigg \rvert \leq \sqrt{2}\tilde{d}_H(\tilde{\mu},\tilde{\mu}^\prime)$. Our proof of the conclusion $\bigg \lvert \sqrt{k} -\sqrt{k^\prime} \bigg \rvert \leq \sqrt{2}\tilde{d}_H(\tilde{\mu},\tilde{\mu}^\prime)$ differs from the proof in \cite{deep_composition} and is more concise. Given this result, our proof of the conclusion $d_H\left(\mu,\mu^\prime \right) \leq \frac{2}{\sqrt{k}}\tilde{d}_H(\tilde{\mu},\tilde{\mu}^\prime)$ is analogous to the proof in \cite{deep_composition}. 
\\

Given the definition of $F_k$, directly combining Proposition~\ref{thm:lipschitz} and Lemma~\ref{lemma:unnormalized2normalized:H} proves Propositions~\ref{prop:ip:normalized_bound_H}---\ref{prop:ps:normalized_bound_H}. The structure of this argument is similar to that of an argument used in the proof of Theorem 8 in \cite{tt_sequential}, although the conclusions differ.

\subsection{Proof of Theorem~\ref{thm:pointwise_lipschitz} under $1$-Wasserstein distance}

Theorem~\ref{thm:pointwise_lipschitz} under the $1$-Wasserstein distance is a conclusion of Propositions~\ref{prop:ip:normalized_bound_W}---\ref{prop:ps:normalized_bound_W}. To prove Propositions~\ref{prop:ip:normalized_bound_W}---\ref{prop:ps:normalized_bound_W}, we first present Proposition~\ref{thm:lipschitz_evidence}, which is a supporting proposition for the proof of Theorem~\ref{thm:pointwise_lipschitz} under the $1$-Wasserstein distance.

\subsubsection{Supporting proposition for the proof of Theorem~\ref{thm:pointwise_lipschitz}: Proposition~\ref{thm:lipschitz_evidence}}
\begin{prop}\label{thm:lipschitz_evidence}
(\textbf{Global Lipschitz continuity of evidence function $Z_k$}) Given the data $y_k$, the evidence function $Z_k: (\bar{\mathcal{P}}_k(\bar{\X}) \bigcap \mathcal{P}_1(\bar{\X}), W_1) \rightarrow (\reals_{+},\lvert \cdot \rvert)$ is globally Lipschitz continuous, i.e., $Z_k$ satisfies
\begin{equation*}
\lvert Z_k(\mu) - Z_k(\mu^\prime) \rvert \leq K_Z(y_k)W_1(\mu,\mu^\prime), \quad \forall \mu,\mu^\prime \in \bar{\mathcal{P}}_k(\bar{\X}) \bigcap \mathcal{P}_1(\bar{\X}),
\end{equation*}
if 
\begin{itemize}
\item in inverse problems, Assumption IP.2 holds (Theorem 15, \cite{lipschitz_stability_ip}); or
\item in state-parameter estimation problems, Assumption SE.2 holds; or
\item in parameter-state estimation problems, Assumption PS.2 holds.
\end{itemize}
The Lipschitz constant $K_Z(y_k)$ is given in Table ~\ref{table:Z}.
\begin{table}[h!]
\caption{(Global) Lipschitz constant $K_Z(y_k)$ of evidence function $Z_k$ in different problems.}
\centering
\begin{tabularx}{1.0\linewidth}{ 
  | >{\centering\arraybackslash}X 
  | >{\centering\arraybackslash}X |>{\centering\arraybackslash}X | }
 \hline
 Inverse problems & State estimation &Parameter-state learning \\
 \hline
$\lVert h \rVert _{\text{Lip}}(y_k)$ & $C_{Th}^*(y_k;k)$ & $\tilde{C}_{Th}^*(y_k;k)$\\
 \hline
\end{tabularx}
\label{table:Z}
\end{table}
\end{prop} 
Proposition~\ref{thm:lipschitz_evidence} is a conclusion of Lemmas~\ref{lemma:evidence_diff_bound:ip}---\ref{lemma:evidence_diff_bound:ps}, as shown below.

\begin{lemma}[Theorem 15, \cite{lipschitz_stability_ip}](\textbf{Upper-bounded evidence difference in inverse problems})
\label{lemma:evidence_diff_bound:ip}
For inverse problems, given the data $y_k$, suppose Assumption IP.2 holds. For any $\mu,\mu^\prime \in \bar{\mathcal{P}}_k(\X) \bigcap \mathcal{P}_1(\X)$, the difference between the two evidences, $Z_k(\mu)$ and $Z_k(\mu^\prime)$, is bounded by:
\begin{equation*}
\big \lvert Z_k(\mu) - Z_k(\mu^\prime) \big \rvert \leq \lVert h \rVert _{\text{Lip}}(y_k)W_1(\mu,\mu^\prime).
\end{equation*}
\end{lemma}
For completeness, we restate Lemma~\ref{lemma:evidence_diff_bound:ip} and provide its full proof below.
\begin{proof}
The difference between the two evidences $Z_k(\mu)$ and $Z_k(\mu^\prime)$ can be written as
\begin{align*}
\big \lvert Z_k(\mu) -Z_k(\mu^\prime) \big \rvert =& \bigg \lvert \int_{\mathcal{X}}h(y_k, x)\mu(dx)-\int_{\mathcal{X}}h(y_k, x)\mu^\prime(dx) \bigg \rvert.
\end{align*}
Given data $y_k$, define a function $\tilde{h}$ as
\begin{equation*}
\tilde{h}(x) := \frac{h(y_k,x)}{\lVert h \rVert_{\text{Lip}}(y_k)}.
\end{equation*}
Then $\tilde{h}$ satisfies
\begin{align}
\sup_{(x_1,x_2) \in \mathcal{X}\times \mathcal{X}}\frac{\big \lvert \tilde{h}(x_1)-\tilde{h}(x_2)\big \rvert}{d_{\X}(x_1,x_2)} =& \frac{1}{\lVert h \rVert_{\text{Lip}}(y_k)} \sup_{(x_1,x_2) \in \mathcal{X}\times \mathcal{X}}\frac{\big \lvert h(y_k,x_1) - h(y_k,x_2) \big \rvert}{d_{\X}(x_1,x_2)} \nonumber \\
=& \frac{\lVert h \rVert_{\text{Lip}}(y_k)}{\lVert h \rVert_{\text{Lip}}(y_k)}=1.
\label{eq:lipschitz_const_h_tilde}
\end{align}
By the definition of $\tilde{h}$, we can re-write $\big \lvert Z_k(\mu) - Z_k(\mu^\prime) \big \rvert$ as 
\begin{align}
\big \lvert Z_k(\mu) - Z_k(\mu^\prime) \big \rvert &=\lVert h \rVert_{\text{Lip}}(y_k)\bigg \lvert \int_{\mathcal{X}}\tilde{h}( x)\mu(dx)-\int_{\mathcal{X}}\tilde{h}( x)\mu^\prime(dx)\bigg \rvert \nonumber \\
&\leq \lVert h \rVert_{\text{Lip}}(y_k)\sup_{f: \mathcal{X} \rightarrow \mathbb{R}, \lVert f \rVert_{\text{Lip}} \leq 1}\bigg \lvert \int_{\mathcal{X}}f( x)\mu(dx)-\int_{\mathcal{X}}f(x)\mu^\prime(dx) \bigg \rvert \nonumber \\
&= \lVert h \rVert_{\text{Lip}}(y_k)W_1(\mu,\mu^\prime).
\label{ieq:ip:W:term2:part1:evidence}
\end{align}
The first inequality in Equation~\eqref{ieq:ip:W:term2:part1:evidence} is obtained because it has been shown in Equation~\eqref{eq:lipschitz_const_h_tilde} that the best Lipschitz constant of function $\tilde{h}$ is 1.
\end{proof}

By extending the proof of Lemma~\ref{lemma:evidence_diff_bound:ip} to state estimation and parameter-state estimation problems, we obtain Lemmas~\ref{lemma:evidence_diff_bound:se} and~\ref{lemma:evidence_diff_bound:ps}, as stated below.
\begin{lemma}(\textbf{Upper-bounded evidence difference in state estimation problems})
\label{lemma:evidence_diff_bound:se}
For state estimation problems, at time step $k$, given the data $y_k$, suppose Assumption SE.2 holds. For any $\mu,\mu^\prime \in \bar{\mathcal{P}}_k(\X) \bigcap \mathcal{P}_1(\X )$, the difference between the two evidences, $Z_k(\mu)$ and $Z_k(\mu^\prime)$, is bounded by:
\begin{equation*}
\big \lvert Z_k(\mu) - Z_k(\mu^\prime) \big \rvert \leq C_{Th}^*(y_k;k)W_1(\mu,\mu^\prime).
\end{equation*}
\end{lemma}
\begin{proof}
Let $\nu$ be a $\sigma$-finite measure defined on the measurable space $(\X,\mathcal{B}(\X))$ such that $\mu \ll \nu$ and $\mu^\prime \ll \nu$. Let $\pi(x)$ and $\pi^\prime(x)$ denote the Radon-Nikodym derivatives $\frac{d\mu}{d\nu}(x)$ and $\frac{d\mu^\prime}{d\nu}(x)$, respectively. By the definition of evidence function $Z_k$, we can write $\lvert Z_k(\mu) - Z_k(\mu^\prime) \rvert$ as
\begin{align}
\lvert Z_k(\mu) - Z_k(\mu^\prime) \rvert & =
\bigg \lvert \int_\mathcal{X}\tilde{F}_k\mu(dx) - \int_\mathcal{X}\tilde{F}_k\mu^\prime(dx) \bigg \rvert \nonumber \\
&=\bigg \lvert \int_\mathcal{X}h_k(y_k, x)\left(\int_\mathcal{X}T_k(x,x_{k-1})\pi(x_{k-1})\nu(dx_{k-1})\right)dx  \nonumber \\
&\qquad - \int_\mathcal{X}h_k(y_k, x)\left(\int_\mathcal{X}T_k(x,x_{k-1})\pi^\prime(x_{k-1})\nu(dx_{k-1})\right)dx \bigg \rvert 
\label{ieq:evidence_diff:se:new:1}
\end{align}
According to Tonelli's theorem, the right-hand side of Equation~\eqref{ieq:evidence_diff:se:new:1} can be written as
\begin{align}
&\bigg \lvert \int_\mathcal{X}h_k(y_k, x)\left(\int_\mathcal{X}T_k(x,x_{k-1})\pi(x_{k-1})\nu(dx_{k-1})\right)dx \nonumber \\
&\qquad \qquad - \int_\mathcal{X}h_k(y_k, x)\left(\int_\mathcal{X}T_k(x,x_{k-1})\pi^\prime(x_{k-1})\nu(dx_{k-1})\right)dx \bigg \rvert \nonumber \\
& \qquad =\bigg \lvert \int_\mathcal{X}\left(\int_\mathcal{X}h_k(y_k, x)T_k(x,x_{k-1})dx\right)\pi(x_{k-1})\nu(dx_{k-1}) \nonumber \\
& \qquad \qquad \qquad - \int_\mathcal{X}\left(\int_\mathcal{X}h_k(y_k, x)T_k(x,x_{k-1})dx\right)\pi^\prime(x_{k-1})\nu(dx_{k-1}) \bigg \rvert. 
\end{align}
Given the data $y_k$, the expression $\int_\mathcal{X}h_k(y_k, x)T_k(x,x_{k-1})dx$ is a function of $x_{k-1}$. We denote this function as $g$:
\begin{equation*}
g(x_{k-1}) := \int_\mathcal{X}h_k(y_k,x)T_k(x,x_{k-1})dx.
\end{equation*}
We now show that $g$ is Lipschitz continuous and its best Lipschitz constant is no greater than $C_{Th}^*(y_k;k)$. \\
For two arbitrary points $x_{k-1}, x_{k-1}^\prime \in \mathcal{X}$, we have
\begin{align*}
\lvert g(x_{k-1}) - g(x_{k-1}^\prime) \rvert 
=& \bigg \lvert \int_\mathcal{X}h_k(y_k,x)\left(T_k(x, x_{k-1})-T_k(x,x_{k-1}^\prime)\right)dx \bigg \rvert \nonumber \\
\leq & \int_\mathcal{X}h_k(y_k,x) \bigg \lvert T_k(x, x_{k-1})-T_k(x,x_{k-1}^\prime)\bigg \rvert dx \nonumber \\
\leq & d_{\X}(x_{k-1},x_{k-1}^\prime)\int_\mathcal{X}h_k(y_k,x)T_{\text{Lip}}(x;k)dx\nonumber \\
=&C_{Th}^*(y_k;k)d_{\X}(x_{k-1},x_{k-1}^\prime).
\end{align*}
It follows that 
\begin{equation*}
\lVert g \rVert_{\text{Lip}} = \sup_{x_{k-1},x_{k-1}^\prime \in \mathcal{X},x_{k-1} \neq x_{k-1}^\prime}\frac{\lvert g(x_{k-1}) - g(x_{k-1}^\prime) \rvert}{d_{\X}(x_{k-1},x_{k-1}^\prime)} \leq C_{Th}^*(y_k;k).
\end{equation*}
We proceed by employing the same proof strategy used in the proof of Lemma~\ref{lemma:evidence_diff_bound:ip}. Define a function $\tilde{g}$ as
\begin{equation*}
\tilde{g}(x) := \frac{g(x)}{C_{Th}^*(y_k;k)}.
\end{equation*}
Then $\tilde{g}$ satisfies
\begin{align*}
\lVert \tilde{g} \rVert_{\text{Lip}} =&
\sup_{x,x^\prime \in \mathcal{X},x \neq x^\prime}\frac{\lvert \tilde{g}(x) - \tilde{g}(x^\prime) \rvert}{d_{\X}(x, x^\prime)}
=\frac{1}{C_{Th}^*(y_k;k)}\sup_{x,x^\prime \in \mathcal{X},x \neq x^\prime}\frac{\lvert g(x) - g(x^\prime) \rvert}{d_{\X}(x, x^\prime)}
\leq 1.
\end{align*}
The term $\lvert Z_k(\mu) - Z_k(\mu^\prime) \rvert$ can be written as:
\begin{align*}
&\lvert Z_k(\mu) - Z_k(\mu^\prime) \rvert \\
&\qquad = C_{Th}^*(y_k;k) \bigg \lvert \int_{\mathcal{X}}\tilde{g}(x_{k-1})\pi(x_{k-1})\nu(dx_{k-1})-\int_{\mathcal{X}}\tilde{g}(x_{k-1})\pi^\prime(x_{k-1})\nu(dx_{k-1}) \bigg \rvert.
\end{align*} 
As we have $\lVert \tilde{g} \rVert_{\text{Lip}} \leq 1$, the following condition is satisfied:
\begin{align*}
&\bigg \lvert \int_{\mathcal{X}}\tilde{g}(x_{k-1})\pi(x_{k-1})\nu(dx_{k-1})-\int_{\mathcal{X}}\tilde{g}(x_{k-1})\pi^\prime(x_{k-1})\nu(dx_{k-1}) \bigg \rvert \nonumber \\
&\qquad \leq \sup_{f: \mathcal{X} \rightarrow \mathbb{R},\lVert f \rVert_{\text{Lip}}\leq1}\bigg \lvert \int_{\mathcal{X}}f(x_{k-1})\pi(x_{k-1})\nu(dx_{k-1})-\int_{\mathcal{X}}f(x_{k-1})\pi^\prime(x_{k-1})\nu(dx_{k-1}) \bigg \rvert \nonumber \\
& \qquad = W_1(\mu,\mu^\prime).
\end{align*}
\end{proof}

\begin{lemma}(\textbf{Upper-bounded evidence difference in parameter-state estimation problems})
\label{lemma:evidence_diff_bound:ps}
For parameter-state estimation problems, at time step $k$, given the data $y_k$, suppose Assumption PS.2 holds. For any $\mu,\mu^\prime \in \bar{\mathcal{P}}_k(\X \times \W) \bigcap \mathcal{P}_1(\X \times \W)$, the difference between the two evidences, $Z_k(\mu)$ and $Z_k(\mu^\prime)$, is bounded by:
\begin{equation*}
\big \lvert Z_k(\mu) - Z_k(\mu^\prime) \big \rvert \leq \tilde{C}_{Th}^*(y_k;k)W_1(\mu,\mu^\prime). 
\end{equation*}
\end{lemma}

\begin{proof}
Let $\pi$ and $\pi^\prime$ be the densities of $\mu$ and $\mu^\prime$ with respect to the Lebesgue measure, respectively. By the definition of evidence function $Z_k$, we can write $\lvert Z_k(\mu) - Z_k(\mu^\prime) \rvert$ as
\begin{align}
&\bigg \lvert Z_k(\mu) - Z_k(\mu^\prime) \bigg \rvert \nonumber \\
& \qquad =\bigg \lvert \int_{\mathcal{X}\times \W}h_k(y_k,x,w)\left(\int_{\mathcal{X}}T_k(x,x_{k-1},w)\pi(x_{k-1},w)dx_{k-1}\right)dxdw \nonumber \\
&\qquad \qquad - \int_{\mathcal{X}\times \W}h_k(y_k,x,w)\left(\int_{\mathcal{X}}T_k(x,x_{k-1},w)\pi^\prime(x_{k-1},w)dx_{k-1}\right)dxdw \bigg \rvert \nonumber \\
& \qquad =\bigg \lvert \int_{\mathcal{X}\times \W}\left(\int_{\mathcal{X}}h_k(y_k,x,w)T_k(x,x_{k-1},w)dx\right)\pi(x_{k-1},w)dx_{k-1}dw \nonumber \\
&\qquad \qquad - \int_{\mathcal{X}\times \W}\left(\int_{\mathcal{X}}h_k(y_k,x,w)T_k(x,x_{k-1},w)dx\right)\pi^\prime(x_{k-1},w)dx_{k-1}dw \bigg \rvert.
\label{eq:lemma:Z:W:ps}
\end{align}
The second equality in Equation~\eqref{eq:lemma:Z:W:ps} is obtained by Tonelli's thoerem. 
The expression $\int_{\mathcal{X}}h_k(y_k,x,w)T_k(x,x_{k-1},w)dx$ in Equation~\eqref{eq:lemma:Z:W:ps} is a function of $(x_{k-1},w)$. We denote this function as $g$:
\begin{equation*}
g(x_{k-1},w) := \int_{\mathcal{X}}h_k(y_k,x,w)T_k(x,x_{k-1},w)dx.
\end{equation*} \\
We now prove that function $g$ is Lipschitz continuous and its best Lipschitz constant is no greater than $\tilde{C}^*_{Th}(y_k;k)$. \\
For two arbitrary points $(x_{k-1},w)$ and $(x_{k-1}^\prime,w^\prime)$ in space $\mathcal{X}\times \W$, we have
\begin{align*}
&\lvert g(x_{k-1},w) - g(x_{k-1}^\prime,w^\prime) \rvert \nonumber \\
& \qquad = \bigg \lvert \int_{\mathcal{X}}h_k(y_k,x,w)T_k(x, x_{k-1},w)dx - \int_{\mathcal{X}}h_k(y_k,x,w^\prime)T_k(x, x_{k-1}^\prime,w^\prime)dx \bigg \rvert \nonumber \\
 & \qquad \leq \int_{\mathcal{X}}\bigg \lvert h_k(y_k,x,w)T_k(x, x_{k-1},w) - h_k(y_k,x,w^\prime)T_k(x, x_{k-1}^\prime,w^\prime) \bigg \rvert dx \nonumber \\
& \qquad \leq d_{\X \times \W}((x_{k-1},w),(x_{k-1}^\prime,w^\prime))\int_{\mathcal{X}}\tilde{f}_{\text{Lip}}(x;k)dx \nonumber \\
& \qquad = \tilde{C}_{Th}^*(y_k;k)d_{\X \times \W}((x_{k-1},w),(x_{k-1}^\prime,w^\prime)).
\end{align*}
As $(x_{k-1},w)$ and $(x_{k-1}^\prime,w^\prime)$ are arbitrary, we have
\begin{equation*}
\sup_{ (x_{k-1},w), (x_{k-1}^\prime,w^\prime) \in \mathcal{X}\times\W, (x_{k-1},w) \neq (x_{k-1}^\prime,w^\prime)}\frac{\lvert g(x_{k-1},w) - g(x_{k-1}^\prime,w^\prime) \rvert}{d_{\X \times \W}((x_{k-1},w),(x_{k-1}^\prime,w^\prime))} \leq \tilde{C}_{Th}^*(y_k;k).
\end{equation*}
By following an argument analogous to those used in the proofs of Lemmas~\ref{lemma:evidence_diff_bound:ip} and~\ref{lemma:evidence_diff_bound:se}, we can show that
\begin{align*}
\lvert Z_k(\mu) - Z_k(\mu^\prime) \rvert 
& \leq\tilde{C}_{Th}^*(y_k;k)\sup_{f: \mathcal{X}\times \W \rightarrow \mathbb{R}, \lVert f \rVert_{\text{Lip}} \leq1}\bigg \lvert \int_{\mathcal{X}\times \W}f(x_{k-1},w)\pi(x_{k-1},w)dx_{k-1}dw \\
&\qquad \qquad \qquad \qquad \qquad -\int_{\mathcal{X}\times \W}f(x_{k-1},w)\pi^\prime(x_{k-1},w)dx_{k-1}dw \bigg \rvert \nonumber \\
&  =\tilde{C}_{Th}^*(y_k;k)W_1(\mu,\mu^\prime).
\end{align*}
\end{proof}

\subsubsection{Theorem~\ref{thm:pointwise_lipschitz} under $1$-Wasserstein distance: Propositions~\ref{prop:ip:normalized_bound_W}---\ref{prop:ps:normalized_bound_W} and their proofs}
Theorem~\ref{thm:pointwise_lipschitz} under the $1$-Wasserstein distance is a conclusion of Propositions~\ref{prop:ip:normalized_bound_W}---\ref{prop:ps:normalized_bound_W}.
\begin{prop}
In inverse problems, given the data $y_k$, under Assumptions IP.1, IP.2, and IP-SE-PS.4, $\forall \mu, \mu^\prime \in \bar{\mathcal{P}}_k(\X) \bigcap \mathcal{P}_1(\X)$, the posteriors $F_k(\mu)$ and $F_k(\mu^\prime)$ satisfy
\begin{equation*}
W_1(F_k(\mu),F_k(\mu^\prime)) \leq \frac{C_h(y_k)+2D\lVert h\rVert_{\text{Lip}}(y_k)}{Z_k(\mu)}W_1(\mu,\mu^\prime).
\end{equation*}
\label{prop:ip:normalized_bound_W}
\end{prop}

\begin{proof}
By Kantorovich-Rubinstein duality, the $1$-Wasserstein distance between the two posterior probability measures $F_k(\mu)$ and $F_k(\mu^\prime)$ can be written as
\begin{align*}
W_1(F_k(\mu),F_k(\mu^\prime)) =& \sup_{f:\mathcal{X} \rightarrow \mathbb{R}, \lVert f \rVert_{\text{Lip}} \leq 1}\bigg \lvert \int_{\mathcal{X}}f(x)F_k\mu(dx) - \int_{\mathcal{X}}f(x)F_k\mu^\prime(dx) \bigg \rvert \nonumber \\
=& \sup_{f:\mathcal{X} \rightarrow \mathbb{R}, \lVert f \rVert_{\text{Lip}} \leq 1}\bigg \lvert \int_{\mathcal{X}}f(x)\frac{h(y_k,x)}{Z_k(\mu)}\mu(dx) - \int_{\mathcal{X}}f(x)\frac{h(y_k,x)}{Z_k(\mu^\prime)}\mu^\prime(dx) \bigg \rvert.
\end{align*}
We begin by following the proof strategy of Theorem 15 in \cite{lipschitz_stability_ip}, and first show that $W_1(F_k(\mu),F_k(\mu^\prime))$ can be expressed as
\begin{align}
&W_1(F_k(\mu),F_k(\mu^\prime)) \nonumber \\
& \qquad = \sup_{f:\mathcal{X} \rightarrow \mathbb{R}, \lVert f \rVert_{\text{Lip}} \leq 1, f(x_0)=0}\bigg \lvert \int_{\mathcal{X}}f(x)\frac{h(y_k,x)}{Z_k(\mu)}\mu(dx) - \int_{\mathcal{X}}f(x)\frac{h(y_k,x)}{Z_k(\mu^\prime)}\mu^\prime(dx) \bigg \rvert.
\label{eq:ip:W:re-write}
\end{align}
The proof for Equation~\eqref{eq:ip:W:re-write} is as follows. 

Let $x_0$ denote an arbitrary point in the space $\X$. Define a function $f_0$ as:
\begin{equation*}
f_0(x) := f(x) - f(x_0).
\end{equation*}
Then we have $f_0(x_0)=0$. For two arbitrary probability measures $\mu_1, \mu_2 \in \mathcal{P}_1(\X)$, we have
\begin{align*}
&\bigg \lvert \int_{\X}f_0(x)\mu_1(dx) - \int_{\X}f_0(x)\mu_2(dx) \bigg \rvert \\
& \qquad =\bigg \lvert \int_{\X}(f(x)-f(x_0))\mu_1(dx) - \int_{\X}(f(x)-f(x_0))\mu_2(dx) \bigg \rvert \\
& \qquad =\bigg \lvert \int_{\X}f(x)\mu_1(dx) - \int_{\X}f(x)\mu_2(dx) -f(x_0)\int_{\X}\mu_1(dx) + f(x_0)\int_{\X}\mu_2(dx) \bigg \rvert \\
& \qquad =\bigg \lvert \int_{\X}f(x)\mu_1(dx) - \int_{\X}f(x)\mu_2(dx) -f(x_0) + f(x_0) \bigg \rvert \\
& \qquad =\bigg \lvert \int_{\X}f(x)\mu_1(dx) - \int_{\X}f(x)\mu_2(dx) \bigg \rvert.
\end{align*}
Therefore, $W_1(F_k(\mu),F_k(\mu^\prime))$ can be re-written as
\begin{align*}
&W_1(F_k(\mu),F_k(\mu^\prime)) \\
& \qquad = \sup_{f:\mathcal{X} \rightarrow \mathbb{R}, \lVert f \rVert_{\text{Lip}} \leq 1, f(x_0)=0}\bigg \lvert \int_{\mathcal{X}}f(x)\frac{h(y_k,x)}{Z_k(\mu)}\mu(dx) - \int_{\mathcal{X}}f(x)\frac{h(y_k,x)}{Z_k(\mu^\prime)}\mu^\prime(dx) \bigg \rvert.
\end{align*}
It follows that
\begin{align}
&W_1(F_k(\mu),F_k(\mu^\prime)) \nonumber \\
& \qquad = \sup_{f:\mathcal{X} \rightarrow \mathbb{R}, \lVert f \rVert_{\text{Lip}} \leq 1, f(x_0)=0}\bigg \lvert \int_{\mathcal{X}}f(x)\frac{h(y_k,x)}{Z_k(\mu)}\mu(dx) - \int_{\mathcal{X}}f(x)\frac{h(y_k,x)}{Z_k(\mu)}\mu^\prime(dx) \nonumber \\
& \qquad \qquad + \int_{\mathcal{X}}f(x)\frac{h(y_k,x)}{Z_k(\mu)}\mu^\prime(dx) - \int_{\mathcal{X}}f(x)\frac{h(y_k,x)}{Z_k(\mu^\prime)}\mu^\prime(dx) \bigg \rvert \nonumber \\
& \qquad \leq \sup_{f:\mathcal{X} \rightarrow \mathbb{R}, \lVert f \rVert_{\text{Lip}} \leq 1, f(x_0)=0}\left[{\bigg \lvert \int_{\mathcal{X}}f(x)\frac{h(y_k,x)}{Z_k(\mu)}\mu(dx) - \int_{\mathcal{X}}f(x)\frac{h(y_k,x)}{Z_k(\mu)}\mu^\prime(dx)\bigg \rvert}\right. \nonumber \\
& \qquad \qquad  \left.{+ \bigg \lvert \int_{\mathcal{X}}f(x)\frac{h(y_k,x)}{Z_k(\mu)}\mu^\prime(dx) - \int_{\mathcal{X}}f(x)\frac{h(y_k,x)}{Z_k(\mu^\prime)}\mu^\prime(dx) \bigg \rvert}\right] \nonumber \\
& \qquad \leq \sup_{f:\mathcal{X} \rightarrow \mathbb{R}, \lVert f \rVert_{\text{Lip}} \leq 1, f(x_0)=0} \bigg \lvert \int_{\mathcal{X}}f(x)\frac{h(y_k,x)}{Z_k(\mu)}\mu(dx) - \int_{\mathcal{X}}f(x)\frac{h(y_k,x)}{Z_k(\mu)}\mu^\prime(dx)\bigg \rvert \nonumber \\
& \qquad \qquad  + \sup_{f:\mathcal{X} \rightarrow \mathbb{R}, \lVert f \rVert_{\text{Lip}} \leq 1, f(x_0)=0} \bigg \lvert \int_{\mathcal{X}}f(x)\frac{h(y_k,x)}{Z_k(\mu)}\mu^\prime(dx) - \int_{\mathcal{X}}f(x)\frac{h(y_k,x)}{Z_k(\mu^\prime)}\mu^\prime(dx) \bigg \rvert.
\label{ieq:ip:W:1}
\end{align}
The first term on the right-hand side in Equation~\eqref{ieq:ip:W:1} can be written as
\begin{align*}
&\sup_{f:\mathcal{X} \rightarrow \mathbb{R}, \lVert f \rVert_{\text{Lip}} \leq 1, f(x_0)=0} \bigg \lvert \int_{\mathcal{X}}f(x)\frac{h(y_k,x)}{Z_k(\mu)}\mu(dx) - \int_{\mathcal{X}}f(x)\frac{h(y_k,x)}{Z_k(\mu)}\mu^\prime(dx)\bigg \rvert \\
& \quad = \frac{1}{Z_k(\mu)}\sup_{f:\mathcal{X} \rightarrow \mathbb{R}, \lVert f \rVert_{\text{Lip}}\leq 1, f(x_0)=0}\bigg \lvert \int_{\mathcal{X}}f(x)h(y_k,x)\mu(dx) - \int_{\mathcal{X}}f(x)h(y_k,x)\mu^\prime(dx) \bigg \rvert.
\end{align*}
Next, we use the proof strategy used in the proof of Theorem 15 in \cite{lipschitz_stability_ip} to show that 
\begin{align*}
&\sup_{f:\mathcal{X} \rightarrow \mathbb{R}, \lVert f \rVert_{\text{Lip}}\leq 1, f(x_0)=0}\bigg \lvert \int_{\mathcal{X}}f(x)h(y_k,x)\mu(dx) - \int_{\mathcal{X}}f(x)h(y_k,x)\mu^\prime(dx) \bigg \rvert \nonumber \\
&\qquad \qquad \leq \left(D\lVert h \rVert_{\text{Lip}}(y_k)+C_h(y_k)\right)W_1(\mu,\tilde{\mu}).
\end{align*}
The proof is as follows.

Define a function $g: \X \rightarrow \reals$ as
\begin{equation*}
g(x) := f(x)h(y_k,x).
\end{equation*}
Then $g$ satisfies
\begin{equation*}
g(x_0) = f(x_0)h(y_k,x_0) = 0.
\end{equation*}
Next, we prove that $g$ is Lipschitz continuous and its best Lipschitz constant is no greater than $D\lVert h \rVert_{\text{Lip}}(y_k)+C_h(y_k)$. 

For two arbitrary points $x_1,x_2 \in \mathcal{X}$, $\lvert g(x_1) - g(x_2) \rvert$ satisfies
\begin{align}
&\lvert g(x_1) - g(x_2) \rvert \nonumber \\
& \qquad = \lvert f(x_1)h(y_k,x_1) - f(x_2)h(y_k,x_2) \rvert \nonumber \\
& \qquad \leq \lvert f(x_1)h(y_k,x_1) - f(x_1)h(y_k,x_2) \rvert + \lvert f(x_1)h(y_k,x_2) - f(x_2)h(y_k,x_2) \rvert \nonumber \\
& \qquad =\lvert f(x_1) \rvert \lvert h(y_k,x_1) - h(y_k,x_2) \rvert + h(y_k,x_2)\lvert f(x_1)-f(x_2) \rvert \nonumber \\
& \qquad \leq \lvert f(x_1) \rvert \lvert h(y_k,x_1) - h(y_k,x_2) \rvert + C_h(y_k) \lvert f(x_1)-f(x_2) \rvert.
\label{ieq:ip:unnormalized:W:1}
\end{align}
Given $\lVert f \rVert_{\text{Lip}} \leq 1$ and $f(x_0)=0$, we have
\begin{align}
\lvert f(x_1) \rvert =& \lvert f(x_1) - f(x_0) + f(x_0) \rvert 
\leq \lvert f(x_1) - f(x_0) \rvert + \lvert f(x_0) \rvert \nonumber \\
=& \lvert f(x_1) - f(x_0) \rvert 
\leq d(x_1,x_0) 
\leq D,
\label{ieq:ip:unnormalized:W:2}
\end{align}
and 
\begin{align}
\lvert f(x_1) - f(x_2) \rvert \leq d_{\X}(x_1,x_2).
\label{ieq:ip:unnormalized:W:3}
\end{align}
The term $\lvert h(y_k,x_1) - h(y_k,x_2) \rvert$ satisfies
\begin{align}
\lvert h(y_k,x_1) - h(y_k,x_2) \rvert \leq \lVert h \rVert_{\text{Lip}}(y_k)d_{\X}(x_1,x_2).
\label{ieq:ip:unnormalized:W:4}
\end{align}
By substituting Equation~\ref{ieq:ip:unnormalized:W:2},~\ref{ieq:ip:unnormalized:W:3} and~\ref{ieq:ip:unnormalized:W:4} into Equation~\ref{ieq:ip:unnormalized:W:1}, we obtain
\begin{align*}
\lvert g(x_1) - g(x_2) \rvert \leq& D\lVert h \rVert_{\text{Lip}}(y_k)d_{\X}(x_1,x_2) + C_h(y_k)d_{\X}(x_1,x_2) \\
=& \left(D\lVert h \rVert_{\text{Lip}}(y_k)+C_h(y_k)\right)d_{\X}(x_1,x_2).
\end{align*}
It follows that 
\begin{equation*}
\frac{\lvert g(x_1) - g(x_2) \rvert}{d_{\X}(x_1,x_2)} \leq D\lVert h \rVert_{\text{Lip}}(y_k)+C_h(y_k), \quad \forall x_1,x_2 \in \X, x_1 \neq x_2.
\end{equation*}
Therefore, the function $g$ is Lipschitz continuous and we have
\begin{equation*}
\lVert g \rVert_{\text{Lip}} = \sup_{x_1,x_2 \in \mathcal{X}, x_1 \neq x_2}\frac{\lvert g(x_1) - g(x_2) \rvert}{d_{\X}(x_1,x_2)} \leq D\lVert h \rVert_{\text{Lip}}(y_k)+C_h(y_k).
\end{equation*}
Define a function $\tilde{f}$ as 
\begin{equation*}
\tilde{f}(x) = \frac{g(x)}{D\lVert h \rVert_{\text{Lip}}(y_k)+C_h(y_k)}.
\end{equation*}
Then $\tilde{f}$ is Lipschitz continuous and it satisfies
\begin{equation*}
\tilde{f}(x_0) = \frac{g(x_0)}{D\lVert h \rVert_{\text{Lip}}(y_k)+C_h(y_k)}= 0,
\end{equation*}
and 
\begin{align*}
\lVert \tilde{f} \rVert_{\text{Lip}} 
=& \sup_{x_1,x_2 \in \mathcal{X} ,x_1 \neq x_2}\frac{1}{D\lVert h \rVert_{\text{Lip}}(y_k)+C_h(y_k)}\frac{\lvert g(x_1) - g(x_2) \rvert}{d_{\X}(x_1,x_2)} \leq 1.
\end{align*}
Then we have 
\begin{align*}
&\sup_{f:\mathcal{X} \rightarrow \mathbb{R}, \lVert f \rVert_{\text{Lip}} \leq 1, f(x_0)=0} \bigg \lvert \int_{\mathcal{X}}f(x)h(y_k,x)\mu(dx) -\int_{\mathcal{X}}f(x)h(y_k,x)\mu^\prime(dx)\bigg \rvert \nonumber \\
&\qquad =\left(D\lVert h \rVert_{\text{Lip}}(y_k)+C_h(y_k)\right)\sup_{f:\mathcal{X} \rightarrow \mathbb{R}, \lVert f \rVert_{\text{Lip}} \leq 1, f(x_0)=0} \bigg \lvert \int_{\mathcal{X}}\tilde{f}(x)\mu(dx)-\int_{\mathcal{X}}\tilde{f}(x)\mu^\prime(dx)\bigg \rvert  \nonumber \\
&\qquad =\left(D\lVert h \rVert_{\text{Lip}}(y_k)+C_h(y_k)\right)\sup_{\tilde{f} \in \tilde{A}}\bigg \lvert \int_{\mathcal{X}}\tilde{f}(x)\mu(dx)- \int_{\mathcal{X}}\tilde{f}(x)\mu^\prime(dx) \bigg \rvert,
\end{align*}
where 
\begin{align}
\tilde{A} :=& \bigg \{\tilde{f}:\mathcal{X} \rightarrow \mathbb{R} \mid  \lVert \tilde{f} \rVert_{\text{Lip}} \leq 1, \tilde{f}(x_0) = 0, \tilde{f}(x)=\frac{h(y_k,x)f(x)}{D\lVert h \rVert_{\text{Lip}}(y_k)+C_h(y_k)}, \nonumber \\
&\qquad \qquad \lVert f \rVert_{\text{Lip}} \leq 1, f(x_0)=0\bigg \}. \nonumber
\end{align}
Note that  
\begin{equation*}
\tilde{A} \subseteq \left \{f: \mathcal{X} \rightarrow \mathbb{R} \mid \lVert f \rVert_{\text{Lip}} \leq 1, f(x_0)=0\right\}.
\end{equation*}
Therefore, we have
\begin{align*}
&\sup_{\tilde{f} \in \tilde{A}} \bigg \lvert \int_{\mathcal{X}}\tilde{f}(x)\mu(dx)- \int_{\mathcal{X}}\tilde{f}(x)\mu^\prime(dx) \bigg \rvert \nonumber \\
&\qquad \leq \sup_{f:\mathcal{X} \rightarrow \mathbb{R},\lVert f \rVert_{\text{Lip}} \leq 1, f(x_0)=0}\bigg \lvert \int_{\mathcal{X}}f(x)\mu(dx)-\int_{\X}f(x)\mu^\prime(dx)\bigg \rvert \nonumber \\
&\qquad = W_1(\mu,\mu^\prime).
\end{align*}
It follows that
\begin{align}
&\sup_{f:\mathcal{X} \rightarrow \mathbb{R}, \lVert f \rVert_{\text{Lip}}\leq 1, f(x_0)=0}\bigg \lvert \int_{\mathcal{X}}f(x)h(y_k,x)\mu(dx) - \int_{\mathcal{X}}f(x)h(y_k,x)\mu^\prime(dx) \bigg \rvert \nonumber \\
&\qquad \leq \left(D\lVert h \rVert_{\text{Lip}}(y_k)+C_h(y_k)\right)W_1(\mu,\mu^\prime).
\label{ieq:ip:W:term1:bound}
\end{align}

So far, we have derived an upper bound for the first term in Equation~\eqref{ieq:ip:W:1} using the proof strategy used for Theorem 15 in \cite{lipschitz_stability_ip}.

Next, we derive an upper bound for the second term in Equation~\eqref{ieq:ip:W:1}. We have
\begin{align}
&\sup_{f:\mathcal{X} \rightarrow \mathbb{R}, \lVert f \rVert_{\text{Lip}} \leq 1, f(x_0)=0} \bigg \lvert \int_{\mathcal{X}}f(x)\frac{h(y_k,x)}{Z_k(\mu)}\mu^\prime(dx) - \int_{\mathcal{X}}f(x)\frac{h(y_k,x)}{Z_k(\mu^\prime)}\mu^\prime(dx) \bigg \rvert \nonumber \\
&\qquad = \sup_{f:\mathcal{X} \rightarrow \mathbb{R}, \lVert f \rVert_{\text{Lip}} \leq 1, f(x_0)=0} \bigg \lvert \left(\frac{1}{Z_k(\mu)}-\frac{1}{Z_k(\mu^\prime)}\right)\int_{\mathcal{X}}f(x)h(y_k, x)\mu^\prime(dx) \bigg \rvert \nonumber \\
&\qquad =\frac{\big \lvert Z_k(\mu) -Z_k(\mu^\prime) \big \rvert}{Z_k(\mu)Z_k(\mu^\prime)}\sup_{f:\mathcal{X} \rightarrow \mathbb{R}, \lVert f \rVert_{\text{Lip}} \leq 1, f(x_0)=0} \bigg \lvert \int_{\mathcal{X}}f(x)h(y_k,x)\mu^\prime(dx) \bigg \rvert.
\label{eq:ip:W:term2}
\end{align}
We first note that the expression $\sup_{f:\mathcal{X} \rightarrow \mathbb{R}, \lVert f \rVert_{\text{Lip}} \leq 1, f(x_0)=0} \bigg \lvert \int_{\mathcal{X}}f(x)h(y_k,x)\mu^\prime(dx) \bigg \rvert$ satisfies
\begin{align*}
&\sup_{f:\mathcal{X} \rightarrow \mathbb{R}, \lVert f \rVert_{\text{Lip}} \leq 1, f(x_0)=0} \bigg \lvert \int_{\mathcal{X}}f(x)h(y_k,x)\mu^\prime(dx) \bigg \rvert \\
&\qquad \leq \sup_{f:\mathcal{X} \rightarrow \mathbb{R}, \lVert f \rVert_{\text{Lip}} \leq 1, f(x_0)=0}  \int_{\mathcal{X}}\bigg \lvert f(x)\bigg \rvert h(y_k,x)\mu^\prime(dx) .
\end{align*}
Recall that for any $x \in \mathcal{X}$, the expression $\lvert f(x) \rvert$ satisfies $\lvert f(x) \rvert \leq D$.
Therefore, we have
\begin{align}
&\sup_{f:\mathcal{X} \rightarrow \mathbb{R}, \lVert f \rVert_{\text{Lip}} \leq 1, f(x_0)=0}  \int_{\mathcal{X}}\bigg \lvert f(x)\bigg \rvert h(y_k,x)\mu^\prime(dx)  \nonumber \\
&\qquad \leq D\sup_{f:\mathcal{X} \rightarrow \mathbb{R}, \lVert f \rVert_{\text{Lip}} \leq 1, f(x_0)=0}  \int_{\mathcal{X}} h(y_k,x)\mu^\prime(dx)  \nonumber \\
&\qquad = D\int_{\mathcal{X}}h(y_k,x)\mu^\prime(dx) = DZ_k(\mu^\prime).
\label{ieq:ip:W:term2:part2}
\end{align}
Substituting Equation~\eqref{ieq:ip:W:term2:part2} into Equation~\eqref{eq:ip:W:term2} gives
\begin{align}
&\sup_{f:\mathcal{X} \rightarrow \mathbb{R}, \lVert f \rVert_{\text{Lip}} \leq 1, f(x_0)=0} \bigg \lvert \int_{\mathcal{X}}f(x)\frac{h(y_k,x)}{Z_k(\mu)}\mu^\prime(dx) - \int_{\mathcal{X}}f(x)\frac{h(y_k,x)}{Z_k(\mu^\prime)}\mu^\prime(dx) \bigg \rvert \nonumber \\
&\qquad \leq \frac{\big \lvert Z_k(\mu) -Z_k(\mu^\prime) \big \rvert}{Z_k(\mu)}D.
\label{ieq:ip:W:term2}
\end{align}
By Lemma ~\ref{lemma:evidence_diff_bound:ip}$, \big \lvert Z_k(\mu) - Z_k(\mu^\prime) \big \rvert$ is bounded by: 
\begin{equation}
\big \lvert Z_k(\mu) - Z_k(\mu^\prime) \big \rvert \leq \lVert h \rVert_{\text{Lip}}(y_k) W_1(\mu,\mu^\prime).
\label{ieq:ip:W:term2:part1}
\end{equation}
Plug Equation~\eqref{ieq:ip:W:term2:part1} into Equation~\eqref{ieq:ip:W:term2}. We obtain
\begin{align}
&\sup_{f:\mathcal{X} \rightarrow \mathbb{R}, \lVert f \rVert_{\text{Lip}} \leq 1, f(x_0)=0} \bigg \lvert \int_{\mathcal{X}}f(x)\frac{h(y_k,x)}{Z_k(\mu)}\mu^\prime(dx) - \int_{\mathcal{X}}f(x)\frac{h(y_k,x)}{Z_k(\mu^\prime)}\mu^\prime(dx) \bigg \rvert \nonumber \\
&\qquad \leq \frac{D\lVert h \rVert_{\text{Lip}}(y_k)}{Z_k(\mu)}W_1(\mu,\mu^\prime).
\label{ieq:ip:W:term2:bound}
\end{align}
Combining Equation~\eqref{ieq:ip:W:term1:bound} with Equation~\eqref{ieq:ip:W:term2:bound} gives
\begin{equation*}
W_1(F_k(\mu),F_k(\mu^\prime)) \leq \frac{2D\lVert h \rVert_{\text{Lip}}(y_k)+C_h(y_k)}{Z_k(\mu)}W_1(\mu,\mu^\prime).
\end{equation*}
\end{proof}

\begin{prop}
In state estimation problems, given the data $y_k$, under Assumption SE.2 and IP-SE-PS.4, $\forall \mu, \mu^\prime \in \bar{\mathcal{P}}_k(\X) \bigcap \mathcal{P}_1(\X)$, the posteriors $F_k(\mu)$ and $F_k(\mu^\prime)$ satisfy
\begin{equation*}
W_1(F_k(\mu),F_k(\mu^\prime)) \leq \frac{2C^*_{Th}(y_k;k)D}{Z_k(\mu)}W_1(\mu,\mu^\prime).
\end{equation*}
\label{prop:se:normalized_bound_W}
\end{prop}
The high-level proof strategy for Proposition~\ref{prop:se:normalized_bound_W} follows that of Proposition~\ref{prop:ip:normalized_bound_W}, with modifications to accommodate the more complex setting of state estimation problems compared to inverse problems.
\begin{proof}
Analogous to the proof of Proposition~\ref{prop:ip:normalized_bound_W}, we can show that the $1$-Wasserstein distance between the two posteriors $F_k(\mu)$ and $F_k(\mu^\prime)$ satisfies
\begin{align}
W_1(F_k(\mu),F_k(\mu^\prime)) 
& \leq \frac{1}{Z_k(\mu)}\sup_{f:\mathcal{X} \rightarrow \mathbb{R}, \lVert f \rVert_{\text{Lip}} \leq 1, f(\bar{x})=0}\bigg \lvert \int_\mathcal{X}f(x)\tilde{F}_k\mu(dx) - \int_\mathcal{X}f(x)\tilde{F}_k\mu^\prime(dx) \bigg \rvert \nonumber \\
&\qquad \qquad + \frac{\lvert Z_k(\mu) - Z_k(\mu^\prime) \rvert}{Z_k(\mu)Z_k(\mu^\prime)}\sup_{f:\mathcal{X} \rightarrow \mathbb{R}, \lVert f \rVert_{\text{Lip}} \leq 1, f(\bar{x})=0}\bigg \lvert \int_\mathcal{X}f(x)\tilde{F}_k\mu^\prime(dx)\bigg \rvert,
\label{ieq:se:normalized:W:1}
\end{align}
where $\bar{x} \in \X$ is an arbitrary point. 

We begin by bounding the first term on the right-hand side of Equation~\eqref{ieq:se:normalized:W:1}, which can be expressed as:

\begin{align}
&\sup_{f:\mathcal{X} \rightarrow \mathbb{R}, \lVert f \rVert_{\text{Lip}} \leq 1, f(\bar{x})=0}\bigg \lvert \int_\mathcal{X}f(x)\tilde{F}_k\mu(dx) - \int_\mathcal{X}f(x)\tilde{F}_k\mu^\prime(dx) \bigg \rvert \nonumber \\
&\qquad = \sup_{f:\mathcal{X} \rightarrow \mathbb{R}, \lVert f \rVert_{\text{Lip}}\leq 1,f(\bar{x})=0}\bigg \lvert \int_{\mathcal{X}}f(x)h_k(y_k, x)\left(\int_{\mathcal{X}}T_k(x,x_{k-1})\mu(dx_{k-1})\right)dx \nonumber \\
&\qquad \qquad - \int_{\mathcal{X}}f(x)h_k(y_k,x)\left(\int_{\mathcal{X}}T_k(x,x_{k-1})\mu^\prime(dx_{k-1})\right)dx\bigg \rvert. 
\label{eq:normalized:se:new:1:new}
\end{align}
Analogous to the proof of Proposition~\ref{prop:ip:normalized_bound_W}, we have $\lvert f(x) \rvert \leq D$. It follows that
\begin{align*}
&\int_{\X \times \X}\bigg \lvert f(x) h_k(y_k,x)T_k(x,x_{k-1})\bigg \rvert \mu(dx_{k-1})dx \\
&\qquad = \int_{\X \times \X} \big \lvert f(x) \big \rvert  h_k(y_k,x)T_k(x,x_{k-1})\mu(dx_{k-1})dx \\
& \qquad = \int_{\X}\big \lvert f(x) \big \rvert h_k(y_k,x)\left(\int_{\X}T_k(x,x_{k-1})\mu(dx_{k-1})\right)dx \\
& \qquad \leq D\int_{\X}h_k(y_k,x)\left(\int_{\X}T_k(x,x_{k-1})\mu(dx_{k-1})\right)dx \\
& \qquad = DZ_k(\mu) < \infty,
\end{align*}
where the second equality is obtained by Tonelli's theorem and the last inequality is obtained because $\mu \in \bar{\mathcal{P}}_k(\X)$. \\
Therefore, the function $f(x) h_k(y_k,x)T_k(x,x_{k-1})$ is integrable and we can use Fubini's theorem to obtain
\begin{align}
&\int_{\mathcal{X}}f(x)h_k(y_k, x)\left(\int_{\mathcal{X}}T_k(x,x_{k-1})\mu(dx_{k-1})\right)dx \nonumber \\
&\qquad = \int_{\mathcal{X}}\left(\int_\mathcal{X}f(x)h_k(y_k, x)T_k(x,x_{k-1})dx\right)\mu(dx_{k-1}).
\label{eq:rewrite_Fubini_mu}
\end{align}
Similarly, we have
\begin{align}
&\int_{\mathcal{X}}f(x)h_k(y_k, x)\left(\int_{\mathcal{X}}T_k(x,x_{k-1})\mu^\prime(dx_{k-1}))\right)dx \nonumber \\
&\qquad = \int_{\mathcal{X}}\left(\int_\mathcal{X}f(x)h_k(y_k, x)T_k(x,x_{k-1})dx\right)\mu^\prime(dx_{k-1}).
\label{eq:rewrite_Fubini_mu_prime}
\end{align}
Substituting Equations~\eqref{eq:rewrite_Fubini_mu} and~\eqref{eq:rewrite_Fubini_mu_prime} into Equation~\eqref{eq:normalized:se:new:1:new} yields:
\begin{align}
&\sup_{f:\mathcal{X} \rightarrow \mathbb{R}, \lVert f \rVert_{\text{Lip}} \leq 1, f(\bar{x})=0}\bigg \lvert \int_\mathcal{X}f(x)\tilde{F}_k\mu(dx) - \int_\mathcal{X}f(x)\tilde{F}_k\mu^\prime(dx) \bigg \rvert \nonumber \\
&\qquad = \sup_{f:\mathcal{X} \rightarrow \mathbb{R}, \lVert f \rVert_{\text{Lip}}\leq 1,f(\bar{x})=0}\bigg \lvert \int_{\mathcal{X}}\left(\int_\mathcal{X}f(x)h_k(y_k, x)T_k(x,x_{k-1})dx\right)\mu(dx_{k-1}) \nonumber \\
&\qquad \qquad - \int_{\mathcal{X}}\left(\int_\mathcal{X}f(x)h_k(y_k, x)T_k(x,x_{k-1})dx\right)\mu^\prime(dx_{k-1})\bigg \rvert.
\label{eq:normalized:se:new:1}
\end{align}
The expression $\int_\mathcal{X}f(x)h_k(y_k, x)T_k(x,x_{k-1})dx$ is a function of $x_{k-1}$. We denote this function as $\tilde{f}$, i.e.,
\begin{equation*}
\tilde{f}(x_{k-1}) := \int_\mathcal{X}f(x)h_k(y_k, x)T_k(x,x_{k-1})dx.
\end{equation*}
Next, we prove that $\tilde{f}$ is Lipschitz continuous and its best Lipschitz constant is smaller than or equal to $DC^*_{Th}(y_k;k)$.
For two arbitrary points $x_{k-1}$, $ x_{k-1}^\prime \in \mathcal{X}$, we have
\begin{align*}
\lvert \tilde{f}(x_{k-1}) - \tilde{f}(x_{k-1}^\prime) \rvert =& \bigg \lvert \int_\mathcal{X}f(x)h_k(y_k,x)T_k(x,x_{k-1})dx - \int_\mathcal{X}f(x)h_k(y_k,x)T_k(x,x_{k-1}^\prime)dx \bigg \rvert \nonumber \\
=& \bigg \lvert \int_\mathcal{X}f(x)h_k(y_k,x)\left(T_k(x,x_{k-1})-T_k(x,x_{k-1}^\prime)\right)dx \bigg \rvert \nonumber \\
\leq & \int_\mathcal{X}\bigg \lvert f(x) \bigg \rvert h_k(y_k,x) \bigg \lvert T_k(x,x_{k-1})-T_k(x,x_{k-1}^\prime) \bigg \rvert  dx.
\end{align*}
As we also have $\lvert f(x) \rvert \leq D$, and 
\begin{equation*}
\lvert T_k(x,x_{k-1})-T_k(x,x_{k-1}^\prime)\rvert \leq T_{\text{Lip}}(x;k)d_{\X}(x_{k-1},x_{k-1}^\prime),
\end{equation*}
the expression $\lvert \tilde{f}(x_{k-1}) - \tilde{f}(x_{k-1}^\prime) \rvert$ satisfies
\begin{align*}
\lvert \tilde{f}(x_{k-1}) - \tilde{f}(x_{k-1}^\prime) \rvert \leq &Dd(x_{k-1},x_{k-1}^\prime) \int_\mathcal{X}h_k(y_k, x)T_{\text{Lip}}(x;k)dx \\
= & DC^*_{Th}(y_k;k)d_{\X}(x_{k-1},x_{k-1}^\prime).
\end{align*}
It follows that
\begin{equation*}
\sup_{x_{k-1},x_{k-1}^\prime \in \mathcal{X},x_{k-1}\neq x^\prime_{k-1}}\frac{\lvert \tilde{f}(x_{k-1}) - \tilde{f}(x_{k-1}^\prime) \rvert}{d_{\X}(x_{k-1},x_{k-1}^\prime)} \leq DC^*_{Th}(y_k;k).
\end{equation*}
Define a function $g$ as
\begin{equation*}
g(x) := \frac{\tilde{f}(x)}{DC^*_{Th}(y_k;k)}.
\end{equation*}
Then $g$ is Lipschitz continuous with the best Lipschitz constant
\begin{align*}
\lVert g \rVert_{\text{Lip}} 
= \frac{1}{DC^*_{Th}(y_k;k)}\sup_{x,x^\prime \in \mathcal{X}, x\neq x^\prime }\frac{\lvert \tilde{f}(x) - \tilde{f}(x^\prime)\rvert}{d_{\X}(x,x^\prime)} 
\leq  1.
\end{align*}
Therefore, we have
\begin{align}
& \sup_{f:\mathcal{X} \rightarrow \mathbb{R}, \lVert f \rVert_{\text{Lip}} \leq 1, f(\bar{x})=0}\bigg \lvert \int_\mathcal{X}f(x)\tilde{F}_k\mu(dx) - \int_\mathcal{X}f(x)\tilde{F}_k\mu^\prime(dx) \bigg \rvert \nonumber \\
& \quad = DC^*_{Th}(y_k;k)\sup_{g \in \hat{A}}\bigg \lvert \int_{\mathcal{X}}g(x_{k-1})\mu(dx_{k-1})-\int_{\mathcal{X}}g(x_{k-1})\mu^\prime(dx_{k-1}) \bigg \rvert.
\label{eq:se:term1:rewrite}
\end{align}
where 
\begin{align*}
\hat{A} = &\bigg \{g: \mathcal{X}\rightarrow \mathbb{R} \mid \lVert g \rVert_{\text{Lip}}\leq 1, \\
& \qquad g(x_{k-1})=\frac{1}{DC^*_{Th}(y_k;k)}\int_{\mathcal{X}}f(x)h_k(y_k,x)T_k(x,x_{k-1})dx, f \in A\bigg \},
\end{align*}
with 
\begin{equation*}
A =\{f: \mathcal{X} \rightarrow \mathbb{R} \mid \lVert f \rVert_{\text{Lip}}\leq 1,f(\bar{x})=0 \}.
\end{equation*}
As $\hat{A} \subseteq \{f: \mathcal{X} \rightarrow \mathbb{R} \mid \lVert f \rVert_{\text{Lip}}\leq 1\}$, we have
\begin{align}
&\sup_{g \in \hat{A}}\bigg \lvert \int_{\mathcal{X}}g(x_{k-1})\mu(dx_{k-1})-\int_{\mathcal{X}}g(x_{k-1})\mu^\prime(dx_{k-1})\bigg \rvert \nonumber \\
&\qquad \leq \sup_{f:\mathcal{X}  \rightarrow \mathbb{R},\lVert f \rVert_{\text{Lip}} \leq 1}\bigg \lvert \int_{\mathcal{X}}f(x_{k-1})\mu(dx_{k-1})-\int_{\mathcal{X}}f(x_{k-1})\mu^\prime(dx_{k-1}) \bigg \rvert \nonumber \\
&\qquad = W_1(\mu,\mu^\prime).
\label{ieq:se:W:term1}
\end{align}
Analogous to the proof of Proposition~\ref{prop:ip:normalized_bound_W}, the second term on the right-hand side of Equation~\eqref{ieq:se:normalized:W:1} satisfies
\begin{align}
\frac{\lvert Z_k(\mu)-Z_k(\mu^\prime)\rvert }{Z_k(\mu)Z_k(\mu^\prime)}\sup_{f:\mathcal{X} \rightarrow \mathbb{R}, \lVert f \rVert_{\text{Lip}} \leq 1, f(\bar{x})=0}\bigg \lvert \int_\mathcal{X}f(x)\tilde{F}_k\mu^\prime(dx) \bigg \rvert 
\leq \frac{D\lvert Z_k(\mu)-Z_k(\mu^\prime)\rvert }{Z_k(\mu)}.
\label{ieq:se:normalized:W:2}
\end{align}
By plugging Equations~\eqref{eq:se:term1:rewrite}, ~\eqref{ieq:se:W:term1}, and ~\eqref{ieq:se:normalized:W:2} into Equation~\eqref{ieq:se:normalized:W:1}, we obtain
\begin{equation*}
W_1(F_k(\mu),F_k(\mu^\prime)) \leq \frac{DC^*_{Th}(y_k;k)}{Z_k(\mu)}W_1(\mu,\mu^\prime) + \frac{D\lvert Z_k(\mu) - Z_k(\mu^\prime)\rvert}{Z_k(\mu)}.
\end{equation*}
By Lemma~\ref{lemma:evidence_diff_bound:se}, the term $\lvert Z_k(\mu) - Z_k(\mu^\prime) \rvert$ satisfies
\begin{equation*}
\lvert Z_k(\mu) - Z_k(\mu^\prime) \rvert \leq C_{Th}^*(y_k;k)W_1(\mu,\mu^\prime).
\end{equation*}
It follows that
\begin{equation*}
W_1\left(F_k(\mu),F_k(\mu^\prime)\right) \leq \frac{2DC^*_{Th}(y_k;k)}{Z_k(\mu)}W_1(\mu,\mu^\prime).
\end{equation*}
\end{proof}

\begin{prop}
In parameter-state estimation problems, given the data $y_k$, under Assumptions PS.1, PS.2, PS.3 and IP-SE-PS.4, $\forall \mu, \mu^\prime \in \bar{\mathcal{P}}_k(\X \times \W) \bigcap \mathcal{P}_1(\X \times \W)$, the posteriors $F_k(\mu)$ and $F_k(\mu^\prime)$ satisfy
\begin{equation*}
W_1(F_k(\mu),F_k(\mu^\prime)) \leq \frac{2\tilde{C}^*_{Th}(y_k;k)D+\tilde{C}_{Th}(y_k;k)}{Z_k(\mu)}W_1(\mu,\mu^\prime).
\end{equation*}
\label{prop:ps:normalized_bound_W}
\end{prop}
The high-level proof strategy for Proposition~\ref{prop:ps:normalized_bound_W} are similar to those of Propositions~\ref{prop:ip:normalized_bound_W} and~\ref{prop:se:normalized_bound_W}, with modifications to adapt to the increased complexity of parameter-state estimation problems compared to inverse and state estimation problems.

\begin{proof}
Analogous to the proof of Proposition~\ref{prop:ip:normalized_bound_W}, we can show that the $1$-Wasserstein distance between the two posteriors $F_k(\mu)$ and $F_k(\mu^\prime)$ satisfies
\begin{align}
&W_1(F_k(\mu),F_k(\mu^\prime)) \nonumber \\
&\qquad \leq \frac{1}{Z_k(\mu)}\sup_{f:\mathcal{X}\times \W \rightarrow \mathbb{R}, \lVert f \rVert_{\text{Lip}}\leq 1,f(\bar{x},\bar{w})=0}\bigg \lvert \int_{\mathcal{X}\times \W}f(x,w)\tilde{F}_k\mu(dx,dw) \nonumber \\
&\qquad \qquad -\int_{\mathcal{X}\times \W}f(x,w)\tilde{F}_k\mu^\prime(dx,dw) \bigg \rvert \nonumber \\
&\qquad \qquad +\frac{\lvert Z_k(\mu)-Z_k(\mu^\prime)\rvert}{Z_k(\mu)Z_k(\mu^\prime)}\sup_{f:\mathcal{X}\times \W \rightarrow \mathbb{R}, \lVert f \rVert_{\text{Lip}}\leq 1,f(\bar{x},\bar{w})=0}\bigg \lvert \int_{\mathcal{X}\times \W}f(x,w)\tilde{F}_k\mu^\prime(dx,dw)  \bigg \rvert.
\label{ieq:ps:W}
\end{align}
where $(\bar{x},\bar{w})$ is an arbitrary point in the space $\mathcal{X}\times \W$.
We first derive an upper bound on the first term on the right-hand side of Equation~\eqref{ieq:ps:W}.

Let $\pi$ and $\pi^\prime$ denote the densities of $\mu$ and $\mu^\prime$ with respect to the Lebesgue measure, respectively. We have
\begin{align}
&\sup_{ f \in A}\bigg \lvert \int_{\mathcal{X}\times \W}f(x,w)\tilde{F}_k\mu(dx,dw) -\int_{\mathcal{X}\times \W}f(x,w)\tilde{F}_k\mu^\prime(dx,dw) \bigg \rvert \nonumber \\
&\qquad = \sup_{f \in A}\bigg \lvert \int_{\mathcal{X}\times \W}f(x,w)h_k(y_k, x,w)\left(\int_{\mathcal{X}}T_k(x, x_{k-1},w)\pi(x_{k-1},w)dx_{k-1}\right)dxdw \nonumber \\
& \qquad \qquad - \int_{\mathcal{X}\times \W}f(x,w)h_k(y_k, x,w)\left(\int_{\mathcal{X}}T_k(x, x_{k-1},w)\pi^\prime(x_{k-1},w)dx_{k-1}\right)dxdw  \bigg \rvert,
\label{eq:term1:ps:normalized:new}
\end{align}
where 
\begin{equation*}
A = \{f:\mathcal{X}\times \W \rightarrow \mathbb{R} \mid  \lVert f \rVert_{\text{Lip}}\leq 1,f(\bar{x},\bar{w})=0\}.
\end{equation*}
Analogous to the proof of Proposition~\ref{prop:ip:normalized_bound_W}, we have 
\begin{equation*}
\lvert f(x,w) \rvert \leq D, \quad \forall x \in \X, w \in \W.
\end{equation*}
It follows that
\begin{align}
&\int_{\X \times \W \times \X}\bigg \lvert f(x,w) h_k(y_k,x,w)T_k(x,x_{k-1},w)\pi(x_{k-1},w) \bigg \rvert dxdw dx_{k-1} \nonumber \\
&\qquad = \int_{\X \times \W \times \X} \big \lvert f(x,w) \big \rvert  h_k(y_k,x,w)T_k(x,x_{k-1},w)\pi(x_{k-1},w) dxdw dx_{k-1} \nonumber \\
&\qquad = \int_{\X\times \W}\big \lvert f(x,w) \big \rvert h_k(y_k,x,w)\left(\int_{\X}T_k(x,x_{k-1},w)\pi(x_{k-1},w) dx_{k-1}\right)dxdw \nonumber \\
& \qquad \leq D\int_{\X\times \W}h_k(y_k,x,w)\left(\int_{\X}T_k(x,x_{k-1},w)\pi(x_{k-1},w) dx_{k-1}\right)dxdw  \nonumber \\
& \qquad = DZ_k(\mu) < \infty.
\label{eq:integrable}
\end{align}
The second equality in Equation~\eqref{eq:integrable} is obtained by Tonelli's theorem and the last inequality in Equation~\eqref{eq:integrable} is obtained since $\mu \in \bar{\mathcal{P}}_k(\X \times \W)$. \\
Then we can use Fubini's theorem to re-write the first term on the right-hand side of Equation~\eqref{eq:term1:ps:normalized:new} as
\begin{align}
&\int_{\mathcal{X}\times \W}f(x,w)h_k(y_k, x,w)\left(\int_{\mathcal{X}}T_k(x, x_{k-1},w)\pi(x_{k-1},w)dx_{k-1}\right)dxdw \nonumber \\
&\qquad = \int_{\mathcal{X}\times \W}\left(\int_{\mathcal{X}}f(x,w)h_k(y_k, x,w)T_k(x, x_{k-1},w)dx\right)\pi(x_{k-1},w)dx_{k-1}dw.
\label{eq:rewrite_Fubini_mu:ps}
\end{align}
Similarly, the second term on the right-hand side of Equation~\eqref{eq:term1:ps:normalized:new} can be written as:
\begin{align}
&\int_{\mathcal{X}\times \W}f(x,w)h_k(y_k, x,w)\left(\int_{\mathcal{X}}T_k(x, x_{k-1},w)\pi^\prime(x_{k-1},w)dx_{k-1}\right)dxdw \nonumber \\
&\qquad = \int_{\mathcal{X}\times \W}\left(\int_{\mathcal{X}}f(x,w)h_k(y_k, x,w)T_k(x, x_{k-1},w)dx\right)\pi^\prime(x_{k-1},w)dx_{k-1}dw.
\label{eq:rewrite_Fubini_mu_prime:ps}
\end{align}
By plugging Equations~\eqref{eq:rewrite_Fubini_mu:ps} and~\eqref{eq:rewrite_Fubini_mu_prime:ps} into Equation~\eqref{eq:term1:ps:normalized:new}, we obtain:
\begin{align}
&\sup_{ f \in A}\bigg \lvert \int_{\mathcal{X}\times \W}f(x,w)\tilde{F}_k\mu(dx,dw) -\int_{\mathcal{X}\times \W}f(x,w)\tilde{F}_k\mu^\prime(dx,dw) \bigg \rvert \nonumber \\
&\quad =\sup_{f \in A}\bigg \lvert \int_{\mathcal{X}\times \W}\left(\int_{\mathcal{X}}f(x,w)h_k(y_k, x,w)T_k(x, x_{k-1},w)dx\right)\pi(x_{k-1},w)dx_{k-1}dw \nonumber \\
& \qquad \qquad - \int_{\mathcal{X}\times \W}\left(\int_{\mathcal{X}}f(x,w)h_k(y_k, x,w)T_k(x, x_{k-1},w)dx\right)\pi^\prime(x_{k-1},w)dx_{k-1}dw \bigg \rvert.
\label{eq:normalized:ps:new:1}
\end{align}

The expression $\int_{\mathcal{X}}f(x,w)h_k(y_k, x,w)T_k(x, x_{k-1},w)dx$ is a function of $(x_{k-1},w)$, which we denote by $\tilde{f}$:
\begin{equation*}
\tilde{f}(x_{k-1},w) := \int_{\mathcal{X}}f(x,w)h_k(y_k,x,w)T_k(x, x_{k-1},w)dx.
\end{equation*}
We now show that $\tilde{f}$ is Lipschitz continuous and its best Lipschitz constant is no greater than $\tilde{C}^*_{Th}(y_k;k)D+\tilde{C}_{Th}(y_k;k)$.

For two arbitrary points $(x_{k-1},w)$ and $(x_{k-1}^\prime, w^\prime)$ in the space $\mathcal{X}\times \W$, we have
\begin{align}
& \big \lvert \tilde{f}(x_{k-1}, w) - \tilde{f}(x_{k-1}^\prime, w^\prime) \big \rvert \nonumber \\
&=\bigg \lvert \int_{\mathcal{X}}f(x,w)h_k(y_k, x,w)T_k(x, x_{k-1},w)dx - \int_{\mathcal{X}}f(x,w^\prime)h_k(y_k, x,w^\prime)T_k(x, x_{k-1}^\prime,w^\prime)dx \bigg \rvert \nonumber \\
&\leq \bigg \lvert \int_{\mathcal{X}}f(x,w)\left(h_k(y_k, x,w)T_k(x, x_{k-1},w)-h_k(y_k, x,w^\prime)T_k(x, x_{k-1}^\prime,w^\prime)\right)dx \bigg \rvert \nonumber \\
&\qquad + \bigg \lvert \int_{\mathcal{X}}\left(f(x,w)-f(x,w^\prime)\right)h_k(y_k, x,w^\prime)T_k(x, x_{k-1}^\prime,w^\prime)dx \bigg \rvert.
\label{ieq:unnormalized:W:ini}
\end{align}
We first show that the first term on the right-hand side of Equation~\eqref{ieq:unnormalized:W:ini} is bounded by the product of a constant and $d_{\X \times \W}((x_{k-1},w),(x_{k-1}^\prime,w^\prime))$. We have
\begin{align*}
&\bigg \lvert \int_{\mathcal{X}}f(x,w)\left(h_k(y_k, x,w)T_k(x, x_{k-1},w)-h_k(y_k, x,w^\prime)T_k(x, x_{k-1}^\prime,w^\prime)\right)dx \bigg \rvert \nonumber \\
& \qquad \leq  \int_{\mathcal{X}}\bigg \lvert f(x,w) \bigg \rvert \bigg \lvert h_k(y_k, x,w)T_k(x, x_{k-1},w)-h_k(y_k, x,w^\prime)T_k(x, x_{k-1}^\prime,w^\prime) \bigg \rvert dx .
\end{align*}
Recall that $\lvert f(x,w) \rvert \leq D$ for any $x \in \X, w \in \W$. Hence we have
\begin{align}
&\int_{\mathcal{X}}\bigg \lvert f(x,w) \bigg \rvert \bigg \lvert h_k(y_k, x,w)T_k(x, x_{k-1},w)-h_k(y_k, x,w^\prime)T_k(x, x_{k-1}^\prime,w^\prime) \bigg \rvert dx \nonumber \\
& \qquad \leq D\int_{\mathcal{X}}\bigg \lvert h_k(y_k, x,w)T_k(x, x_{k-1},w)-h_k(y_k, x,w^\prime)T_k(x, x_{k-1}^\prime,w^\prime) \bigg \rvert dx \nonumber \\
& \qquad \leq D \int_{\mathcal{X}}\tilde{f}_{\text{Lip}}(x;k)d_{\X \times \W}((x_{k-1},w),(x_{k-1}^\prime,w^\prime))dx\nonumber \\
&\qquad = Dd_{\X \times \W}((x_{k-1},w),(x_{k-1}^\prime,w^\prime))\int_{\mathcal{X}}\tilde{f}_{\text{Lip}}(x;k)dx \nonumber \\
&\qquad = D\tilde{C}^*_{Th}(y_k;k)d_{\X \times \W}((x_{k-1},w),(x_{k-1}^\prime,w^\prime)).
\label{ieq:unnormalized:W:term1}
\end{align}
Next, we prove that the second term on the right-hand side of Equation~\eqref{ieq:unnormalized:W:ini} is also bounded by the product of a constant and $d_{\X \times \W}((x_{k-1},w),(x_{k-1}^\prime,w^\prime))$.
\begin{align}
&\bigg \lvert \int_{\mathcal{X}}\left(f(x,w)-f(x,w^\prime)\right)h_k(y_k, x,w^\prime)T_k(x, x_{k-1}^\prime,w^\prime)dx \bigg \rvert \nonumber \\
&\qquad \leq \int_{\mathcal{X}}\bigg \lvert f(x,w)-f(x,w^\prime) \bigg \rvert h_k(y_k, x,w^\prime)T_k(x, x_{k-1}^\prime,w^\prime)dx  \nonumber \\
& \qquad \leq \int_{\mathcal{X}}d_{\X \times \W}((x,w),(x,w^\prime))h_k(y_k, x,w^\prime)T_k(x, x_{k-1}^\prime,w^\prime)dx \nonumber \\
& \qquad \leq \int_{\mathcal{X}}d_{\X \times \W}((x_{k-1},w),(x_{k-1}^\prime,w^\prime))h_k(y_k, x,w^\prime)T_k(x, x_{k-1}^\prime,w^\prime)dx \nonumber \\
&\qquad = d_{\X \times \W}((x_{k-1},w),(x_{k-1}^\prime,w^\prime))\int_{\mathcal{X}}h_k(y_k, x,w^\prime)T_k(x, x_{k-1}^\prime,w^\prime)dx \nonumber \\
&\qquad \leq \tilde{C}_{Th}(y_k;k)d_{\X \times \W}((x_{k-1},w),(x_{k-1}^\prime,w^\prime)),
\label{ieq:unnormalized:W:term2}
\end{align}
where the third inequality is obtained according to Assumption PS.3. 
Combining Equation~\eqref{ieq:unnormalized:W:term1} with Equation~\eqref{ieq:unnormalized:W:term2} gives
\begin{align*}
\lVert \tilde{f} \rVert_{\text{Lip}} &= \sup_{(x_{k-1},w), (x_{k-1}^\prime,w^\prime) \in \mathcal{X} \times \W, (x_{k-1},w) \neq (x_{k-1}^\prime,w^\prime)}\frac{\lvert \tilde{f}(x_{k-1},w) - \tilde{f}(x_{k-1}^\prime,w^\prime)\rvert}{d_{\X \times \W}((x_{k-1},w),(x_{k-1}^\prime,w^\prime))} \\
&\leq \tilde{C}^*_{Th}(y_k;k)D+\tilde{C}_{Th}(y_k;k).
\end{align*}
Then following an argument analogous to that used in the proof of Proposition~\ref{prop:se:normalized_bound_W}, we can show that 
\begin{align}
&\sup_{ f \in A}\bigg \lvert \int_{\mathcal{X}\times \W}f(x,w)\tilde{F}_k\mu(dx,dw) -\int_{\mathcal{X}\times \W}f(x,w)\tilde{F}_k\mu^\prime(dx,dw) \bigg \rvert \nonumber \\
&\qquad \leq \left(\tilde{C}^*_{Th}(y_k;k)D+\tilde{C}_{Th}(y_k;k)\right)W_1(\mu,\mu^\prime).
\label{ieq:ps:W:term1}
\end{align}
Similar to the proof of Proposition~\ref{prop:ip:normalized_bound_W}, the second term on the right-hand side of Equation~\eqref{ieq:ps:W} satisfies
\begin{align*}
&\frac{\lvert Z_k(\mu)-Z_k(\mu^\prime) \rvert}{Z_k(\mu)Z_k(\mu^\prime)}\sup_{f:\mathcal{X}\times \W \rightarrow \mathbb{R}, \lVert f \rVert_{\text{Lip}}\leq 1,f(\bar{x},\bar{w})=0}\bigg \lvert 
\int_{\mathcal{X}\times \W}f(x,w)\tilde{F}_k\mu^\prime(dxdw)  \bigg \rvert \nonumber \\
&\qquad  \leq \frac{D\lvert Z_k(\mu)-Z_k(\mu^\prime)\rvert }{Z_k(\mu)}.
\end{align*}
By Lemma~\ref{lemma:evidence_diff_bound:ps}, the term $\lvert Z_k(\mu) -Z_k(\mu^\prime)\rvert$ satisfies
\begin{equation*}
\lvert Z_k(\mu) -Z_k(\mu^\prime)\rvert \leq \tilde{C}_{Th}^*(y_k;k)W_1(\mu,\mu^\prime).
\end{equation*}
It follows that
\begin{align}
&\frac{\lvert Z_k(\mu)-Z_k(\mu^\prime) \rvert}{Z_k(\mu)Z_k(\mu^\prime)}\sup_{f:\mathcal{X}\times \W \rightarrow \mathbb{R}, \lVert f \rVert_{\text{Lip}}\leq 1,f(\bar{x},\bar{w})=0}\bigg \lvert \int_{\mathcal{X}\times \W}f(x,w)\tilde{F}_k\mu^\prime(dxdw)  \bigg \rvert \nonumber \\
& \nonumber \\
&\qquad \leq \frac{D\tilde{C}_{Th}^*(y_k;k)}{Z_k(\mu)}W_1(\mu,\mu^\prime).
\label{ieq:ps:W:term2}
\end{align}
Plugging Equations~\eqref{ieq:ps:W:term1} and ~\eqref{ieq:ps:W:term2} into Equation~\eqref{ieq:ps:W} gives
\begin{equation*}
W_1(F_k(\mu),F_k(\mu^\prime)) \leq \frac{2D\tilde{C}_{Th}^*(y_k;k)+\tilde{C}_{Th}(y_k;k)}{Z_k(\mu)}W_1(\mu,\mu^\prime).
\end{equation*}
\end{proof}

\section{Proofs of Theorem~\ref{thm:learning_error}}
\label{sec:app:learning_error}
\begin{proof}
We first prove that Theorem~\ref{thm:learning_error} holds for inverse problems under the total variation distance. 

By the definitions of the evidence $Z_k$ and the true posterior $P_k$ for inverse problems, we have
\begin{align*}
Z_k(P_{k-1}) =\int_{\X}h(y_k,x)P_{k-1}(dx) 
= \int_{\X}p(y_k \mid x, \Y_{1:k-1})P_{k-1}(dx) 
= p(y_k \mid \Y_{1:k-1}),
\end{align*}
where the third equality is obtained because $P_{k-1}(dx)$ represent the probability of $dx$ given $\Y_{1:k-1}$. The remainder of the proof proceeds by mathematical induction.\\
Statement: for any integer $k \geq 2$, we have
\begin{equation*}
d_{TV}(P_k,Q_k) \leq \sum_{j=1}^{k-1}\frac{\prod_{i=j+1}^kC_h(y_i)}{p(\Y_{j+1:k}\mid \Y_{1:j})}d_{TV}(Q_j^*,Q_j)+d_{TV}(Q_k^*,Q_k),
\end{equation*}
and
\begin{equation*}
d_{TV}(P_k,Q_k) \leq \sum_{j=1}^{k-1}\left(\prod_{i=j+1}^k\frac{C_h(y_i)}{Z_i(Q_{i-1})}\right)d_{TV}(Q_j^*,Q_j)+d_{TV}(Q_k^*,Q_k).
\end{equation*}
Base case: Consider $k=2$. By the triangle inequality satisfied by the metric $d_{TV}$, the learning error $d_{TV}(P_2, Q_2)$ can be bounded as
\begin{align*}
d_{TV}(P_2,Q_2) \leq d_{TV}(P_2,Q_2^*) + d_{TV}(Q_2^*,Q_2) 
=d_{TV}(F_2(P_1),F_2(Q_1)) + d_{TV}(Q_2^*,Q_2). 
\end{align*}
According to Theorem~\ref{thm:pointwise_lipschitz}, $d_{TV}(F_2(P_1),F_2(Q_1))$ satisfies
\begin{align*}
d_{TV}(F_2(P_1),F_2(Q_1)) \leq& \frac{C_h(y_2)}{Z_2(P_1)}d_{TV}(P_1,Q_1), \\
d_{TV}(F_2(P_1),F_2(Q_1)) \leq& \frac{C_h(y_2)}{Z_2(Q_1)}d_{TV}(P_1,Q_1). 
\end{align*}
As we have
\begin{equation*}
Q_1^* = F_1(P_0) = P_1,
\end{equation*}
it follows that 
\begin{align*}
d_{TV}(P_1,Q_1) \leq d_{TV}(P_1,Q_1^*)+ d_{TV}(Q_1^*,Q_1) = d_{TV}(Q_1^*,Q_1).
\end{align*}
Then we have
\begin{align*}
d_{TV}(P_2,Q_2) \leq& \frac{C_h(y_2)}{Z_2(P_1)}d_{TV}(Q_1^*,Q_1) + d_{TV}(Q_2^*,Q_2) \nonumber \\
=& \sum_{j=1}^{2-1}\frac{\prod_{i=j+1}^2C_h(y_i)}{p(\Y_{j+1:2}\mid \Y_{1:j})}d_{TV}(Q_j^*,Q_j)+d_{TV}(Q_2^*,Q_2),
\end{align*}
and 
\begin{align*}
d_{TV}(P_2,Q_2) \leq& \frac{C_h(y_2)}{Z_2(Q_1)}d_{TV}(Q_1^*,Q_1) + d_{TV}(Q_2^*,Q_2) \nonumber \\
=& \sum_{j=1}^{2-1}\left(\prod_{i=j+1}^2\frac{C_h(y_i)}{Z_i(Q_{i-1})}\right)d_{TV}(Q_j^*,Q_j)+d_{TV}(Q_2^*,Q_2).
\end{align*}
Therefore, we conclude that the statement holds for $k=2$.\\
Induction step: Assume the statement holds for some integer $k \geq 2$. As $d_{TV}$ satisfies the triangle inequality, the learning error at step $k+1$ satisfies
\begin{align}
d_{TV}(P_{k+1},Q_{k+1}) \leq& d_{TV}(P_{k+1},Q_{k+1}^*) + d_{TV}(Q_{k+1}^*,Q_{k+1}).
\label{ieq:le:1}
\end{align}
By Theorem~\ref{thm:pointwise_lipschitz}, we have
\begin{align}
d_{TV}(P_{k+1},Q_{k+1}^*) &= d_{TV}(F_{k+1}(P_k),F_{k+1}(Q_k)) \nonumber \\
&\leq \frac{C_h(y_{k+1})}{Z_{k+1}(P_{k})}d_{TV}(P_{k},Q_{k})  
= \frac{C_h(y_{k+1})}{p(y_{k+1}\mid \Y_{1:k})}d_{TV}(P_{k},Q_{k}),
\label{ieq:le:2}
\end{align}
and 
\begin{equation*}
d_{TV}(P_{k+1},Q_{k+1}^*) = d_{TV}(F_{k+1}(P_k),F_{k+1}(Q_k)) \leq \frac{C_h(y_{k+1})}{Z_{k+1}(Q_{k})}d_{TV}(P_{k},Q_{k}).
\end{equation*}
Plugging Equation~\eqref{ieq:le:2} into Equation~\eqref{ieq:le:1} gives
\begin{align*}
&d_{TV}(P_{k+1},Q_{k+1}) \leq \frac{C_h(y_{k+1})}{p(y_{k+1}\mid \Y_{1:k})}d_{TV}(P_{k},Q_{k}) + d_{TV}(Q_{k+1}^*,Q_{k+1}) \\
&\qquad \leq \frac{C_h(y_{k+1})}{p(y_{k+1}\mid \Y_{1:k})}\left(\sum_{j=1}^{k-1}\frac{\prod_{i=j+1}^kC_h(y_i)}{p(\Y_{j+1:k}\mid \Y_{1:j})}d_{TV}(Q_j^*,Q_j)+d_{TV}(Q_k^*,Q_k)\right) \\
& \qquad \qquad \qquad \qquad + d_{TV}(Q_{k+1}^*,Q_{k+1}) \\
&\qquad = \frac{C_h(y_{k+1})}{p(y_{k+1}\mid \Y_{1:k})}\sum_{j=1}^{k-1}\frac{\prod_{i=j+1}^kC_h(y_i)}{p(\Y_{j+1:k}\mid \Y_{1:j})}d_{TV}(Q_j^*,Q_j) \\
&\qquad \qquad \qquad \qquad + \frac{C_h(y_{k+1})}{p(y_{k+1}\mid \Y_{1:k})}d_{TV}(Q_k^*,Q_k) + d_{TV}(Q_{k+1}^*,Q_{k+1}) \\
&\qquad = \sum_{j=1}^{k-1}\frac{\prod_{i=j+1}^{k+1}C_h(y_i)}{p(\Y_{j+1:k+1}\mid \Y_{1:j})}d_{TV}(Q_j^*,Q_j) + \frac{C_h(y_{k+1})}{p(y_{k+1}\mid \Y_{1:k})}d_{TV}(Q_k^*,Q_k) \\
&\qquad \qquad \qquad \qquad  + d_{TV}(Q_{k+1}^*,Q_{k+1}) \\
&\qquad =\sum_{j=1}^{k}\frac{\prod_{i=j+1}^{k+1}C_h(y_i)}{p(\Y_{j+1:k+1}\mid \Y_{1:j})}d_{TV}(Q_j^*,Q_j)+d_{TV}(Q_{k+1}^*,Q_{k+1}).
\end{align*}
Similarly, we can prove that
\begin{equation*}
d_{TV}(P_{k+1},Q_{k+1}) \leq \sum_{j=1}^{k}\left(\prod_{i=j+1}^{k+1}\frac{C_h(y_i)}{Z_i(Q_{i-1})}\right)d_{TV}(Q_j^*,Q_j)+d_{TV}(Q_{k+1}^*,Q_{k+1}).
\end{equation*}
Hence, the statement holds for $k+1$.\\
Conclusion: the statement holds for any $k\geq2$.

So far, we have shown that Theorem~\ref{thm:learning_error} holds for inverse problems under the total variation distance. Analogously, we can prove that Theorem~\ref{thm:learning_error} holds for inverse problems under the Hellinger and $1$-Wasserstein distances.

For state estimation problems, by the definition of $Z_k$, we can write $Z_k(P_{k-1})$  as
\begin{align}
Z_k(P_{k-1}) =& \int_{\X}h_k(y_k,x_k)\left(\int_{\X}T_k(x_k,x_{k-1})P_{k-1}(dx_{k-1})\right)dx_k \nonumber \\
=& \int_{\X}p(y_k \mid x_k)p(x_k \mid \Y_{1:k-1})dx_k \nonumber \\
=& p(y_k \mid \Y_{1:k-1}).
\label{le:se:1}
\end{align}
Similarly, for parameter-state estimation problems, the evidence term $Z_k(P_{k-1})$ can be expressed as
\begin{align}
Z_k(P_{k-1}) =& \int_{\X \times \W}h_k(y_k,x_k,w)\left(\int_{\X}T_k(x_k,x_{k-1},w)\pi_{k-1}(x_{k-1},w)dx_{k-1}\right)dx_kdw \nonumber \\
=& \int_{\X \times \W}p(y_k \mid x_k,w)p(x_k,w \mid \Y_{1:k-1})dx_kdw \nonumber \\
=& p(y_k \mid \Y_{1:k-1}),
\label{le:ps:1}
\end{align}
where $\pi_{k-1}$ denotes the density of $P_{k-1}$ with respect to the Lebesgue measure.

Given Equations~\eqref{le:se:1} and ~\eqref{le:ps:1}, the proofs of Theorem~\ref{thm:learning_error} for state estimation and parameter-state estimation problems are analogous to the proof of Theorem~\ref{thm:learning_error} for inverse problems under the total variation distance.

\end{proof}

\section{Proofs of error reduction theorems: Theorems~\ref{thm:er:all:tv},~\ref{thm:er:all:H},~\ref{thm:er:ip:W}, and ~\ref{thm:er:se&ps:W}}
\subsection{Proof of Theorem~\ref{thm:er:all:tv}}
\label{sec:app:er:tv}
Theorem~\ref{thm:er:all:tv} is a conclusion of Propositions~\ref{prop:ip:er:tv}---\ref{prop:ps:er:tv}.
\begin{prop}\label{prop:ip:er:tv}
For inverse problems, given the data $y_k$, assume that $P_{k-1} \in \bar{\mathcal{P}}_k(\X)$ and $Q_{k-1} \in \bar{\mathcal{P}}_k(\X)$. Suppose there exists a $\sigma$-finite measure $\nu$ defined on $(\X,\mathcal{B}(\X))$ such that $P_{k-1} \ll \nu, Q_{k-1} \ll \nu$, and 
\begin{align*}
&\int_{\X^*}h(y_k,x)\bigg \lvert \frac{dP_{k-1}}{d\nu}(x)-\frac{dQ_{k-1}}{d\nu}(x) \bigg \rvert \nu(dx)  \\
&\qquad \leq \int_{\X^*}h(y_k,x)\nu(dx)\int_{\X^*}\bigg \lvert \frac{dP_{k-1}}{d\nu}(x)-\frac{dQ_{k-1}}{d\nu}(x) \bigg \rvert \nu(dx),
\end{align*}
and 
\begin{equation*}
\int_{\X^*}h(y_k,x)\nu(dx) \leq  \int_{\mathcal{X}}h(y_k,x)P_{k-1}(dx) \vee \int_{\mathcal{X}}h(y_k,x)Q_{k-1}(dx) ,
\end{equation*}
where
\begin{equation*}
\X^* := \begin{cases} \{x \in \mathcal{X}: \frac{dP_{k-1}}{d\nu}(x) \geq \frac{dQ_{k-1}}{d\nu}(x)\}, \quad \text{if } Z_k(P_{k-1}) \geq Z_k(Q_{k-1}), \nonumber \\
\{x \in \mathcal{X}: \frac{dP_{k-1}}{d\nu}(x) \leq \frac{dQ_{k-1}}{d\nu}(x)\}, \quad \text{otherwise}. \end{cases}
\end{equation*}
Then the posteriors $P_{k}$ and $Q_k^*$ satisfy
\begin{equation*}
d_{TV}(P_{k},Q_k^*) \leq d_{TV}(P_{k-1},Q_{k-1}).
\end{equation*}
\end{prop}

\begin{proof}
Since $P_k=F_k(P_{k-1})$ and $Q_k^*=F_k(Q_{k-1})$, according to Equations~\eqref{ieq:intermediate:tv:ip} and~\eqref{eq:ip:normalized:prior_diff:rewrite}, we have
\begin{align}
d_{TV}(P_k,Q_k^*) \leq & 
\frac{1}{Z_k(P_{k-1})\vee Z_k(Q_{k-1})}\int_{\X^*} h(y_k,x) \bigg \lvert  \frac{dP_{k-1}}{d\nu}(x) - \frac{dQ_{k-1}}{d\nu}(x) \bigg \rvert \nu(dx),
\label{ieq:er:ip:tv:1}
\end{align}
and 
\begin{equation}
d_{TV}(P_{k-1},Q_{k-1}) = \int_{\X^*}\bigg \lvert \frac{dP_{k-1}}{d\nu}(x) - \frac{dQ_{k-1}}{d\nu}(x) \bigg \rvert \nu(dx).
\label{eq:er:ip:tv:1}
\end{equation}
If $d_{TV}(P_{k-1},Q_{k-1})=0$, then $P_{k-1} = Q_{k-1}$. It follows that $F_k(P_{k-1}) = F_k(Q_{k-1})$ and $d_{TV}(P_k,Q_k^*)=0$.
If $d_{TV}(P_{k-1},Q_{k-1}) > 0$, we have
\begin{align*}
\frac{d_{TV}(P_k,Q_k^*)}{d_{TV}(P_{k-1},Q_{k-1})} \leq& \frac{1}{Z_k(P_{k-1}) \vee Z_k(Q_{k-1})}\frac{\int_{\X^*}h(y_k, x)\bigg \lvert \frac{dP_{k-1}}{d\nu}(x) - \frac{dQ_{k-1}}{d\nu}(x) \bigg \rvert \nu(dx)}{\int_{\X^*}\bigg \lvert \frac{dP_{k-1}}{d\nu}(x) - \frac{dQ_{k-1}}{d\nu}(x) \bigg \rvert \nu(dx)}.
\end{align*}
If $\nu$ satisfies
\begin{align*}
&\int_{\X^*}h(y_k,x)\bigg \lvert \frac{dP_{k-1}}{d\nu}(x)-\frac{dQ_{k-1}}{d\nu}(x) \bigg \rvert \nu(dx) \\
& \qquad \leq \int_{\X^*}h(y_k,x)\nu(dx)\int_{\X^*}\bigg \lvert \frac{dP_{k-1}}{d\nu}(x)-\frac{dQ_{k-1}}{d\nu}(x) \bigg \rvert \nu(dx),
\end{align*}
and 
\begin{equation*}
\int_{\X^*}h(y_k,x)\nu(dx) \leq  \int_{\mathcal{X}}h(y_k,x)P_{k-1}(dx) \vee \int_{\mathcal{X}}h(y_k,x)Q_{k-1}(dx) ,
\end{equation*}
then we have
\begin{align*}
\frac{d_{TV}(P_k,Q_k^*)}{d_{TV}(P_{k-1},Q_{k-1})} \leq& \frac{1}{Z_k(P_{k-1}) \vee Z_k(Q_{k-1})}\frac{\int_{\X^*}h(y_k, x)\bigg \lvert \frac{dP_{k-1}}{d\nu}(x) - \frac{dQ_{k-1}}{d\nu}(x) \bigg \rvert \nu(dx)}{\int_{\X^*}\bigg \lvert \frac{dP_{k-1}}{d\nu}(x) - \frac{dQ_{k-1}}{d\nu}(x) \bigg \rvert \nu(dx)} \nonumber \\
=& \frac{\int_{\X^*}h(y_k, x)\nu(dx)}{\int_\mathcal{X}h(y_k,x)P_{k-1}(dx) \vee \int_\mathcal{X}h(y_k,x)Q_{k-1}(dx)} \nonumber \\
\leq & 1.
\end{align*}
The statement of the theorem is proved.
\end{proof}

\begin{prop}\label{prop:se:er:tv}
In state estimation problems, given the data $y_k$, assume $P_{k-1} \in \bar{\mathcal{P}}_k(\X)$ and $Q_{k-1} \in \bar{\mathcal{P}}_k(\X)$. Suppose there exists a $\sigma$-finite measure $\nu$ defined on $(\X,\mathcal{B}(\X))$ such that $P_{k-1} \ll \nu, Q_{k-1} \ll \nu$, and 
\begin{align}
&\int_{\X^*}g(x,y_k)\bigg \lvert \frac{dP_{k-1}}{d\nu}(x)-\frac{dQ_{k-1}}{d\nu}(x) \bigg \rvert \nu(dx)  \\
&\qquad \leq \int_{\X^*}g(x,y_k)\nu(dx)\int_{\X^*}\bigg \lvert \frac{dP_{k-1}}{d\nu}(x)-\frac{dQ_{k-1}}{d\nu}(x) \bigg \rvert \nu(dx),
\label{ieq:er:se:tv:assump_1}
\end{align}
and 
\begin{equation}
\int_{\X^*}g(x,y_k)\nu(dx) \leq  \int_{\mathcal{X}}g(x,y_k)P_{k-1}(dx) \vee \int_{\mathcal{X}}g(x,y_k)Q_{k-1}(dx),
\label{ieq:er:se:tv:assump_2}
\end{equation}
where $g(x,y_k):=\int_{\X}h_k(y_k,x_k)T_k(x_k,x)dx_k$, and 
\begin{equation*}
\X^* := \begin{cases} \{x \in \mathcal{X}: \frac{dP_{k-1}}{d\nu}(x) \geq \frac{dQ_{k-1}}{d\nu}(x)\}, \quad \text{if } Z_k(P_{k-1}) \geq Z_k(Q_{k-1}), \nonumber \\
\{x \in \mathcal{X}: \frac{dP_{k-1}}{d\nu}(x) \leq \frac{dQ_{k-1}}{d\nu}(x)\}, \quad \text{otherwise}. \end{cases}
\end{equation*}
Then the posteriors $P_k$ and $Q_k^*$ satisfy
\begin{equation*}
d_{TV}(P_k,Q_k^*) \leq d_{TV}(P_{k-1},Q_{k-1}).
\end{equation*}
\end{prop}
\begin{proof}
Let $\nu$ be a $\sigma$-finite measure defined on $(\X,\mathcal{B}(\X))$ such that $P_{k-1} \ll \nu$, $Q_{k-1} \ll \nu$. Since $P_k=F_k(P_{k-1})$ and $Q_k^*=F_k(Q_{k-1})$, given Equations~\eqref{eq:se:tv:def:evidence_mu},~\eqref{eq:se:tv:def:evidence_mu_prime},~\eqref{ieq:intermediate:tv:se}, and~\eqref{ieq:intermediate:tv:se_prior}, and assuming that $\nu$ satisfies the conditions in Equations~\eqref{ieq:er:se:tv:assump_1} and~\eqref{ieq:er:se:tv:assump_2}, the conclusion of Proposition~\ref{prop:se:er:tv} follows by an argument analogous to that used in the proof of Proposition~\ref{prop:ip:er:tv}.
\end{proof}

\begin{prop}\label{prop:ps:er:tv}
In parameter-state estimation problems, given the data $y_k$, suppose $P_{k-1} \in \bar{\mathcal{P}}_k(\X\times \W)$ and $Q_{k-1} \in \bar{\mathcal{P}}_k(\X\times \W)$. Assume the state transition model $h_k$, the observation model $T_k$, and priors $P_{k-1}$ and $Q_{k-1}$ satisfy 
\begin{align}
&\int_{\mathcal{S}^*}g(x,w,y_k)\bigg \lvert \frac{dP_{k-1}}{d\lambda}(x,w)-\frac{dQ_{k-1}}{d\lambda}(x,w) \bigg \rvert dxdw  \nonumber \\
&\qquad \leq \int_{\mathcal{S}^*}g(x,w,y_k)dxdw\int_{\mathcal{S}^*}\bigg \lvert \frac{dP_{k-1}}{d\lambda}(x,w)-\frac{dQ_{k-1}}{d\lambda}(x,w) \bigg \rvert dxdw,
\label{ieq:er:ps:tv:assump_1}
\end{align}
and 
\begin{equation}
\int_{\mathcal{S}^*}g(x,w,y_k)dxdw \leq  \int_{\mathcal{X}\times \W}g(x,w,y_k)dP_{k-1}(x,w) \vee \int_{\mathcal{X}\times \W}g(x,w,y_k)dQ_{k-1}(x,w),
\label{ieq:er:ps:tv:assump_2}
\end{equation}
where $\lambda$ is the Lebesgue measure defined on the space $(\X\times \W,\mathcal{B}(\X\times \W))$, and
\begin{equation*}
g(x,w,y_k):=\int_{\X}h_k(y_k,x_k,w)T_k(x_k,x,w)dx_k,
\end{equation*}
and
\begin{equation*}
\mathcal{S}^* := \begin{cases} \{(x,w) \in \mathcal{X}\times \W: \frac{dP_{k-1}}{d\lambda}(x,w) \geq \frac{dQ_{k-1}}{d\lambda}(x,w)\}, \quad \text{if } Z_k(P_{k-1}) \geq Z_k(Q_{k-1}), \nonumber \\
\{(x,w) \in \mathcal{X}\times \W: \frac{dP_{k-1}}{d\lambda}(x,w) \leq \frac{dQ_{k-1}}{d\lambda}(x,w)\}, \quad \text{otherwise}. \end{cases}
\end{equation*}
Then the posteriors $P_k$ and $Q_k^*$ satisfy
\begin{equation*}
d_{TV}(P_k,Q_k^*) \leq d_{TV}(P_{k-1},Q_{k-1}).
\end{equation*}
\end{prop}

\begin{proof}
Since $P_k=F_k(P_{k-1})$ and $Q_k^*=F_k(Q_{k-1})$, given Equations~\eqref{eq:ps:tv:def:evidence_mu},~\eqref{eq:ps:tv:def:evidence_mu_prime},~\eqref{ieq:intermediate:tv:ps}, and~\eqref{ieq:intermediate:tv:ps_prior}, if Equations~\eqref{ieq:er:ps:tv:assump_1} and~\eqref{ieq:er:ps:tv:assump_2} hold, the conclusion of Proposition~\ref{prop:ps:er:tv} follows by an argument analogous to that used in the proof of Proposition~\ref{prop:ip:er:tv}.    
\end{proof}

\subsection{Proof of Theorem~\ref{thm:er:all:H}}
\label{sec:app:er:H}
\subsubsection{Proof of Theorem~\ref{thm:er:all:H} for inverse problems}
The proof of Theorem~\ref{thm:er:all:H} for inverse problems is as follows.
\begin{proof} 
First, we prove that the conclusion of Theorem~\ref{thm:er:all:H} holds if Assumption~\ref{assump:er:all:H:1} is satisfied. Let $\nu$ be a measure which satisfies the conditions in Assumption~\ref{assump:er:all:H:1}. By the definition of Hellinger distance, we can write $d_H^2(P_k, Q_k^*)$ as 
\begin{align*}
d_H^2(P_k, Q_k^*) =& \frac{1}{2}\int_\mathcal{X}\left(\sqrt{\frac{dP_k}{d\nu}(x)}- \sqrt{\frac{dQ_k^*}{d\nu}(x)}\right)^2\nu(dx) \nonumber \\
=&\frac{1}{2}\int_{\X}\frac{dP_k}{d\nu}(x)\nu(dx)+\frac{1}{2}\int_{\X}\frac{dQ_k^*}{d\nu}(x)\nu(dx) \\
&\qquad \qquad -\int_\mathcal{X}\sqrt{\frac{dP_k}{d\nu}(x)}\sqrt{\frac{dQ_k^*}{d\nu}(x)}\nu(dx).
\end{align*}
Note that we have
\begin{align*}
\int_{\X}\frac{dP_k}{d\nu}(x)\nu(dx)=& 1, \qquad 
\int_{\X}\frac{dQ_k^*}{d\nu}(x)\nu(dx) =1.
\end{align*}
Therefore, $d_H^2(P_k, Q_k^*)$ can be re-written as
\begin{align*}
d_H^2(P_k, Q_k^*) &= 1 - \int_\mathcal{X}\sqrt{\frac{dP_k}{d\nu}(x)}\sqrt{\frac{dQ_k^*}{d\nu}(x)}\nu(dx) \nonumber \\
& =1 - \frac{1}{\sqrt{Z_k(P_{k-1})Z_k(Q_{k-1})}}\int_\mathcal{X}h(y_k,x)\sqrt{\frac{dP_{k-1}}{d\nu}(x)\frac{dQ_{k-1}}{d\nu}(x)}\nu(dx).
\end{align*}
Similarly, we can write $d_H^2(P_{k-1},Q_{k-1})$ as
\begin{align*}
d_H^2(P_{k-1},Q_{k-1})=& 1 - \int_\mathcal{X}\sqrt{\frac{dP_{k-1}}{d\nu}(x)\frac{dQ_{k-1}}{d\nu}(x)}\nu(dx).
\end{align*}
Then the term $d_H^2(P_k, Q_k^*) - d_H^2(P_{k-1},Q_{k-1})$ can be expressed as
\begin{align*}
&d_H^2(P_k, Q_k^*) - d_H^2(P_{k-1},Q_{k-1}) 
= \int_\mathcal{X}\sqrt{\frac{dP_{k-1}}{d\nu}(x)\frac{dQ_{k-1}}{d\nu}(x)}\nu(dx) \nonumber \\
&\qquad - \frac{1}{\sqrt{Z_k(P_{k-1})Z_k(Q_{k-1})}}\int_\mathcal{X}h(y_k,x)\sqrt{\frac{dP_{k-1}}{d\nu}(x)\frac{dQ_{k-1}}{d\nu}(x)}\nu(dx).
\end{align*}
As $\nu$ satisfies the conditions in Assumption~\ref{assump:er:all:H:1}, we have
\begin{align*}
&\int_\mathcal{X}\sqrt{\frac{dP_{k-1}}{d\nu}(x)\frac{dQ_{k-1}}{d\nu}(x)}\nu(dx) \\
&\quad - \frac{1}{\sqrt{Z_k(P_{k-1})Z_k(Q_{k-1})}}\int_\mathcal{X}h(y_k,x)\sqrt{\frac{dP_{k-1}}{d\nu}(x)\frac{dQ_{k-1}}{d\nu}(x)}\nu(dx) \\
& \qquad \leq \int_\mathcal{X}\sqrt{\frac{dP_{k-1}}{d\nu}(x)\frac{dQ_{k-1}}{d\nu}(x)}\nu(dx) \\
&\qquad \qquad - \frac{\int_{\X}h(y_k,x)\nu(dx)}{\sqrt{Z_k(P_{k-1})Z_k(Q_{k-1})}}\int_\mathcal{X}\sqrt{\frac{dP_{k-1}}{d\nu}(x)\frac{dQ_{k-1}}{d\nu}(x)}\nu(dx) \nonumber \\
&\qquad = \int_\mathcal{X}\sqrt{\frac{dP_{k-1}}{d\nu}(x)\frac{dQ_{k-1}}{d\nu}(x)}\nu(dx)\left(1 - \frac{\int_{\X}h(y_k,x)\nu(dx)}{\sqrt{Z_k(P_{k-1})Z_k(Q_{k-1})}}\right) 
\leq 0.
\end{align*}

Next, we prove that if Assumption~\ref{assump:er:all:H:2} is satisfied, we then have $d_H(P_k, Q_k^*) \leq d_H(P_{k-1},Q_{k-1})$. 

We can re-write $d_H^2(P_k, Q_k^*)$ as
\begin{align*}
d_H^2(P_k, Q_k^*) &= 1 - \frac{1}{\sqrt{Z_k(P_{k-1})Z_k(Q_{k-1})}}\int_\mathcal{X}h(y_k,x)\sqrt{\frac{dP_{k-1}}{d\nu}(x)\frac{dQ_{k-1}}{d\nu}(x)}\nu(dx) \nonumber \\
& = 1 + \frac{1}{2}\frac{1}{\sqrt{Z_k(P_{k-1})Z_k(Q_{k-1})}}\int_\mathcal{X}h(y_k,x)\left({\left(\sqrt{\frac{dP_{k-1}}{d\nu}(x)}-\sqrt{\frac{dQ_{k-1}}{d\nu}(x)}\right)^2}\right. \\
&\qquad \qquad \left.{-\frac{dP_{k-1}}{d\nu}(x)-\frac{dQ_{k-1}}{d\nu}(x)}\right)\nu(dx) \nonumber \\
&= 1- \frac{1}{2}\left(\sqrt{\frac{Z_k(P_{k-1})}{Z_k(Q_{k-1})}} + \sqrt{\frac{Z_k(Q_{k-1})}{Z_k(P_{k-1})}}\right) \\
&\qquad + \frac{1}{2\sqrt{Z_k(P_{k-1})Z_k(Q_{k-1})}}\int_\mathcal{X}h(y_k,x)\left(\sqrt{\frac{dP_{k-1}}{d\nu}(x)}-\sqrt{\frac{dQ_{k-1}}{d\nu}(x)}\right)^2\nu(dx) \nonumber \\
&\leq  \frac{1}{2\sqrt{Z_k(P_{k-1})Z_k(Q_{k-1})}}\int_\mathcal{X}h(y_k,x)\left(\sqrt{\frac{dP_{k-1}}{d\nu}(x)}-\sqrt{\frac{dQ_{k-1}}{d\nu}(x)}\right)^2\nu(dx). 
\end{align*}
Hence, the term $d_H^2(P_k,Q_k^*) - d_H^2(P_{k-1},Q_{k-1})$ satisfies 
\begin{align}
&d_H^2(P_k,Q_k^*) - d_H^2(P_{k-1},Q_{k-1}) \nonumber \\
&\qquad \leq \frac{1}{2\sqrt{Z_k(P_{k-1})Z_k(Q_{k-1})}}\int_\mathcal{X}h(y_k,x)\left(\sqrt{\frac{dP_{k-1}}{d\nu}(x)}-\sqrt{\frac{dQ_{k-1}}{d\nu}(x)}\right)^2\nu(dx) \nonumber \\
&\qquad \qquad \qquad - \frac{1}{2}\int_\mathcal{X}\left(\sqrt{\frac{dP_{k-1}}{d\nu}(x)}-\sqrt{\frac{dQ_{k-1}}{d\nu}(x)}\right)^2\nu(dx).
\label{ieq:ip:er:H:assump2:new:1}
\end{align}
When $\nu$ satisfies the conditions in Assumption~\ref{assump:er:all:H:2}, the right-hand side of Equation~\eqref{ieq:ip:er:H:assump2:new:1} satisfies
\begin{align*}
&\frac{1}{2\sqrt{Z_k(P_{k-1})Z_k(Q_{k-1})}}\int_\mathcal{X}h(y_k,x)\left(\sqrt{\frac{dP_{k-1}}{d\nu}(x)}-\sqrt{\frac{dQ_{k-1}}{d\nu}(x)}\right)^2\nu(dx) \\
&\qquad- \frac{1}{2}\int_\mathcal{X}\left(\sqrt{\frac{dP_{k-1}}{d\nu}(x)}-\sqrt{\frac{dQ_{k-1}}{d\nu}(x)}\right)^2\nu(dx) \\
& \qquad \qquad\leq \frac{1}{2}\int_\mathcal{X}\left(\sqrt{\frac{dP_{k-1}}{d\nu}(x)}-\sqrt{\frac{dQ_{k-1}}{d\nu}(x)}\right)^2\nu(dx) \\
&\qquad \qquad \qquad \cdot \left(\frac{\int_\mathcal{X}h(y_k,x)\nu(dx)}{\sqrt{\int_{\X}h(y_k,x)P_{k-1}(dx)}\sqrt{\int_{\X}h(y_k,x)Q_{k-1}(dx)}} - 1 \right) \nonumber \\
& \qquad \qquad \leq0.
\end{align*}
The statement of Theorem~\ref{thm:er:all:H} for inverse problems is proved.
\end{proof}

\subsubsection{Proof of Theorem~\ref{thm:er:all:H} for state estimation problems}
The proof of Theorem~\ref{thm:er:all:H} for state estimation problems is as follows.
\begin{proof} 
First, we prove that if Assumption~\ref{assump:er:all:H:1} is satisfied, we have $d_H(P_k,Q_k^*) \leq d_H(P_{k-1},Q_{k-1})$. Let $\nu$ be a measure which satisfies the conditions in Assumption SE.ER.1. Let $\pi$ and $\pi^\prime$ denote the Radon-Nikodym derivatives $\frac{dP_{k-1}}{d\nu}$ and $\frac{dQ_{k-1}}{d\nu}$, respectively. Then we can write $d_H^2(P_k,Q_k^*)$ as 
\begin{align}
d_H^2(P_k,Q_k^*) =& \frac{1}{2}\int_\mathcal{X}\left({\sqrt{\frac{h_k(y_k,x)\int_{\X}T_k(x,x_{k-1})P_{k-1}(dx_{k-1})}{Z_k(P_{k-1})}}}\right. \nonumber \\
&\qquad \qquad \left.{- \sqrt{\frac{h_k(y_k,x)\int_{\X}T_k(x,x_{k-1})Q_{k-1}(dx_{k-1})}{Z_k(Q_{k-1})}}}\right)^2dx \nonumber \\
=& 1 - \frac{1}{\sqrt{Z_k(P_{k-1})Z_k(Q_{k-1})}}\int_\mathcal{X}h_k(y_k, x)\sqrt{\int_\mathcal{X}T_k(x,x_{k-1})\pi(x_{k-1})\nu(dx_{k-1})} \nonumber \\
&\qquad \qquad \cdot \sqrt{\int_\mathcal{X}T_k(x, x_{k-1})\pi^\prime(x_{k-1})\nu(dx_{k-1})}dx.
\label{eq:er:se:new:1}
\end{align}
Using Hölder's inequality, we obtain
\begin{align}
&\sqrt{\int_\mathcal{X}T_k(x,x_{k-1})\pi(x_{k-1})\nu(dx_{k-1})} \sqrt{\int_\mathcal{X}T_k(x, x_{k-1})\pi^\prime(x_{k-1})\nu(dx_{k-1})} \nonumber \\  
&\qquad \geq \int_{\X}T_k(x,x_{k-1})\sqrt{\pi(x_{k-1})\pi^\prime(x_{k-1})}\nu(dx_{k-1}) 
\label{ieq:er:se:new:1}
\end{align}
Substituting Equation~\eqref{ieq:er:se:new:1} into Equation~\eqref{eq:er:se:new:1} yields
\begin{align*}
&d_H^2(P_k,Q_k^*) \\
&\quad \leq 1 - \frac{\int_\mathcal{X}h_k(y_k, x)\left(\int_\mathcal{X}T_k(x,x_{k-1})\sqrt{\pi(x_{k-1})}\sqrt{\pi^\prime(x_{k-1})}\nu(dx_{k-1})\right)dx}{\sqrt{Z_k(P_{k-1})Z_k(Q_{k-1})}}.
\end{align*}
By the definition of the Hellinger distance, we can write $d_H^2(P_{k-1},Q_{k-1})$ as
\begin{align*}
d_H^2(P_{k-1},Q_{k-1}) =& \frac{1}{2}\int_\mathcal{X}\left(\sqrt{\pi(x_{k-1})}-\sqrt{\pi^\prime(x_{k-1})}\right)^2\nu(dx_{k-1}) \nonumber \\
=& 1 - \int_\mathcal{X}\sqrt{\pi(x_{k-1})\pi^\prime(x_{k-1})}\nu(dx_{k-1}).
\end{align*}
Thus the term $d_H^2(P_k,Q_k^*)-d_H^2(P_{k-1},Q_{k-1})$ satisfies
\begin{align*}
&d_H^2(P_k,Q_k^*)-d_H^2(P_{k-1},Q_{k-1}) \nonumber \\
&\qquad \leq \int_\mathcal{X}\sqrt{\pi(x_{k-1})\pi^\prime(x_{k-1})}\nu(dx_{k-1}) \\
&\qquad \qquad \qquad - \frac{\int_\mathcal{X}h_k(y_k, x)\left(\int_\mathcal{X}T_k(x,x_{k-1})\sqrt{\pi(x_{k-1})}\sqrt{\pi^\prime(x_{k-1})}\nu(dx_{k-1})\right)dx}{\sqrt{Z_k(P_{k-1})Z_k(Q_{k-1})}}.
\end{align*}
By Tonelli's theorem, we have
\begin{align*}
&\int_\mathcal{X}h_k(y_k, x)\left(\int_\mathcal{X}T_k(x,x_{k-1})\sqrt{\pi(x_{k-1})}\sqrt{\pi^\prime(x_{k-1})}\nu(dx_{k-1})\right)dx \\
&\qquad = \int_{\X}\left(\int_{\X}h_k(y_k,x)T_k(x,x_{k-1})dx\right)\sqrt{\pi(x_{k-1})\pi^\prime(x_{k-1})}\nu(dx_{k-1}).
\end{align*}
When $\nu$ satisfies the conditions in Assumption~\ref{assump:er:all:H:1}, an argument analogous to that used in the proof of Theorem~\ref{thm:er:all:H} for inverse problems yields
\begin{align*}
& d_H^2(P_k,Q_k^*)-d_H^2(P_{k-1},Q_{k-1}) \leq 0.
\end{align*}

Next, we prove that under Assumption~\ref{assump:er:all:H:2}, we have $d_H(P_k,Q_k^*) \leq d_H(P_{k-1},Q_{k-1})$.

Following an argument similar to that is used in the proof for Theorem~\ref{thm:er:all:H} for inverse problems, we can show that
\begin{align}
&d_H^2(P_k,Q_k^*) \nonumber \\
& \leq \frac{1}{2\sqrt{Z_k(P_{k-1})Z_k(Q_{k-1})}}\int_\mathcal{X}h_k(y_k, x) \nonumber \\ 
&\qquad \cdot \left(\sqrt{\int_\mathcal{X}T_k(x,x_{k-1})\pi(x_{k-1})\nu(dx_{k-1})} - \sqrt{\int_\mathcal{X}T_k(x,x_{k-1})\pi^\prime(x_{k-1})\nu(dx_{k-1})} \right)^2dx \nonumber \\
& \quad \leq \frac{\int_{\mathcal{X}}h_k(y_k,x)\left(\int_{\X}T_k(x,x_{k-1})\left(\sqrt{\pi(x_{k-1})}-\sqrt{\pi^\prime(x_{k-1})}\right)^2\nu(dx_{k-1})\right)dx}{2\sqrt{Z_k(P_{k-1})Z_k(Q_{k-1})}}.
\label{ieq:se:er:H:assump2:new:1}
\end{align}
The second inequality in Equation~\eqref{ieq:se:er:H:assump2:new:1} is obtained according to Equation~\eqref{ieq:er:se:new:1}.

By Tonelli's theorem, we have
\begin{align*}
& \int_{\mathcal{X}}h_k(y_k,x)\left(\int_{\X}T_k(x,x_{k-1})\left(\sqrt{\pi(x_{k-1})}-\sqrt{\pi^\prime(x_{k-1})}\right)^2\nu(dx_{k-1})\right)dx \\
&\quad = \int_{\mathcal{X}}\left(\int_{\X}h_k(y_k,x)T_k(x,x_{k-1})dx\right)\left(\sqrt{\pi(x_{k-1})}-\sqrt{\pi^\prime(x_{k-1})}\right)^2\nu(dx_{k-1}).
\end{align*}
Then we can prove that $d_H(P_k,Q_k^*) \leq d_H(P_{k-1},Q_{k-1})$ holds under Assumption~\ref{assump:er:all:H:2} by following an argument analogous to that used in the proof of Theorem~\ref{thm:er:all:H} for inverse problems.
\end{proof}
\subsubsection{Proof of Theorem~\ref{thm:er:all:H} for parameter-state estimation problems}

The proof of Theorem~\ref{thm:er:all:H} for parameter-state estimation problems is analogous to the proof of Theorem~\ref{thm:er:all:H} for state estimation problems. 

\subsection{Proof of Theorem~\ref{thm:er:ip:W}}
\label{sec:app:er:ip:W}
\begin{proof}
For the metric space $\left(\mathcal{X},d_{\X}\right)$, the $1$-Wasserstein distance between the posteriors $P_k$ and $Q_k^*$ is defined as
\begin{equation*}
W_1(P_k,Q_k^*) := \inf_{\gamma \in \Gamma(P_k,Q_k^*)}\int_{\mathcal{X}\times\mathcal{X}}d_{\X}(x_1,x_2)\gamma(dx_1,dx_2),
\end{equation*}
where $\Gamma(P_k,Q_k^*)$ denotes the set of all couplings of $P_k$ and $Q_k^*$. 
Define a probability measure $\gamma_0: \mathcal{X} \times \mathcal{X} \rightarrow \mathbb{R}$ as 
\begin{equation*}
\gamma_0(dx,dx^\prime) := P_k(dx)Q_k^*(dx^\prime).
\end{equation*}
Then $\gamma_0$ is a coupling of $P_k$ and $Q_k^*$. It follows that
\begin{align}
W_1(P_k,Q_k^*) & 
\leq \int_{\mathcal{X}\times \mathcal{X}}d_{\X}(x,x^\prime)P_k(dx)Q_k^*(dx^\prime) \nonumber \\
& =\frac{1}{Z_k(P_{k-1})Z_k(Q_{k-1})}\int_{\mathcal{X}\times \mathcal{X}}d_{\X}(x,x^\prime)h(y_k,x)h(y_k, x^\prime)P_{k-1}(dx)Q_{k-1}(dx^\prime) \nonumber \\
&  = \frac{\E_{(X,X^\prime)\sim P_{k-1} \otimes Q_{k-1}}[d_{\X}(X,X^\prime)h(y_k,X)h(y_k,X^\prime)]}{E_{X \sim P_{k-1}}[h(y_k,X)]E_{X \sim Q_{k-1}}[h(y_k,X)]}.
\label{ieq:er:ip:W:1}
\end{align}
According to the Kantorovich-Rubinstein duality, the $1$-Wasserstein distance between the two priors $P_{k-1}$ and $Q_{k-1}$ can be written as
\begin{equation*}
W_1(P_{k-1},Q_{k-1}) = \sup_{f:\mathcal{X} \rightarrow \mathbb{R}, \lVert f \rVert_{\text{Lip}} \leq 1}\bigg \lvert \int_{\mathcal{X}}f(x)P_{k-1}(dx) - \int_{\mathcal{X}}f(x)Q_{k-1}(dx) \bigg \rvert.
\end{equation*}
Define a function $g: \mathcal{X} \rightarrow \mathbb{R}$ as
\begin{equation*}
g(x) := d_{\X}(x,x_0),
\end{equation*}
where $x_0 \in \mathcal{X}$ is an arbitrary point in space $\mathcal{X}$. We first show that the function $g$ is Lipschitz continuous on metric space $\left(\mathcal{X},d_{\X}\right)$ and its best Lipschitz constant $\lVert g \rVert_{\text{Lip}}$ is no greater than 1. \\
For two arbitrary points $x, x^\prime \in \mathcal{X}$, we have
\begin{equation*}
\lvert g(x) - g(x^\prime) \rvert = \lvert d_{\X}(x,x_0) - d_{\X}(x^\prime, x_0) \rvert.
\end{equation*}
Since metric $d_{\X}$ is symmetric and satisfies triangle inequality, we have
\begin{equation*}
d_{\X}(x,x_0) \leq d_{\X}(x,x^\prime) + d_{\X}(x^\prime,x_0), \nonumber 
\end{equation*}
and
\begin{equation*}
d_{\X}(x^\prime,x_0) \leq d_{\X}(x^\prime,x)+d_{\X}(x,x_0) = d_{\X}(x,x^\prime)+d_{\X}(x,x_0).
\end{equation*}
Then $d_{\X}(x,x^\prime)$ satisfies
\begin{equation*}
d_{\X}(x,x^\prime) \geq \lvert d_{\X}(x,x_0) - d_{\X}(x^\prime,x_0) \rvert.
\end{equation*}
It follows that
\begin{equation*}
\lVert g \rVert_{\text{Lip}} = \sup_{x,x^\prime \in \mathcal{X}, x \neq x^\prime}\frac{\lvert d_{\X}(x,x_0) - d_{\X}(x^\prime,x_0) \rvert}{d_{\X}(x,x^\prime)} \leq 1.
\end{equation*}
Then we have
\begin{equation*}
W_1(P_{k-1},Q_{k-1}) \geq \bigg \lvert \int_\mathcal{X}d_{\X}(x,x_0)P_{k-1}(dx) - \int_\mathcal{X}d_{\X}(x,x_0)Q_{k-1}(dx) \bigg \rvert.
\end{equation*}
As $x_0 \in \mathcal{X}$ is arbitrary, $W_1(P_{k-1},Q_{k-1})$ satisfies
\begin{align}
W_1(P_{k-1},Q_{k-1}) \geq& \sup_{x_0 \in \mathcal{X}}\bigg \lvert \int_\mathcal{X}d_{\X}(x,x_0)P_{k-1}(dx) - \int_\mathcal{X}d_{\X}(x,x_0)Q_{k-1}(dx) \bigg \rvert \nonumber \\
=& \sup_{x_0 \in \mathcal{X}}\bigg \lvert \E_{X \sim P_{k-1}}[d_{\X}(X,x_0)] - \E_{X \sim Q_{k-1}}[d_{\X}(X,x_0)] \bigg \rvert.
\label{ieq:er:ip:W:2}
\end{align}
Given Equations~\eqref{ieq:er:ip:W:1} and~\eqref{ieq:er:ip:W:2}, if the following inequality holds:
\begin{align*}
&\frac{\E_{(X,X^\prime)\sim P_{k-1} \otimes Q_{k-1}}[d_{\X}(X,X^\prime)h(y_k,X)h(y_k,X^\prime)]}{E_{X \sim P_{k-1}}[h(y_k,X)]E_{X \sim Q_{k-1}}[h(y_k,X)]} \\
&\qquad \leq \sup_{x_0 \in \mathcal{X}}\bigg \lvert \E_{X \sim P_{k-1}}[d_{\X}(X,x_0)] - \E_{X \sim Q_{k-1}}[d_{\X}(X,x_0)] \bigg \rvert,
\end{align*}
then we have
$$W_1(P_k,Q_k^*) \leq W_1(P_{k-1},Q_{k-1}).$$
\end{proof}
\subsection{Proof of Theorem~\ref{thm:er:se&ps:W}}
\label{sec:app:er:se&ps:W}
\begin{proof}
First, we prove that Theorem~\ref{thm:er:se&ps:W} holds for state estimation problems. 

For state estimation problems, define a probability measure $\gamma_0$ as 
\begin{equation*}
\gamma_0(dx,dx^\prime) = P_k(dx)Q_k^*(dx^\prime).
\end{equation*}
Then $\gamma_0$ is a coupling of $P_k$ and $Q_k^*$. According to the definition of $1$-Wasserstein, we have
\begin{equation*}
W_1(P_k,Q_k^*) \leq \int_{\X\times \X}d_{\X}(x,x^\prime)P_k(dx)Q_k^*(dx^\prime). 
\end{equation*}
By the definition of $F_k$ for state estimation problems, $P_k(dx)$ can be written as
\begin{equation*}
P_k(dx) = \frac{h_k(y_k,x)P_k^-(dx)}{Z_k(P_{k-1})}.
\end{equation*}
Similarly, we can express $Q_k^*(dx^\prime)$ as
\begin{equation*}
Q_k^*(dx^\prime) = \frac{h_k(y_k,x^\prime)Q_k^-(dx^\prime)}{Z_k(Q_{k-1})}.
\end{equation*}
It follows that 
\begin{align}
W_1(P_k,Q_k^*) \leq& \frac{1}{Z_k(P_{k-1})Z_k(Q_{k-1})}\int_{\X\times\X}h_k(y_k,x)h_k(y_k,x^\prime)P_k^-(dx)Q_k^-(dx^\prime) \nonumber \\
=& \frac{\E_{(X,X^\prime)\sim P_k^- \otimes Q_k^-}[d_{\X}(X,X^\prime)h_k(y_k,X)h_k(y_k,X^\prime)]}{\E_{X \sim P_k^-}[h_k(y_k,X)]\E_{X \sim Q_k^-}[h_k(y_k,X)]}.
\label{ieq:er:es:W:1}
\end{align}

According to the Kantorovich-Rubinstein duality, the $1$-Wasserstein distance between the two priors $P_{k-1}$ and $Q_{k-1}$ can be written as
\begin{equation*}
W_1(P_{k-1},Q_{k-1}) = \sup_{f:\mathcal{X} \rightarrow \mathbb{R}, \lVert f \rVert_{\text{Lip}}\leq 1}\bigg \lvert \int_{\mathcal{X}}f(x)p_{k-1}(x)dx - \int_{\mathcal{X}}f(x)q_{k-1}(x)dx \bigg \rvert.
\end{equation*}
Following an arguments similar to that used in the proof of Theorem~\ref{thm:er:ip:W}, we can show that 
\begin{equation}
W_1(P_{k-1},Q_{k-1}) \geq \sup_{\hat{x} \in \mathcal{X}}\bigg \lvert \E_{X \sim P_{k-1}}[d_{\X}(X,\hat{x})]-\E_{X \sim Q_{k-1}}[d_{\X}(X,\hat{x})] \bigg \rvert.
\label{ieq:er:es:W:2}
\end{equation}
Under the conditions given in Theorem~\ref{thm:er:se&ps:W}, we have
\begin{align*}
&\frac{\E_{(X,X^\prime)\sim P_k^- \otimes Q_k^-}[d_{\X}(X,X^\prime)h_k(y_k,X)h_k(y_k,X^\prime)]}{\E_{X \sim P_k^-}[h_k(y_k,X)]\E_{X \sim Q_k^-}[h_k(y_k,X)]} \\
&\qquad \leq \sup_{\hat{x} \in \mathcal{X}}\bigg \lvert \E_{X \sim P_{k-1}}[d_{\X}(X,\hat{x})]-\E_{X \sim Q_{k-1}}[d_{\X}(X,\hat{x})] \bigg \rvert.
\end{align*}
Given Equations~\eqref{ieq:er:es:W:1} and~\eqref{ieq:er:es:W:2}, it follows that
\begin{equation*}
W_1(P_k,Q_k^*) \leq W_1(P_{k-1},Q_{k-1}).
\end{equation*}
Analogously, Theorem~\ref{thm:er:se&ps:W} can be shown to hold for parameter-state estimation problems.
\end{proof}

\section{Proofs of Corollary~\ref{coro:example1} and Corollary~\ref{coro:example2}}
\subsection{Proof of Corollary~\ref{coro:example1}}
\label{sec:app:coro1}
\begin{proof}
Given the underlying system, for any integer $k\geq1$, we have
\begin{align*}
h_k(y_k,x,w) = p_{\mathcal{N}}(y_k \mid \Phi_k(x,w), \Gamma) \leq (2\pi)^{-\frac{r}{2}}\det (\Gamma)^{-\frac{1}{2}}.
\end{align*}
It follows that
\begin{align}
\tilde{C}_{Th}(y_k;k) &= \sup_{x_{k-1}\in \X, w\in \W}\int_{\X}h_k(y_k,x,w)T_k(x,x_{k-1},w)dx \nonumber \\
&\leq (2\pi)^{-\frac{r}{2}}\det (\Gamma)^{-\frac{1}{2}}\sup_{x_{k-1}\in \X, w\in \W}\int_{\X}T_k(x,x_{k-1},w)dx  \nonumber \\
&= (2\pi)^{-\frac{r}{2}}\det (\Gamma)^{-\frac{1}{2}} < \infty.
\label{ieq:vi:type1:C}
\end{align}
Hence Assumption PS.1 is satisfied for all $k \geq1$ and Theorem~\ref{thm:learning_error} can be applied to this problem under the total variation and Hellinger distance.

Next, we derive an upper bound of the approximation error $d(Q_k,Q_k^*)$ for any integer $k \geq 1$.

Let $q_k^*$ denote the density of $Q_k^*$ with respect to the Lebesgue measure. The Kullback-Leibler (KL) divergence between $Q_k$ and $Q_k^*$ is given by
\begin{align}
&KL(Q_k \lVert Q_k^*) = \E_{q_k(x_k,w)}\left[\log \frac{q_k(x_k,w)}{q_k^*(x_k,w)}\right] \nonumber \\
&\qquad = \E_{q_k(x_k,w)}\left[\log\left(\int_{\X\times\W}p_{\mathcal{N}}(y_k \mid \Phi_k(x_k,w),\Gamma)q_k^{*-}(x_k,w)dx_kdw\right)\right] \nonumber \\
&\qquad \qquad -\E_{q_k(x_k,w)}\left[\log \left(p_{\mathcal{N}}(y_k \mid \Phi_k(x_k,w),\Gamma)q_k^{*-}(x_k,w)\right)\right] + \E_{q_k(x_k,w)}[\log q_k(x_k,w)] \nonumber \\
&\qquad = \log\left(\int_{\X\times\W}p_{\mathcal{N}}(y_k \mid \Phi_k(x_k,w),\Gamma)q_k^{*-}(x_k,w)dx_kdw\right) - \mathcal{L}_k(Q_k).
\label{ieq:vi:type1:kl}
\end{align}
The first term on the right-hand side of Equation~\eqref{ieq:vi:type1:kl} satisfies
\begin{align}
&\log(\int_{\X\times\W}p_{\mathcal{N}}(y_k \mid \Phi_k(x_k,w),\Gamma)q_k^{*-}(x_k,w)dx_kdw \nonumber \\
&\qquad \leq \log\left((2\pi)^{-\frac{r}{2}}\det(\Gamma)^{-\frac{1}{2}}\int_{\X\times\W}q_k^{*-}(x_k,w)dx_kdw \right)\nonumber \\
&\qquad = -\frac{r}{2}\log(2\pi)-\frac{1}{2}\log(\det(\Gamma)).
\label{ieq:vi:type1:log_evidence}
\end{align}
Substituting Equation~\eqref{ieq:vi:type1:log_evidence} and $\mathcal{L}_k(Q_k) \geq \epsilon_k$ into Equation~\eqref{ieq:vi:type1:kl} yields:
\begin{equation*}
KL(Q_k \lVert Q_k^*) \leq -\frac{r}{2}\log(2\pi)-\frac{1}{2}\log(\det(\Gamma)) - \epsilon_k.
\end{equation*}
Therefore, we have
\begin{equation}
d_{TV}(Q_k^*,Q_k) \leq \frac{1}{\sqrt{2}}\sqrt{KL(Q_k \lVert Q_k^*)} \leq \frac{1}{\sqrt{2}}\sqrt{-\frac{r}{2}\log(2\pi)-\frac{1}{2}\log(\det(\Gamma)) - \epsilon_k},
\label{ieq:vi:type1:tv_approx}
\end{equation}
and 
\begin{equation}
d_{H}(Q_k^*,Q_k) \leq \frac{1}{\sqrt{2}}\sqrt{KL(Q_k \lVert Q_k^*)} \leq \frac{1}{\sqrt{2}}\sqrt{-\frac{r}{2}\log(2\pi)-\frac{1}{2}\log(\det(\Gamma)) - \epsilon_k}.
\label{ieq:vi:type1:H_approx}
\end{equation}
Combining Theorem~\ref{thm:learning_error} and Equations~\eqref{ieq:vi:type1:C},~\eqref{ieq:vi:type1:tv_approx}, and~\eqref{ieq:vi:type1:H_approx} completes the proof of Corollary~\ref{coro:example1} under the total variation and Hellinger distances.

According to~\cite{choose_metric}, if the metric space $(\X \times \W, d_{\X\times\W})$ satisfies
\begin{equation*}
\sup_{(x,w),(x^\prime,w^\prime) \in \X\times \W}d_{\X\times\W}((x,w),(x^\prime,w^\prime))=D < \infty,
\end{equation*}
then we have
\begin{equation}
W_1(P_k,Q_k) \leq Dd_{TV}(P_k,Q_k).
\label{ieq:vi:type1:W_lr}
\end{equation}
The statement of Corollary~\ref{coro:example1} for the $1$-Wasserstein distance is directly proved by combining the result of Corollary~\ref{coro:example1} for the total variation distance and Equation~\eqref{ieq:vi:type1:W_lr}.
\end{proof}
\subsection{Proof of Corollary~\ref{coro:example2}}
\label{sec:app:coro2}
\begin{proof}
For any time step $k\geq1$, we have
\begin{align*}
h_k(y_k,x) = p_{\mathcal{N}}(y_k \mid \Phi_k(x), \Gamma) \leq (2\pi)^{-\frac{r}{2}}\det (\Gamma)^{-\frac{1}{2}}.
\end{align*}
Thus the underlying system with the true system parameter $\bar{w}$ satisfies
\begin{align}
C_{Th}(y_k;k) &= \sup_{x_{k-1}\in \X}\int_{\X}h_k(y_k,x)T_k(x,x_{k-1},\bar{w})dx \nonumber \\
&\leq (2\pi)^{-\frac{r}{2}}\det (\Gamma)^{-\frac{1}{2}}\sup_{x_{k-1}\in \X}\int_{\X}T_k(x,x_{k-1},\bar{w})dx  \nonumber \\
&= (2\pi)^{-\frac{r}{2}}\det (\Gamma)^{-\frac{1}{2}} < \infty.
\label{ieq:vi:type2:C}
\end{align}
Therefore, the underlying system with true parameter $\bar{w}$ satisfies Assumption SE.2.
Next, we derive an upper bound for the approximation error $(Q_k^*,Q_k)$. Here, $Q_k^*$ is the exact posterior obtained by propagating the prior $Q_{k-1}$ under the system with the true parameter $\bar{w}$. Let $Q_{k,\hat{w}_k}^*$ be the exact posterior obtained using the same prior $Q_{k-1}$ but under system with the estimated parameter $\hat{w}_k$. By the triangle inequality satisfied by the metric $d$, we have
\begin{equation}
d(Q_k^*,Q_k) \leq d(Q_k^*,Q_{k,\hat{w}_k}^*)+d(Q_{k,\hat{w}_k}^*,Q_k).
\end{equation}
First, we derive an upper bound of $d(Q_k^*,Q_{k,\hat{w}_k}^*)$. Note that $Z_k(Q_{k-1})$ and $Z_{k,\hat{w}_k}(Q_{k-1})$ are corresponding evidences of $Q_k^*$ and $Q_{k,\hat{w}_k}^*$, respectively. Then we can write $d_{TV}(Q_k^*,Q_{k,\hat{w}_k}^*)$ as:
\begin{align}
d_{TV}(Q_k^*,Q_{k,\hat{w}_k}^*) =&\frac{1}{2}\int_{\X}\bigg \lvert \frac{h_k(y_k,x)}{Z_{k,\hat{w}_k}(Q_{k-1})}\int_{\X}T_k(x,x_{k-1},\hat{w}_k)q_{k-1}(x_{k-1})dx_{k-1} \nonumber \\
&\qquad  -\frac{h_k(y_k,x)}{Z_k(Q_{k-1})}\int_{\X}T_k(x,x_{k-1},\bar{w})q_{k-1}(x_{k-1})dx_{k-1}\bigg \rvert dx \nonumber \\
\leq & I_1 + I_2,
\label{ieq:vi:type2:tv:total}
\end{align}
where 
\begin{align*}
I_1 &= \frac{1}{2Z_{k,\hat{w}_k}(Q_{k-1})}\int_{\X}h_k(y_k,x)\bigg \lvert \int_{\X}T_k(x,x_{k-1},\hat{w}_k)q_{k-1}(x_{k-1})dx_{k-1} \nonumber \\
&\qquad  - \int_{\X}T_k(x,x_{k-1},\bar{w})q_{k-1}(x_{k-1})dx_{k-1} \bigg \rvert dx,
\end{align*}
and 
\begin{align*}
I_2 &= \frac{1}{2}\bigg \lvert \frac{1}{Z_{k,\hat{w}_k}(Q_{k-1})}-\frac{1}{Z_{k}(Q_{k-1})} \bigg \rvert \nonumber \\
&\qquad \cdot \int_{\X}h_k(y_k,x)\left(\int_{\X}T_k(x,x_{k-1},\bar{w})q_{k-1}(x_{k-1})dx_{k-1}\right) dx.
\end{align*}
First, we derive an upper bound for $I_1$. By Assumption VI.1, $I_1$ satisfies
\begin{align}
I_1 =& \frac{1}{2Z_{k,\hat{w}_k}(Q_{k-1})}\int_{\X}h_k(y_k,x) \nonumber \\
&\qquad \cdot \bigg \lvert \int_{\X}q_{k-1}(x_{k-1})\left(T_k(x,x_{k-1},\hat{w}_k)-T_k(x,x_{k-1},\bar{w})\right)dx_{k-1} \bigg \rvert dx \nonumber \\
\leq& \frac{1}{2Z_{k,\hat{w}_k}(Q_{k-1})}\int_{\X}h_k(y_k,x) \nonumber \\
&\qquad \cdot \left(\int_{\X}q_{k-1}(x_{k-1})\bigg \lvert T_k(x,x_{k-1},\hat{w}_k)-T_k(x,x_{k-1},\bar{w})\bigg \rvert dx_{k-1}\right) dx \nonumber \\
\leq& \frac{\lVert \hat{w}_k - \bar{w} \rVert_m }{2Z_{k,\hat{w}_k}(Q_{k-1})}\int_{\X}h_k(y_k,x)g_{\text{Lip}}(x;k)\left(\int_{\X}q_{k-1}(x_{k-1})dx_{k-1}\right) dx \nonumber \\
=& \frac{\lVert \hat{w}_k - \bar{w} \rVert_m }{2Z_{k,\hat{w}_k}(Q_{k-1})}\int_{\X}h_k(y_k,x)g_{\text{Lip}}(x;k)dx = \frac{\tilde{C}_{VI}(y_k;k)\lVert \hat{w}_k - \bar{w} \rVert_m }{2Z_{k,\hat{w}_k}(Q_{k-1})}.
\label{ieq:bound_I_1}
\end{align}
We then derive an upper bound for $I_2$. Note that
\begin{align*}
\int_{\X}h_k(y_k,x)\left(\int_{\X}T_k(x,x_{k-1},\bar{w})q_{k-1}(x_{k-1})dx_{k-1}\right)dx = Z_k(Q_{k-1}).
\end{align*}
Then $I_2$ satisfies
\begin{align}
I_2 &= \frac{\bigg \lvert Z_{k,\hat{w}_k}(Q_{k-1})-Z_{k}(Q_{k-1}) \bigg \rvert}{2Z_{k,\hat{w}_k}(Q_{k-1})} \nonumber \\
&= \frac{1}{2Z_{k,\hat{w}_k}(Q_{k-1})} \nonumber \\
&\qquad \cdot \bigg \lvert \int_{\X}h_k(y_k,x)\left(\int_{\X}\left(T_k(x,x_{k-1},\hat{w}_k)-T_k(x,x_{k-1},\bar{w})\right)q_{k-1}(x_{k-1})dx_{k-1}\right)dx \bigg \rvert \nonumber \\
&\leq  \frac{1}{2Z_{k,\hat{w}_k}(Q_{k-1})} \nonumber \\
&\qquad \cdot \int_{\X}h_k(y_k,x)\left(\int_{\X}\bigg \lvert T_k(x,x_{k-1},\hat{w}_k)-T_k(x,x_{k-1},\bar{w})\bigg \rvert q_{k-1}(x_{k-1})dx_{k-1}\right)dx \nonumber \\
& \leq \frac{\tilde{C}_{VI}(y_k;k)\lVert \hat{w}_k - \bar{w} \rVert_m }{2Z_{k,\hat{w}_k}(Q_{k-1})}.
\label{ieq:bound_I_2}
\end{align}
Substituting Equations~\eqref{ieq:bound_I_1} and~\eqref{ieq:bound_I_2} into Equation~\eqref{ieq:vi:type2:tv:total} yields:
\begin{align}
d_{TV}(Q_k^*,Q_{k,\hat{w}_k}^*) \leq \frac{\tilde{C}_{VI}(y_k;k)\lVert \hat{w}_k - \bar{w} \rVert_m }{Z_{k,\hat{w}_k}(Q_{k-1})}.
\label{ieq:vi:tyep2:tv_part1}
\end{align}

Subsequently, we derive an upper bound for $d_{H}(Q_k^*,Q_{k,\hat{w}_k}^*)$. Let $q_k^*$ and $q_{k,\hat{w}_k}^*$ denote the Lebesgue densities of $Q_k^*$ and $Q_{k,\hat{w}_k}^*$, respectively. Let $\tilde{Q}_k^*$ and $\tilde{Q}_{k,\hat{w}_k}^*$ be two measures with Lebesgue densities $\tilde{q}_k^*(x)= Z_k(Q_{k-1})q_k^*(x)$ and $\tilde{q}_{k,\hat{w}_k}^*= Z_{k,\hat{w}_k}(Q_{k-1})q_{k,\hat{w}_k}^*(x)$. Then the distance $\tilde{d}_H(\tilde{Q}_k^*,\tilde{Q}_{k,\hat{w}_k}^*)$ satisfies
\begin{align*}
\tilde{d}_H^2(\tilde{Q}_k^*,\tilde{Q}_{k,\hat{w}_k}^*) &= \frac{1}{2}\int_{\X}\left({\sqrt{h_k(y_k,x)\int_{\X}T_k(x,x_{k-1},\hat{w}_k)q_{k-1}(x_{k-1})dx_{k-1}}}\right. \nonumber \\
&\qquad \left.{-\sqrt{h_k(y_k,x)\int_{\X}T_k(x,x_{k-1},\bar{w})q_{k-1}(x_{k-1})dx_{k-1}}}\right)^2dx \nonumber \\
&=\frac{1}{2}\int_{\X}h_k(y_k,x)\left({\sqrt{\int_{\X}T_k(x,x_{k-1},\hat{w}_k)q_{k-1}(x_{k-1})dx_{k-1}}}\right. \nonumber \\
&\qquad \left.{-\sqrt{\int_{\X}T_k(x,x_{k-1},\bar{w})q_{k-1}(x_{k-1})dx_{k-1}}}\right)^2dx.
\end{align*}
As we have
\begin{align*}
&\left(\sqrt{\int_{\X}T_k(x,x_{k-1},\hat{w}_k)q_{k-1}(x_{k-1})dx_{k-1}}-\sqrt{\int_{\X}T_k(x,x_{k-1},\bar{w})q_{k-1}(x_{k-1})dx_{k-1}}\right)^2 \\
&\quad =\int_{\X}T_k(x,x_{k-1},\hat{w}_k)q_{k-1}(x_{k-1})dx_{k-1}+\int_{\X}T_k(x,x_{k-1},\bar{w})q_{k-1}(x_{k-1})dx_{k-1} \\
&\qquad \quad- 2\sqrt{\int_{\X}T_k(x,x_{k-1},\hat{w}_k)q_{k-1}(x_{k-1})dx_{k-1}}\sqrt{\int_{\X}T_k(x,x_{k-1},\bar{w})q_{k-1}(x_{k-1})dx_{k-1}} \\
& \quad \leq \int_{\X}T_k(x,x_{k-1},\hat{w}_k)q_{k-1}(x_{k-1})dx_{k-1}+\int_{\X}T_k(x,x_{k-1},\bar{w})q_{k-1}(x_{k-1})dx_{k-1} \\
&\qquad \quad- 2\int_{\X}\sqrt{T_k(x,x_{k-1},\hat{w}_k)T_k(x,x_{k-1},\bar{w})}q_{k-1}(x_{k-1})dx_{k-1}\\
& \quad= \int_{\X}q_{k-1}(x_{k-1})\left(\sqrt{T_k(x,x_{k-1},\hat{w}_k)}-\sqrt{T_k(x,x_{k-1},\bar{w})}\right)^2dx_{k-1},
\end{align*}
where the first inequality is obtained according to Hölder's inequality, $\tilde{d}_H(\tilde{Q}_k^*,\tilde{Q}_{k,\hat{w}_k}^*)$ satisfies
\begin{align*}
&\tilde{d}_H^2(\tilde{Q}_k^*,\tilde{Q}_{k,\hat{w}_k}^*) \\
&\; \leq \frac{1}{2}\int_{\X}h_k(y_k,x)\left(\int_{\X}q_{k-1}(x_{k-1})\left(\sqrt{T_k(x,x_{k-1},\hat{w}_k)}-\sqrt{T_k(x,x_{k-1},\bar{w})}\right)^2dx_{k-1}\right)dx.
\end{align*}
The term $\left(\sqrt{T_k(x,x_{k-1},\hat{w}_k)}-\sqrt{T_k(x,x_{k-1},\bar{w})}\right)^2$ satisfies
\begin{align*}
&\left(\sqrt{T_k(x,x_{k-1},\hat{w}_k)}-\sqrt{T_k(x,x_{k-1},\bar{w})}\right)^2 \\
&\quad \leq \bigg \lvert \sqrt{T_k(x,x_{k-1},\hat{w}_k)}-\sqrt{T_k(x,x_{k-1},\bar{w})} \bigg \rvert \cdot \bigg \lvert \sqrt{T_k(x,x_{k-1},\hat{w}_k)}+\sqrt{T_k(x,x_{k-1},\bar{w})} \bigg \rvert \\
&\qquad = \bigg \lvert T_k(x,x_{k-1},\hat{w}_k)-T_k(x,x_{k-1},\bar{w}) \bigg \rvert \\
&\qquad \leq g_{\text{Lip}}(x;k)\lVert \hat{w}_k - \bar{w} \rVert_m.
\end{align*}
It follows that
\begin{align*}
\tilde{d}_H^2(\tilde{Q}_k^*,\tilde{Q}_{k,\hat{w}_k}^*) 
&\leq \frac{1}{2}\lVert \hat{w}_k - \bar{w} \rVert_m \int_{\X}h_k(y_k,x)g_{\text{Lip}}(x;k)dx \nonumber \\
&= \frac{1}{2}\tilde{C}_{VI}(y_k;k)\lVert \hat{w}_k - \bar{w} \rVert_m.
\end{align*}
Therefore, by Lemma~\ref{lemma:unnormalized2normalized:H}, we have
\begin{align}
d_H(Q_k^*,Q_{k,\hat{w}_k}^*) &\leq \frac{2}{\sqrt{Z_{k,\hat{w}_k}(Q_{k-1})}}\tilde{d}_H(\tilde{Q}_k^*,\tilde{Q}_{k,\hat{w}_k}^*) \nonumber \\
&= \sqrt{\frac{2\tilde{C}_{VI}(y_k;k)\lVert \hat{w}_k - \bar{w} \rVert_m}{Z_{k,\hat{w}_k}(Q_{k-1})}}.
\label{ieq:vi:tyep2:H_part1}
\end{align}

We next derive an upper bound of $d(Q_{k,\hat{w}_k}^*,Q_k)$. The KL divergence between $Q_k$ and $Q_{k,\hat{w}_k}^*$ satisfies:
\begin{align}
&KL(Q_k \lVert Q_{k,\hat{w}_k}^*) = \E_{q_k(x_k)}\left[\log \frac{q_k(x_k)}{q_{k,\hat{w}_k}^*(x_k)}\right] \nonumber \\
&\qquad = \E_{q_k(x_k)}\left[\log\left(\int_{\X}p_{\mathcal{N}}(y_k \mid \Phi_k(x_k),\Gamma)q_{k,\hat{w}_k}^{*-}(x_k)dx_k\right)\right] \nonumber \\
&\qquad \qquad -\E_{q_k(x_k)}\left[\log \left(p_{\mathcal{N}}(y_k \mid \Phi_k(x_k),\Gamma)q_{k,\hat{w}_k}^{*-}(x_k)\right)\right] + \E_{q_k(x_k)}[\log q_k(x_k)] \nonumber \\
&\qquad = \log\left(\int_{\X}p_{\mathcal{N}}(y_k \mid \Phi_k(x_k),\Gamma)q_{k,\hat{w}_k}^{*-}(x_k)dx_k\right) - \mathcal{L}_k(Q_k,\hat{w}_k).
\label{ieq:vi:type2:kl}
\end{align}
The first term on the right-hand side of Equation~\eqref{ieq:vi:type2:kl} satisfies
\begin{align*}
\log\left(\int_{\X}p_{\mathcal{N}}(y_k \mid \Phi_k(x_k),\Gamma)q_{k,\hat{w}_k}^{*-}(x_k)dx_k\right) \leq -\frac{r}{2}\log(2\pi)-\frac{1}{2}\log(\det(\Gamma)).
\end{align*}
As we also have $\mathcal{L}_k(Q_k,\hat{w}_k) \leq \epsilon_k$, the KL divergence $KL(Q_k \lVert Q_{k,\hat{w}_k}^*)$ satisfies
\begin{equation*}
KL(Q_k \lVert Q_{k,\hat{w}_k}^*) \leq -\frac{r}{2}\log(2\pi)-\frac{1}{2}\log(\det(\Gamma)) - \epsilon_k.
\end{equation*}
It follows that
\begin{align}
d_{TV}(Q_k ,Q_{k,\hat{w}_k}^*) \leq \frac{1}{\sqrt{2}}\sqrt{KL(Q_k \lVert Q_{k,\hat{w}_k}^*)} \leq \frac{1}{\sqrt{2}}\sqrt{-\frac{r}{2}\log(2\pi)-\frac{1}{2}\log(\det(\Gamma)) - \epsilon_k},
\label{ieq:vi:tyep2:tv_part2}
\end{align}
and 
\begin{align}
d_{H}(Q_k ,Q_{k,\hat{w}_k}^*) \leq \frac{1}{\sqrt{2}}\sqrt{KL(Q_k \lVert Q_{k,\hat{w}_k}^*)} \leq \frac{1}{\sqrt{2}}\sqrt{-\frac{r}{2}\log(2\pi)-\frac{1}{2}\log(\det(\Gamma)) - \epsilon_k}.
\label{ieq:vi:tyep2:H_part2}
\end{align}
Combining Equation~\eqref{ieq:vi:tyep2:tv_part1} with  Equation~\eqref{ieq:vi:tyep2:tv_part2} gives
\begin{equation}
d_{TV}(Q_k^*,Q_k) \leq \frac{\tilde{C}_{VI}(y_k;k)\lVert \hat{w}_k - \bar{w} \rVert_m }{Z_{k,\hat{w}_k}(Q_{k-1})}+\frac{1}{\sqrt{2}}\sqrt{-\frac{r}{2}\log(2\pi)-\frac{1}{2}\log(\det(\Gamma)) - \epsilon_k}.
\label{ieq:vi:tyep2:tv_final}
\end{equation}
Similarly, the bound of $d_H(Q_k^*,Q_k)$ can be obtained by combining Equation~\eqref{ieq:vi:tyep2:H_part1} and Equation~\eqref{ieq:vi:tyep2:H_part2}:
\begin{align}
d_{H}(Q_k^*,Q_k) &\leq \sqrt{\frac{2\tilde{C}_{VI}(y_k;k)\lVert \hat{w}_k - \bar{w} \rVert_m}{Z_{k,\hat{w}_k}(Q_{k-1})}} +\frac{1}{\sqrt{2}}\sqrt{-\frac{r}{2}\log(2\pi)-\frac{1}{2}\log(\det(\Gamma)) - \epsilon_k}.
\label{ieq:vi:tyep2:H_final}
\end{align}
Combining Theorem~\ref{thm:learning_error} with Equations~\eqref{ieq:vi:type2:C},~\eqref{ieq:vi:tyep2:tv_final}, and~\eqref{ieq:vi:tyep2:H_final} completes the proof the Corollary~\ref{coro:example2} for the total variation and Hellinger distances.

According to~\cite{choose_metric}, if the metric space $(\X,d_{\X})$ satisfies
\begin{equation*}
\sup_{x,x^\prime \in \X}d_{\X}(x,x^\prime) = D < \infty,
\end{equation*}
then we have
\begin{equation}
W_1(P_k,Q_k) \leq Dd_{TV}(P_k,Q_k).
\label{ieq:vi:tyep2:W_final}
\end{equation}
The conclusion of Corollary~\ref{coro:example2} for the $1$-Wasserstein distance is proved by combining the results of Corollary~\ref{coro:example2} for the total variation distance and inequality~\eqref{ieq:vi:tyep2:W_final}.

\end{proof}

\vskip 0.2in
\bibliography{ref}

\end{document}